\documentclass[12pt]{amsart}
\usepackage{amsfonts,amssymb,amscd,amsmath,latexsym,amsbsy,tikz-cd}

\usepackage[all,cmtip]{xy}
\usepackage{color}
\usepackage{verbatim}
\usepackage{version}
\DeclareSymbolFont{symbolsC}{U}{pxsyc}{m}{n}
\SetSymbolFont{symbolsC}{bold}{U}{pxsyc}{bx}{n}
\DeclareFontSubstitution{U}{pxsyc}{m}{n}
\DeclareMathSymbol{\medcirc}{\mathbin}{symbolsC}{7}

\numberwithin{equation}{section}
\newtheorem{theorem}{Theorem}[section]
\newtheorem{lemma}[theorem]{Lemma}
\newtheorem{proposition}[theorem]{Proposition}
\newtheorem{corollary}[theorem]{Corollary}
\theoremstyle{definition}
\newtheorem{definition}[theorem]{Definition}

\newtheorem{example}[theorem]{Example}

\newtheorem{conjecture}[theorem]{Conjecture}
\newtheorem{remark}[theorem]{Remark}

\newcommand{\Tr}{\text{Tr}}
\newcommand{\id}{\text{id}}

\newcommand{\Ker}{\text{Ker\,}}

\newcommand{\Hom}{\text{Hom}}
\newcommand{\Ad}{\text{Ad}}

\newcommand{\X}{{\bold X}}
\newcommand{\D}{{\bold D}}
\newcommand{\Y}{{\bold Y}}

\newcommand{\DAHA}{{H}\mkern-6mu{H}}
\newcommand{\bDAHA}{{\Bbb H}\mkern-6mu{\Bbb H}}

\newcommand{\ben}{\begin{enumerate}}
\newcommand{\een}{\end{enumerate}}

\newcommand{\CC}{{\mathbb{C}}}
\newcommand{\ZZ}{{\mathbb{Z}}}

\theoremstyle{plain}

\newtheorem*{sol}{Solution}

\theoremstyle{definition}

\theoremstyle{remark}

\newcommand{\solu}[1]{\begin{sol}{\bf (\ref{#1})}}

\def\Ad{\mathop{\mathrm{Ad}}\nolimits}

\def\Hom{\mathrm{Hom}}
\def\Ker{\mathrm{Ker}}

\def\GL{\mathop{GL}}

\newcommand\nc{\newcommand}
\nc{\unl}{\underline}
\nc{\ol}{\overline}
\nc{\on}{\operatorname}
\nc{\iso}{{\stackrel{\sim}{\longrightarrow}}}

\nc{\BA}{{\mathbb{A}}}
\nc{\BC}{{\mathbb{C}}}
\nc{\BD}{{\mathbb{D}}}
\nc{\BF}{{\mathbb{F}}}
\nc{\BG}{{\mathbb{G}}}
\nc{\BM}{{\mathbb{M}}}
\nc{\BN}{{\mathbb{N}}}
\nc{\BO}{{\mathbb{O}}}
\nc{\BQ}{{\mathbb{Q}}}
\nc{\BP}{{\mathbb{P}}}
\nc{\BR}{{\mathbb{R}}}
\nc{\BZ}{{\mathbb{Z}}}
\nc{\BS}{{\mathbb{S}}}

\nc{\CA}{{\mathcal{A}}}
\nc{\CB}{{\mathcal{B}}}
\nc{\CalC}{{\mathcal C}}
\nc{\CalD}{{\mathcal D}}
\nc{\CE}{{\mathcal{E}}}
\nc{\CF}{{\mathcal{F}}}
\nc{\CG}{{\mathcal{G}}}
\nc{\CH}{{\mathcal{H}}}
\nc{\CI}{{\mathcal{I}}}
\nc{\CK}{{\mathcal{K}}}
\nc{\CL}{{\mathcal{L}}}
\nc{\CM}{{\mathcal{M}}}
\nc{\CN}{{\mathcal{N}}}
\nc{\CO}{{\mathcal{O}}}
\nc{\CP}{{\mathcal{P}}}
\nc{\CQ}{{\mathcal{Q}}}
\nc{\CR}{{\mathcal{R}}}
\nc{\CS}{{\mathcal{S}}}
\nc{\CT}{{\mathcal{T}}}
\nc{\CU}{{\mathcal{U}}}
\nc{\CV}{{\mathcal{V}}}
\nc{\CW}{{\mathcal{W}}}
\nc{\CX}{{\mathcal{X}}}
\nc{\CY}{{\mathcal{Y}}}
\nc{\CZ}{{\mathcal{Z}}}

\nc{\fa}{{\mathfrak{a}}}
\nc{\fb}{{\mathfrak{b}}}
\nc{\fg}{{\mathfrak{g}}}
\nc{\fgl}{{\mathfrak{gl}}}
\nc{\fh}{{\mathfrak{h}}}
\nc{\fj}{{\mathfrak{j}}}
\nc{\fl}{{\mathfrak{l}}}
\nc{\fm}{{\mathfrak{m}}}
\nc{\fn}{{\mathfrak{n}}}
\nc{\fu}{{\mathfrak{u}}}
\nc{\fp}{{\mathfrak{p}}}
\nc{\frr}{{\mathfrak{r}}}
\nc{\fs}{{\mathfrak{s}}}
\nc{\ft}{{\mathfrak{t}}}
\nc{\fw}{{\mathfrak{w}}}
\nc{\fz}{{\mathfrak{z}}}

\nc{\fA}{{\mathfrak{A}}}
\nc{\fB}{{\mathfrak{B}}}
\nc{\fD}{{\mathfrak{D}}}
\nc{\fE}{{\mathfrak{E}}}
\nc{\fF}{{\mathfrak{F}}}
\nc{\fG}{{\mathfrak{G}}}
\nc{\fI}{{\mathfrak{I}}}
\nc{\fJ}{{\mathfrak{J}}}
\nc{\fK}{{\mathfrak{K}}}
\nc{\fL}{{\mathfrak{L}}}
\nc{\fM}{{\mathfrak{M}}}
\nc{\fN}{{\mathfrak{N}}}
\nc{\frP}{{\mathfrak{P}}}
\nc{\fQ}{{\mathfrak Q}}
\nc{\fR}{{\mathfrak R}}
\nc{\fS}{{\mathfrak S}}
\nc{\fT}{{\mathfrak{T}}}
\nc{\fU}{{\mathfrak{U}}}
\nc{\fW}{{\mathfrak{W}}}
\nc{\fZ}{{\mathfrak{Z}}}

\nc{\ba}{{\mathbf{a}}}
\nc{\bb}{{\mathbf{b}}}
\nc{\bc}{{\mathbf{c}}}
\nc{\bd}{{\mathbf{d}}}
\nc{\bi}{{\mathbf{i}}}
\nc{\bj}{{\mathbf{j}}}
\nc{\bn}{{\mathbf{n}}}
\nc{\bp}{{\mathbf{p}}}
\nc{\bq}{{\mathbf{q}}}
\nc{\bu}{{\mathbf{u}}}
\nc{\bv}{{\mathbf{v}}}
\nc{\bw}{{\mathbf{w}}}
\nc{\bx}{{\mathbf{x}}}
\nc{\by}{{\mathbf{y}}}
\nc{\bz}{{\mathbf{z}}}

\nc{\bA}{{\mathbf{A}}}
\nc{\bB}{{\mathbf{B}}}
\nc{\bC}{{\mathbf{C}}}
\nc{\bD}{{\mathbf{D}}}
\nc{\bE}{{\mathbf{E}}}
\nc{\bK}{{\mathbf{K}}}
\nc{\bH}{{\mathbf{H}}}
\nc{\bM}{{\mathbf{M}}}
\nc{\bN}{{\mathbf{N}}}
\nc{\bO}{{\mathbf{O}}}
\nc{\bQ}{{\mathbf Q}}
\nc{\bS}{{\mathbf{S}}}
\nc{\bT}{{\mathbf{T}}}
\nc{\bW}{{\mathbf{W}}}
\nc{\bX}{{\mathbf{X}}}
\nc{\bP}{{\mathbf{P}}}
\nc{\bZ}{{\mathbf{Z}}}

\nc{\sA}{{\mathsf{A}}}
\nc{\sB}{{\mathsf{B}}}
\nc{\sC}{{\mathsf{C}}}
\nc{\sD}{{\mathsf{D}}}
\nc{\sF}{{\mathsf{F}}}
\nc{\sK}{{\mathsf{K}}}
\nc{\sM}{{\mathsf{M}}}
\nc{\sO}{{\mathsf{O}}}
\nc{\sQ}{{\mathsf{Q}}}
\nc{\sP}{{\mathsf{P}}}
\nc{\sV}{{\mathsf{V}}}
\nc{\sW}{{\mathsf{W}}}
\nc{\sZ}{{\mathsf{Z}}}

\nc{\sfb}{{\mathsf{b}}}
\nc{\sfc}{{\mathsf{c}}}
\nc{\sd}{{\mathsf{d}}}
\nc{\sg}{{\mathsf{g}}}
\nc{\sk}{{\mathsf{k}}}
\nc{\sfl}{{\mathsf{l}}}
\nc{\sfp}{{\mathsf{p}}}
\nc{\sr}{{\mathsf{r}}}
\nc{\st}{{\mathsf{t}}}
\nc{\sfu}{{\mathsf{u}}}
\nc{\sz}{{\mathsf{z}}}

\nc{\Fl}{{{\mathcal F}\ell}}
\nc{\CHH}{{\CH\!\!\CH}}
\begin{document}

\title{Cyclotomic double affine Hecke algebras}

\author{Alexander Braverman}
\address{Department of Mathematics, University of Toronto,
and Perimeter Institute of Theoretical Physics,
Waterloo, Ontario, Canada, N2L 2Y5;
Skolkovo Institute of Science and Technology}
\email{braval@math.toronto.edu}

\author{Pavel Etingof}
\address{Department of Mathematics, Massachusetts Institute of Technology,
Cambridge, MA 02139, USA}
\email{etingof@math.mit.edu}

\author{Michael Finkelberg}
\address{National Research University Higher School of Economics,
Russian Federation,
Department of Mathematics, 6 Usacheva st., Moscow 119048;
Skolkovo Institute of Science and Technology;
Institute for Information Transmission Problems of RAS}

\email{fnklberg@gmail.com}

\maketitle

\centerline{\sf With an appendix by Hiraku Nakajima and Daisuke Yamakawa}
\vskip .10in

\vskip .05in
\centerline{\bf To Ivan Cherednik with admiration}
\vskip .05in

\begin{abstract} We show that the partially spherical cyclotomic rational Cherednik algebra (obtained from the full rational Cherednik algebra by averaging out the cyclotomic part of the underlying reflection group) has four other descriptions: (1) as a subalgebra of the degenerate DAHA of type A given by generators; (2) as an algebra given by generators and relations; (3) as an algebra of differential-reflection operators preserving some spaces of functions; 
(4) as equivariant Borel-Moore homology of a certain variety. Also, we define a new $q$-deformation of this algebra, which we call {\it cyclotomic DAHA}. Namely, we give a $q$-deformation of each of the above four descriptions of the partially spherical  rational Cherednik algebra, replacing differential operators with difference operators, degenerate DAHA with DAHA, and homology with K-theory, and show that they give the same algebra. In addition, we show that spherical cyclotomic DAHA 
are quantizations of certain multiplicative quiver and bow varieties, which may be interpreted as K-theoretic Coulomb branches of a framed quiver gauge theory. Finally, we apply cyclotomic DAHA to prove new flatness results for various kinds of spaces of $q$-deformed quasiinvariants.  
\end{abstract} 

\section{Introduction}
\label{sec1}

Let $N\ge 0,l\ge 0$ be integers, $c_0,...,c_{l-1},\hbar,k$ be parameters, and $c=(c_0,...,c_{l-1})$.
Let $\bDAHA_N^{l,{\rm cyc}}(c,\hbar,k)$ be the cyclotomic rational Cherednik algebra attached to the complex reflection
group $W=S_N\ltimes (\Bbb Z/l\Bbb Z)^N$. Let $\bold p$ be the symmetrizer of the subgroup
$(\Bbb Z/l\Bbb Z)^N$, and $\bDAHA_N^{l,{\rm psc}}(c,\hbar,k):=\bold p \bDAHA_N^{l,{\rm cyc}}(c,\hbar,k)\bold p$ be the corresponding
partially spherical subalgebra.

In this paper we give a geometric interpretation of $\bDAHA_N^{l,{\rm psc}}$ as the equivariant Borel-Moore homology of a certain variety ${\mathcal R}={\mathcal R}(N,l)$ equipped with a group action. This allows us to define a natural $q$-deformation $\DAHA_N^l$ of $\bDAHA_N^{l,{\rm psc}}$ in terms of the equivariant K-theory of ${\mathcal R}(N,l)$, which we call the {\it cyclotomic double affine Hecke algebra} (DAHA). 

The existence of this $q$-deformation may seem somewhat surprising from the viewpoint of classical algebraic theory of DAHA (\cite{Ch1}), since typically DAHA are attached to crystallographic reflection groups (Weyl groups), while the group $W$ is not crystallographic for $l\ge 3$. Yet, we also give a purely algebraic definition of cyclotomic DAHA. Namely, we characterize the cyclotomic DAHA as the subalgebra of the usual Cherednik's DAHA for $GL_N$ generated by certain elements, and also as the subalgebra preserving certain spaces of functions. Finally, we present cyclotomic DAHA by generators and relations. These three descriptions also make sense in the trigonometric limit $q\to 1$ (for partially spherical cyclotomic rational Cherednik algebras). We note that for $l=1$, cyclotomic DAHA essentially appeared in \cite{BF}.

We also connect cyclotomic DAHA with multiplicative quiver and bow varieties. Namely, we show that the spherical cyclotomic DAHA $\bold e\DAHA_N^l(Z,1,t)\bold e$ (where $\bold e$ is the symmetrizer of the finite Hecke algebra) is commutative, and its spectrum for generic parameters is isomorphic to the algebra of regular functions on the multiplicative quiver variety for the cyclic quiver of length $l$ with dimension vector $(N,...,N)$; hence $\bold e\DAHA_N^l(Z,q,t)\bold e$ is a quantization of this variety. In particular, we show that this multiplicative quiver variety is connected, and that $\DAHA_N^l(Z,1,t)$ is an Azumaya algebra of degree $N!$ over this variety. We also show that if $t$ is not a root of unity then the algebra $\bold e\DAHA_N^l(Z,1,t)\bold e$ is an integrally closed Cohen-Macaulay domain isomorphic to the center ${\mathcal Z}(\DAHA_N^l(Z,1,t))$, while $\DAHA_N^l(Z,1,t)\bold e$ is a Cohen-Macaulay module over this algebra. 

Finally, we provide some applications of cyclotomic DAHA to the theory of quasiinvariants. Namely, we show
that natural $q$-deformations of various classes of spaces of quasi-invariants are flat, and therefore free modules over the algebra of 
symmetric polynomials. We also introduce a new type of quasiinvariants (namely, twisted quasiinvariants) and their $q$-deformation, and prove 
the freeness property for them. 

We note that the degenerate cyclotomic DAHA were studied in a way similar to ours by 
R.~Kodera and H.~Nakajima in \cite{KN}. In fact, their paper was one of the starting points for our work.

The paper is organized as follows. In Section~\ref{sec2} 
we develop the theory of partially spherical cyclotomic
rational Cherednik algebras as subalgebras in trigonometric (degenerate) DAHA, and give their presentation.
In Section~\ref{sec3} we define cyclotomic DAHA as subalgebras of DAHA, 
and study their properties. We also give a presentation of cyclotomic DAHA, which allows us to find various 
bases in them and prove flatness results. In Section~\ref{sec5} we give a geometric description of cyclotomic DAHA and 
their degenerate versions in terms
of equivariant K-theory and Borel-Moore homology, and apply it to proving 
flatness of these algebras. In Section~\ref{sec4} we relate the spherical subalgebra of cyclotomic DAHA
at $q=1$ with certain multiplicative quiver and bow varieties; the latter
are isomorphic to the $K$-theoretic Coulomb branch of framed quiver gauge
theories of affine type $A$. We also study the properties of the spectrum of the spherical cyclotomic DAHA for $q=1$. 
In Section~\ref{sec6} we give applications of cyclotomic DAHA to 
proving flatness of $q$-deformation of various spaces of quasi-invariants.
Finally, Appendix~\ref{apend}, written by H.~Nakajima and D.~Yamakawa, explains
the relations between multiplicative bow varieties and (various
versions of) multiplicative quiver varieties for a cyclic quiver.

{\bf Acknowledgements.} The work of P.E.\ was partially supported by the NSF
grant DMS-1502244.
The work of M.F. has been funded within the framework of the HSE University Basic Research Program
and the Russian Academic Excellence Project `5-100'.
We are grateful to J.~Stokman for useful discussions and reference~\cite{BF}; to O. Chalykh for his comments cited in Remarks \ref{cha1}, \ref{cha2}; to 
J. F. van Diejen and S. Ruijsenaars for Remark \ref{ruij}; to E. Rains for explaining the connection with \cite{GKV} and \cite{R1,R2} (Remark \ref{rai}) and pointing out that Lemma \ref{rank1lem} and Theorem \ref{charac}(i) need the assumption that $k\notin \Bbb Z+1/2$; 
to B.~Webster for sharing~\cite{W} with us prior to its publication and explaining its results (see Remark \ref{KNW});
to J.~Kamnitzer for explaining the KLR-type construction of the convolution algebras
of Section~\ref{triples} to us; and to H.~Nakajima for explaining to us the results 
of \cite{KN} (see Remark \ref{KNW}). Also, Sections~\ref{bow} and~\ref{Coul} are due to H.~Nakajima's
patient explanations.

\section{Degenerate cyclotomic DAHA}
\label{sec2}

\subsection{Notation}
In this paper, we will consider many different algebras depending on parameters.
So let us clarify our conventions.

First of all, if an algebra depends on parameters, we will list the
parameters explicitly when they are given numerical values, and
omit them when they are indeterminates (i.e., we work over a commutative base algebra
generated by them). Also, throughout the paper, we will use the following notation, to be defined below.

$\bullet$ $\DAHA_{N,{\rm deg}}(\hbar, k)$: the degenerate (trigonometric) DAHA, Definition \ref{trigdaha};

$\bullet$ $\DAHA_{N,{\rm deg}}^l(z,\hbar,k)$, $z:=(z_1,...,z_l)$: the degenerate cyclotomic DAHA, 
Definition \ref{dcycldaha};

$\bullet$ $\bDAHA_N^{l,{\rm cyc}}(c,\hbar,k)$, $c:=(c_0,...,c_{l-1})$: the cyclotomic rational Cherednik algebra for the group
$S_n\ltimes (\Bbb Z/l\Bbb Z)^n$, Definition \ref{cyclcher};

$\bullet$ $\bDAHA_N^{l,{\rm psc}}(c,\hbar,k):=\bold p \bDAHA_N^{l,{\rm cyc}}(c,\hbar,k)\bold p$: the partially spherical cyclotomic rational Cherednik algebra, Subsection \ref{compa};

$\bullet$ $\DAHA_N(q,t)$: Cherednik's DAHA, Definition \ref{DAHAdef};

$\bullet$ $\DAHA_N^{\rm formal}(\hbar,k)$: formal Cherednik's DAHA over $\Bbb C[[\varepsilon]]$, with $q=e^{\varepsilon \hbar}$ and $t=e^{-\varepsilon k}$, Subsection \ref{dfd};

$\bullet$ $\DAHA_N^l(Z,q,t)$, $Z:=(Z_1,...,Z_l)$: the cyclotomic DAHA, Definition \ref{cyclotomDAHA};

$\bullet$ $\DAHA_N^{l,{\rm formal}}(z,\hbar,k)$, the formal cyclotomic DAHA, with $q=e^{\varepsilon \hbar}$, $Z_i=e^{\varepsilon z_i}$ and $t=e^{-\varepsilon k}$, Definition \ref{cyclotomDAHA};

$\bullet$ $\DAHA_N^{\rm rat}(\hbar,k)$, the rational Cherednik algebra for $S_N$, Example \ref{ratdaha};

$\bullet$ $\CHH^l_{N,{\rm deg}}$: the geometric version of the degenerate cyclotomic DAHA, $H^{\BC^\times\times T(W)_\CO\times\CP\rtimes\BC^\times}_\bullet(\CR)$, Section~\ref{sec5};

$\bullet$ $\CHH^l_N$: the geometric version of the cyclotomic DAHA, 
$K^{\BC^\times\times T(W)_\CO\times\CP\rtimes\BC^\times}(\CR)$, Section~\ref{sec5}.

\subsection{Degenerate (trigonometric) DAHA}

Let $\hbar,k$  be variables.

\begin{definition}\label{trigdaha} The degenerate (or trigonometric) double affine Hecke algebra (DAHA) $\DAHA_{N,{\rm deg}}$ is
generated over $\Bbb C[\hbar,k]$ by $\pi^{\pm 1}, s_0,...,s_{N-1}, y_1,...,y_N$
with defining relations
$$
s_i^2=1,\ s_is_{i+1}s_i=s_{i+1}s_is_{i+1},\ s_is_j=s_js_i\text{ if }i-j\ne \pm1,\
\pi s_i=s_{i+1}\pi,
$$
$$
[y_i,y_j]=0,\ \pi y_i=y_{i+1}\pi\text{ if }i\ne N,\ \pi y_N=(y_1-\hbar)\pi,
$$
$$
s_iy_i=y_{i+1}s_i+k,\ \text{ if }i\ne 0, s_0y_N=(y_1-\hbar)s_0+k,
$$
$$
[s_i,y_j]=0\text{ if }i-j\ne \pm 1,
$$
where addition is mod $N$.
\end{definition}

\begin{proposition}\label{trigDAHA} The algebra $\DAHA_{N,{\rm deg}}$ is generated by $S_N\ltimes \Bbb Z^N$
(generated by $s_i$ and invertible commuting elements $X_1,...,X_N$) and elements $y_1,...,y_N$
with commutation relations
\begin{gather*}
s_iy_i=y_{i+1}s_i+k,\ s_iy_j=y_js_i, j\ne i,i+1,\\
[y_i,y_j]=0,\\
[y_i,X_j]=kX_js_{ij},\ i>j,\\
[y_i,X_j]=kX_is_{ij},\ i<j,\\
[y_i,X_i]=\hbar X_i-k\sum_{r<i}X_rs_{ir}-k\sum_{r>i}X_is_{ir},
\end{gather*}
and the relations of $S_N\ltimes \Bbb Z^N$, where $s_{ij}$ is the transposition of $i$ and $j$.
Namely, the transition between the two definitions is given by the formulas
$$
\pi=X_1s_1...s_{N-1},\ s_0=X_N^{-1}X_1s_{1N}.
$$
\end{proposition}

\begin{proof} The proposition is standard, and the proof is by a direct computation; see e.g. \cite{Su}, Section 2.
\end{proof}

\begin{remark}\label{trigDAHA1} The commutation relations between $y_i$ and $X_j$
in Proposition \ref{trigDAHA} can be replaced by the relations
\begin{equation}\label{firrel}
[y_i,\prod_j X_j]=\hbar\prod_j X_j\text{ or }
[\sum_i y_i,X_j]=\hbar X_j,
\end{equation}
and
\begin{equation}\label{secrel}
[y_2,X_1]=kX_1s_1.
\end{equation}
Indeed, the relation for $[y_i,X_j]$, $i\ne j$ can be obtained from \eqref{secrel} by the action of $S_N$, and then the relation
for $[y_i,X_i]$ can be obtained by using one of the relations \eqref{firrel}.
\end{remark}

Note that $\DAHA_{N,{\rm deg}}$ is a bigraded algebra:
$\deg(X_i)=(1,0)$, $\deg(y_i)=\deg(\hbar)=\deg(k)=(0,1)$, $\deg(s_i)=(0,0)$.
Moreover, the PBW theorem for the degenerate DAHA, which follows from the existence of its
polynomial representation (see Subsection \ref{pr}) implies that $\DAHA_{N,{\rm deg}}$
is a free bigraded module over $\Bbb C[\hbar,k]$.

Also, we have the specialization $\DAHA_{N,{\rm deg}}(\hbar,k)$ of the degenerate DAHA at $\hbar,k\in \Bbb C$, which is defined by the same generators and relations,
where $\hbar, k$ are numerical. The bigrading on $\DAHA_{N,{\rm deg}}$ induces a grading on this specialization
defined by $\deg(X_i)=1$, $\deg(y_i)=\deg(s_j)=0$ and an increasing filtration $F^\bullet$ compatible with this grading,
defined by $\deg(X_i)=\deg(s_j)=0$, $\deg(y_i)=1$.

\subsection{The polynomial representation}\label{pr}
Let $D_i$ be the {\it rational Dunkl operators}
$$
D_i:=\hbar\partial_i-\sum_{j\ne i}\frac{k}{X_i-X_j}(1-s_{ij}),
$$
where $\partial_i$ is the derivative with respect to $X_i$, and
$s_{ij}$ is the permutation of $i$ and $j$ (so $s_i=s_{i,i+1}$).
Define the {\it trigonometric Dunkl operators} by the formula
$$
D_i^{\rm trig}:=X_iD_i-k\sum_{j<i}s_{ij}.
$$

The following well known proposition is due to Cherednik (\cite{Ch2}; see also \cite{Su}, Proposition 3.1).

\begin{proposition}\label{chere} We have a representation $\rho$
of $\DAHA_{N,{\rm deg}}(\hbar,k)$ on $\bold P:=\Bbb C[X_1^{\pm 1},...,X_N^{\pm 1}]$,
defined by
$$
\rho(s_i)=s_i\text{ for }i\ne 0,\ \rho(\pi)=X_1s_1...s_{N-1},
$$
$$
\rho(s_0)=s_{1N}X_1^{-1}X_N,
$$
$$
\rho(y_i)=D_i^{\rm trig}.
$$
\end{proposition}

\begin{proof} The relations involving only $\pi$ and $s_i$ are easy.
Let us prove that $[\rho(y_i),\rho(y_m)]=0$ for $i<m$. Using that $[D_i,X_m]=ks_{im}$ and $[D_i,D_m]=0$, we get
$$
[\rho(y_i),\rho(y_m)]=
$$
$$
kX_is_{im}D_m-kX_ms_{im}D_i-k[X_iD_i,s_{im}]+k^2\sum_{i<j}[s_{ij},s_{im}+s_{jm}].
$$
The first three summands cancel, and the last summand is zero, as desired.

Let us prove the commutation relations between $\rho(\pi)$ and $\rho(y_i)$.
For $i<N$ we have
$$
\rho(\pi) \rho(y_i)=X_1s_1...s_{N-1}(X_iD_i-k\sum_{j<i}s_{ij})=
$$
$$
X_1(X_{i+1}D_{i+1}-k\sum_{j<i}s_{j+1,i+1})s_1...s_{N-1}=
$$
$$
(X_{i+1}D_{i+1}-k\sum_{j<i}s_{j+1,i+1})X_1s_1...s_{N-1}-kX_{i+1}s_{1,i+1}s_1...s_{N-1}=\rho(y_{i+1})\rho(\pi).
$$
Also $\rho(\pi^n)=X_1...X_n$, so $\rho(\pi^n) \rho(y_i)=(\rho(y_i)-\hbar)\rho(\pi^n)$,
which implies that the relation $\pi y_N=(y_1-\hbar)\pi$ is preserved.

Let us show that the relation $s_iy_i=y_{i+1}s_i+k$ for $i\ne 0$ is preserved. We have
$$
\rho(s_i)\rho(y_i)=s_{i,i+1}(X_iD_i-k\sum_{j<i}s_{ij})=
$$
$$
(X_{i+1}D_{i+1}-k\sum_{j<i}s_{i+1,j})s_{i,i+1}=\rho(y_{i+1})\rho(s_i)+k.
$$
The relation $s_0y_N-(y_1-\hbar)s_0+k$ is obtained from the previous relation by conjugation by $\pi$.
Finally, the relation $[s_i,y_j]=0$ for $i-j\ne \pm 1$ is easy for $i\ne 0$, and for $i=0$ is obtained from the case $i=1$ by conjugation
by $\pi$. The proposition is proved.
\end{proof}

\begin{definition}
The representation $\bold P$ of $\DAHA_{N,{\rm deg}}(\hbar,k)$ is called the {\it polynomial representation}.
\end{definition}

It is easy to see that the polynomial representation is faithful when $\hbar\ne 0$.
Moreover, replacing $\hbar \partial_i$ with momentum variables $p_i$, we can make it faithful in the limit $\hbar=0$ (see \cite{EM}, 2.10).

Let ${\mathcal D}_N$ be the algebra of differential operators
in $X_1,...,X_N$ with poles at $X_i=0$ and $X_i=X_j$. Then the algebra $\Bbb CS_N\ltimes \mathcal{D}_N$ acts naturally on $\bold P[\Delta^{-1}]$,
where
$$
\Delta:=\prod_{i<j}(X_i-X_j),
$$
and
$\rho$ defines an inclusion
$$
\DAHA_{N,{\rm deg}}(1,k)\hookrightarrow \Bbb CS_N\ltimes \mathcal{D}_N.
$$
We will use this inclusion to view $\DAHA_{N,{\rm deg}}(1,k)$ as a subalgebra of $\Bbb CS_N\ltimes \mathcal{D}_N$.
Note that the filtration $F^\bullet$ on $\DAHA_{N,{\rm deg}}(1,k)$ is induced under this inclusion by the order filtration on differential operators.

Let $\psi_\kappa$ be the automorphism of the algebra $\Bbb C S_N\ltimes {\mathcal D}_N$ fixing
$X_i$ and sending $s_i$ to $-s_i$ and $\partial_i$ to
$\partial_i+\sum_{j\ne i}\frac{\kappa}{X_i-X_j}$, i.e., conjugation
by $|\Delta|^\kappa {\rm sign}(\Delta)$ on the real locus.

\begin{lemma}\label{prese0}
 The algebra $\psi_{2k}(\DAHA_{N,{\rm deg}}(1,k))\subset \Bbb CS_N\ltimes \mathcal{D}_N$ preserves $\bold P$. In other words,
 $\DAHA_{N,{\rm deg}}(1,k)$ preserves $|\Delta|^{2k} {\rm sign}(\Delta)\bold P$.
\end{lemma}

\begin{proof} We need to check that
$$
\psi_{2k}(D_1)=D_1-\sum_{j\ne 1}\frac{2k}{X_1-X_j}=
$$
$$
\partial_1-\sum_{j\ne 1}\frac{k}{X_1-X_j}(1+s_{ij})+\sum_{j\ne 1}\frac{2k}{X_1-X_j}=
\partial_1+\sum_{j\ne 1}\frac{k}{X_1-X_j}(1-s_{ij})
$$
preserves $\bold P$, which is straightforward.
\end{proof}

Let $\bold e=\frac{1}{N!}\sum_{s\in S_N}s$ be the symmetrizer of $S_N$. The algebra 
$\bold e\DAHA_{N,{\rm deg}}(\hbar,k)\bold e$ is called the {\it spherical subalgebra}
of $\DAHA_{N,{\rm deg}}(\hbar,k)$; it has the polynomial representation $\bold e\bold P=\Bbb C[X_1^{\pm 1},...,X_N^{\pm 1}]^{S_N}$. Let
$H:=\bold e (\sum_i y_i^2)\bold e\in \bold e\DAHA_{N,{\rm deg}}(\hbar,k)\bold e$. This element acts on the polynomial representation by the trigonometric Calogero-Moser Hamiltonian.

\begin{lemma}\label{generat} The algebra $A:=\bold e\DAHA_{N,{\rm deg}}(1,k)\bold e$ is generated by $H$ and  $\Bbb C[X_1^{\pm 1},...,X_N^{\pm 1}]^{S_N}$.
\end{lemma}

\begin{proof} The algebra $A$ has a filtration $F^\bullet$ given by $\deg(X_i)=\deg(s_j)=0$, $\deg(y_i)=1$, and ${\rm gr}(A)=\Bbb C[X_1^{\pm 1},...,X_N^{\pm 1},p_1,...,p_N]^{S_N}$
with Poisson bracket corresponding to the symplectic form $\sum_i dp_i\wedge dX_i/X_i$. Since the symbol of $H$ is $H_0:=\sum_i p_i^2$,
it suffices to check that ${\rm gr}(A)$ is Poisson generated by $\Bbb C[X_1^{\pm 1},...,X_N^{\pm 1}]^{S_N}$ and $H_0$. For this it suffices to show that the Poisson algebra $B$ generated by these elements contains
$F_{r,s}:=\sum_i X_i^rp_i^s$, where $r$ is any integer and $s$ a nonnegative integer, as such elements generate ${\rm gr}(A)$ as a commutative algebra (by a theorem of H.~Weyl). The Poisson bracket is given by $\lbrace{p_j,X_i\rbrace}=\delta_{ij}X_i$, and $X_i$ are pairwise Poisson commutative, as are $p_j$. So for $r\ne 0$
$$
F_{r,s}=\frac{1}{2r}\lbrace{\sum_i p_i^2,F_{r,s-1}\rbrace},
$$
and $F_{r,0}=\sum_i X_i^r\in  \Bbb C[X_1^{\pm 1},...,X_N^{\pm 1}]^{S_N}\subset B$. So it remains to show that $F_{0,s}\in B$. For this, it suffices to note that
$$
F_{0,s}=\frac{1}{s+1}\lbrace{F_{-1,s+1},F_{1,0}\rbrace}.
$$
\end{proof}

\subsection{Degenerate cyclotomic DAHA}

\begin{definition}\label{dcycldaha} Let $l\in \Bbb Z_{\ge 0}$, $z_1,...,z_l\in \Bbb C$, and $z=(z_1,...,z_l)$. The degenerate cyclotomic DAHA is
the subalgebra $\DAHA_{N,{\rm deg}}^l(z,\hbar,k)$ of $\DAHA_{N,{\rm deg}}(\hbar,k)$ generated by $s_i$, $i=1,...,N-1$,
$y_i$, $i=1,...,N$, $\pi$, and the element
$$
\pi_-:=\pi^{-1}\prod_{i=1}^l (y_1-z_i).
$$
Similarly, if $z_i$, $\hbar$, and $k$ are variables, we define $\DAHA_{N,{\rm deg}}^l$ to be the subalgebra of $\DAHA_{N,{\rm deg}}[z_1,...,z_l]$ generated by the same elements.
\end{definition}

Note that by this definition
$$
\DAHA_{N,{\rm deg}}^0(\hbar,k)=\DAHA_{N,{\rm deg}}(\hbar,k),\ \DAHA_{N,{\rm deg}}^{l'}(z',\hbar,k)\subset \DAHA_{N,{\rm deg}}^l(z,\hbar,k)
$$
if $l'\ge l$ and $z'\supset z$ as a multiset.

For $u\in \Bbb C$, let $\phi_u$ be the automorphism
of $\Bbb C S_N\ltimes {\mathcal D}_N$ which preserves $X_i,s_i$, and sends
$\partial_i$ to $\partial_i+uX_i^{-1}$
(i.e., conjugation by $(\prod_i X_i)^u$).

\begin{proposition}\label{prese} The algebras $\phi_{z_i}(\DAHA_{N,{\rm deg}}^l(z,1,k))$, $i\in [1,l]$, preserve the subspace $\bold P_+:=\Bbb C[X_1,...,X_N]\subset \bold P$.
In other words, the algebra $\DAHA_{N,{\rm deg}}^l(z,1,k)$ preserves $(\prod_j X_j)^{z_i}\bold P_+$ for all $i$.
\end{proposition}

\begin{proof} We only need to show that $\pi^{-1}\phi_{z_i}((y_1-z_1)...(y_1-z_l))$
preserves $\bold P_+$. Note that $y_1=X_1D_1$, so $\phi_{z_i}(y_1)=y_1+z_i$.
Thus, we need to check that $\pi^{-1}y_1\prod_{j\ne i}(y_1-z_j+z_i)$ preserves ${\bold P}_+$. But this holds since $\pi^{-1}y_1=s_{N-1}...s_1D_1$
preserves ${\bold P}_+$ (and, of course, so does $y_1$).
\end{proof}

A similar result holds when $z_i,k$ are variables.

\begin{theorem}\label{charac}
(i) (\cite{GKV}) If $k\notin \Bbb Z+1/2$ then $\DAHA_{N,{\rm deg}}(1,k)$ is the algebra of all elements of the algebra $S_N\ltimes {\mathcal D}_N$
which preserve $\bold P$ and $|\Delta|^{2k} {\rm sign}(\Delta)\bold P$.

(ii) Suppose $z_i-z_j$ are not integers
and $k\in \Bbb C$ is Weil generic\footnote{We don't know if Theorem \ref{charac} (ii),(iii) actually fails for any value of $k\notin \Bbb Z+1/2$.} (i.e., outside a countable set). Then the algebra $\DAHA_{N,{\rm deg}}^l(z,1,k)$ is the subalgebra of all elements of
$\DAHA_{N,{\rm deg}}(1,k)$ which preserve $(\prod_j X_j)^{z_i}\bold P_+$ for all $i$.

(iii) Under the assumption of (ii),
the algebra $\DAHA_{N,{\rm deg}}^l(z,1,k)$ is the subalgebra of all elements of
$S_N\ltimes {\mathcal D}_N$
which preserve $\bold P$, $|\Delta|^\kappa {\rm sign}(\Delta)\bold P$, and $(\prod_j X_j)^{z_i}\bold P_+$ for all $i$.
\end{theorem}

Theorem \ref{charac} is proved in the next two subsections.

\subsection{The case $N=1$}

Let us first prove Theorem \ref{charac} for $N=1$.
Let $B_l(z):=\DAHA_1^l(1,z)$
(the parameter $k$ does not enter in this case).
The algebra $B_l(z)$ is generated by $X=X_1$, the Euler element $E:=X\partial=y_1$, and
$$
L:=\pi_-=X^{-1}(X\partial -z_1)...(X\partial -z_l).
$$
Then Theorem \ref{charac} for $N=1$ reduces to the following statement.

\begin{proposition}\label{1var} If $z_i-z_j$ are not integers for $i\ne j$ then
$B_l(z)$ is the algebra of all differential operators on $\Bbb C^*$
which preserve $X^{z_i}\Bbb C[X]$ for all $i=1,...,l$.
\end{proposition}

\begin{proof}
It is clear that every element of $B_l(z)$ preserves $X^{z_i}\Bbb C[X]$,
since so do $X$, $E$ and $L$. So we need to show that any operator $M$ which preserves $X^{z_i}\Bbb C[X]$ belongs to
$B_l(z)$.

Suppose that $M$ is of degree $r$, i.e., $M=X^r g(E)$, where $g$ is some polynomial. If $r\ge 0$, then $M\in B_l(z)$, as $E\in B_l(z)$. So it remains to consider the case $r=-q$, where $q>0$. Then $M$ has to annihilate $X^{z_i},...,X^{z_i+q-1}$ for all $i$ (which are all different thanks to the condition on the $z_i$), so the degree of $g$ is at least $ql$.
On the other hand, if we take $g(y)=P(y)P(y-1)...P(y-q+1)h(y)$ where $P(y):=(y-z_1)...(y-z_l)$, then $M:=X^{-q}g(E)=L^qh(E)$, so $M\in B_l(z)$. This shows that $M$ always belongs to $B_l(z)$
(as we can subtract $L^qh(E)$ to make the degree of the polynomial $<ql$).
The proposition is proved.
\end{proof}

\begin{remark} We have $E=XL+z_1$ for $l=1$ and $E=\frac{1}{2}([L,X]+z_1+z_2-1)$ for $l=2$, so one may ask if $B_l(z)$ is in fact generated by
$L,X$ (i.e., if $E$ can be expressed via $L,X$). It is not hard to show that this is indeed the case for generic $z_i$. But if $l\ge 3$ and $z_i$ are special,
then the algebra $B_l(z)$ may not be generated by $X,L$. Indeed, let $z_1=0$, $z_2=1$, $z_3=2$, and the other $z_i$ be arbitrary.
Then $B_l(z)$ preserves $\Bbb C[X]$ and $X^2\Bbb C[X]$, so has a 2-dimensional representation
$V=\Bbb C[X]/(X^2)$ where $L$ acts by $0$ and $X$ acts nilpotently. Hence, if $B_l'(z)\subset B_l(z)$ is the subalgebra
generated by $L,X$ then every element of $B_l'(z)$ has only one eigenvalue on $V$. On the other hand, $E$ has eigenvalues $0,1$ on $V$, so $E\notin B_l'(z)$.
\end{remark}

We will also need the following ``unsymmetrized'' version of
Proposition \ref{1var} for $l=2$. Let $A(k)$ be the rational Cherednik algebra
with parameter $k$ attached to the group $\Bbb Z/2\Bbb Z$, i.e.,
generated by $x$, $s\in \Bbb Z/2\Bbb Z$ such that $sx=-xs$, and $D=\partial-\frac{k}{x}(1-s)$.

\begin{lemma}\label{rank1lem} If $k\notin \Bbb Z+\frac{1}{2}$ then $A(k)$ is the algebra of all elements of $\Bbb C\Bbb Z/2\Bbb Z\ltimes {\mathcal{D}}(\Bbb C^*)$ which preserve $\Bbb C[x]$ and $|x|^{2k}{\rm sign}(x)\Bbb C[x]$.\footnote{Lemma \ref{rank1lem} fails if $k\in \Bbb Z+1/2$. For example, if $k=1/2$ 
then the element $(x^{-1}\partial-x^{-2})(1-s)$ preserves the required spaces but is not in $A(k)$.}
\end{lemma}

\begin{proof}
It is clear that $A(k)$ preserves these spaces, since so do its generators.
So it remains to show that any element $M$ preserving these spaces is in $A(k)$.
We may assume that $M$ is homogeneous. Note that $E:=\frac{1}{2}(xD+Dx-1+2k)=x\partial\in A(k)$. So if $M$ is of nonnegative degree, then $M=x^r (g_1(E)(1-s)+g_2(E)(1+s))$, where $g_1,g_2$ are some polynomials, hence $M\in A(k)$ automatically. Now suppose $\deg(M)=-q<0$,
i.e., $M=x^{-q} (g_1(E)(1-s)+g_2(E)(1+s))$. The operator $M$
has to annihilate $x^m$ and ${\rm sign}(x)|x|^{2k}x^m$
for $0\le m\le q-1$, so
$g_1(u)$ is divisible by $\prod_{j=0}^{q-1}(u-j-(1+(-1)^j)k)$, while
$g_2(u)$ is divisible by $\prod_{j=0}^{q-1}(u-j-(1-(-1)^j)k)$ 
(here we use that these products have no repeated factors since $k\notin \Bbb Z+1/2$).
On the other hand, it is easy to see by acting on monomials $x^p$ that
$$
x^{-q}\prod_{j=0}^{q-1}(E-j-(1+(-1)^j)k)\cdot (1-s)=D^q(1-s),
$$
while
$$
x^{-q}\prod_{j=0}^{q-1}(E-j-(1-(-1)^j)k)\cdot (1+s)=D^q(1+s),
$$
which are both in $A(k)$. This implies the lemma.
\end{proof}

\subsection{Proof of Theorem \ref{charac}}\label{pfcharac}

Let us prove (i). By Lemma \ref{rank1lem} and Lemma \ref{prese0}, $\DAHA_{N,{\rm deg}}(1,k)$ is the space of elements of $S_N\ltimes {\mathcal{D}}_N$
which upon formal completion at a generic point of each hyperplane $X_i=X_j$ (in the sense of \cite{BE})
lie in the formal completion of $A(k)\otimes {\Bbb W}_{N-1}$, where ${\Bbb W}_{N-1}$ is the Weyl algebra of $N-1$ variables.
So (i) follows from the results of \cite{E1} (see also the appendix to \cite{BE}).

It is clear that given (i), statements (ii) and (iii) are equivalent, so let us prove (ii) (assuming $l>0$, as the case $l=0$ is trivial).
First consider the case $k=0$. In this case, the result follows from the following lemma.

\begin{lemma}\label{kis0} Let $L\in S_N\ltimes {\mathcal D}(\Bbb C^*)^{\otimes N}$ be an element preserving the space $(\prod_j X_j)^{z_i}\bold P_+$ for all $i$.
Then $L\in S_N\ltimes B_l(z)^{\otimes N}$.
\end{lemma}

\begin{proof} Let $L=\sum_{\sigma \in S_N}\sigma L_\sigma$, where $L_\sigma$ are differential operators.
Consider a generic point $\bold x$ in the hyperplane $X_1=0$, and let $E_i$ be the formal completion of the $\Bbb C[X_1,...,X_N]$-module
$(\prod_j X_j)^{z_i}\bold P_+$ near the $S_N$-orbit of $\bold x$. Then $E_i=\oplus_{\sigma\in S_N}E_{i,\sigma}$, where $E_{i,\sigma}$ is the completion
of $(\prod_j X_j)^{z_i}\bold P_+$ at the point $\sigma \bold x$. It is clear that $L$ preserves $E_i$ for all $i$. This implies that for each $\sigma$, $L_{\sigma}$ preserves
$E_{i,1}$. Hence $L_\sigma$ preserves $(\prod_j X_j)^{z_i}\bold P_+$ for all $i$ and $\sigma$. Thus we may assume without loss of generality that
$L\in {\mathcal D}(\Bbb C^*)^{\otimes N}$ is a differential operator.

It is clear from taking completions that $L$ preserves the space 
$X_1^{z_i}\Bbb C[X_1,X_2^{\pm 1},...,X_N^{\pm 1}]$
for all $i$. Therefore, by Proposition \ref{1var}, for any $v\in \Bbb C[X_2^{\pm 1},...,X_N^{\pm 1}]$,
$\psi\in \Bbb C[X_2^{\pm 1},...,X_N^{\pm 1}]^*$ the differential operator
$({\rm Id}\otimes \psi)(L({\rm Id}\otimes v))\in {\mathcal D}(\Bbb C^*)$ in fact belongs to $B_l(z)$.
Let $\{a_i\}$ be a basis of $B_l(z)$, and $\{a_j'\}$ 
its extension to a basis of ${\mathcal D}(\Bbb C^*)$.
We can uniquely write $L$ as
$$
L=\sum_i a_i\otimes L_i+\sum_j a_j'\otimes L_j'
$$
Thus we have
$$
\sum_i \psi(L_iv)a_i+\sum_j \psi(L_j'v)a_j'\in B_l(z).
$$
Hence $\psi(L_j'v)=0$ for all $j$. Since this holds for all $\psi$, we have $L_j'v=0$.
Since this holds for all $v$, we have $L_j'=0$. Thus $L=\sum_i a_i\otimes L_i\in B_l(z)\otimes {\mathcal D}(\Bbb C^*)^{\otimes N-1}$, i.e., the first component of $L$
lies in $B_l(z)$. A similar argument applies to all the other components. Thus, $L\in B_l(z)^{\otimes N}$, as desired.
\end{proof}

Now, consider the case $k\ne 0$. By Proposition \ref{prese}, elements of $\DAHA_{N,{\rm deg}}^l(1,k)$
preserve the spaces $(\prod_j X_j)^{z_i}\bold P_+$. Thus we only have to show that
$\DAHA_{N,{\rm deg}}^l(1,k)$ is ``big enough'', i.e., coincides with the algebra $A_N(z,k)$ of all the elements preserving the spaces $(\prod_j X_j)^{z_i}\bold P_+$.
But this follows for Weil generic $k$
from the case $k=0$ by a standard deformation argument (the algebra can only ``get bigger'' if we deform
its generators). More precisely, recall that we have a grading on the algebra $S_N\ltimes \mathcal{D}_N$ defined by $\deg(X_i)=1$, $\deg(s_i)=0$,
$\deg(\partial_i)=-1$ inherited by $\DAHA_{N,{\rm deg}}(1,k)$, $A_N(z,k)$, and $\DAHA_{N,{\rm deg}}^l(z,1,k)$, and the filtration $F^\bullet$ by order of differential operators, and
it is not hard to see that for each $s$, $F^sA_N(z,k)$ is a finitely generated graded $\Bbb C[X_1,...,X_N]$-module.
Thus, for each $r,s$, the space $F^sA_N(z,k)[r]$ is finite dimensional. Hence, for each $r,s$ the set
of $k\in \Bbb C$ for which $F^s\DAHA_{N,{\rm deg}}^l(z,1,k)[r]\ne F^sA_N(z,k)[r]$ is finite.

\subsection{Comparison to the cyclotomic rational Cherednik algebra for $N=1$}

Let us now consider the cyclotomic rational Cherednik algebra
$\bDAHA_1^{l,{\rm cyc}}$ of rank $1$ with parameters $\hbar$ and $c=(c_0,...,c_{l-1})$ (see \cite{EG2,EM}.
By definition, this algebra is generated over $\Bbb C[\hbar,c_0,...,c_{l-1}]$ by $x$ and the cyclotomic Dunkl operator
$$
D_{\rm cyc}:=\hbar\partial-x^{-1}\sum_{i=0}^{l-1}c_i\sigma^i,
$$
where 
$$
\sigma(x)=\zeta x, \quad \zeta=e^{2\pi i/l}.
$$ 
This algebra is bigraded by
$\deg(x)=(1,0)$, $\deg(D_{\rm cyc})=(-1,1)$, $\deg(\sigma)=(0,0)$, $\deg(\hbar)=\deg(c)=(0,1)$, and by the PBW theorem (see \cite{DO, EG1}) is a free
bigraded module over $\Bbb C[c,\hbar]$. We also have the algebra $\bDAHA_1^{l,{\rm cyc}}(c,\hbar)$ with numerical parameters,
i.e., the specialization of $\bDAHA_1^{l,{\rm cyc}}$, and it carries a grading and a compatible filtration $F^\bullet$.

We have $\bDAHA_1^{l,{\rm cyc}}(c,1)\subset \Bbb C\Bbb Z/l\Bbb Z\ltimes {\mathcal{D}}$,
where ${\mathcal{D}}={\mathcal{D}}_1$ is the algebra of
differential operators on $\Bbb C^*$, and the filtration $F^\bullet$ on $\bDAHA_1^{l,{\rm cyc}}(c,1)$ is induced by the order filtration on differential operators.

Let ${\bold p}$ be the symmetrizer of $\Bbb Z/l\Bbb Z$.
Then we have a spherical subalgebra
${\Bbb B}_l(c):={\bold p}\bDAHA_1^{l,{\rm cyc}}(c,1){\bold p}$.
This algebra acts naturally on $\Bbb C[x^{\pm 1}]$, where $X=x^l$.

\begin{proposition}\label{compar}If $z_i-z_j$ are not integers then
$B_l(z)={\Bbb B}_l(c)$ inside ${\rm End}\Bbb C[x^{\pm 1}]$,
where $c_i$ are related to $z_j$ by the linear inhomogeneous
change of variables
\begin{equation}\label{ziform}
z_i=\frac{1}{l}(l-i+\sum_j c_j\zeta^{ij}).
\end{equation}
\end{proposition}

\begin{proof} Fix $c$ and let us find values of $u$ for which $X^u\Bbb C[X^{\pm 1}]$
is preserved by ${\Bbb B}_l(c)$.
The condition is that there exists $1\le i\le l$ such that
$D_{\rm cyc}(x^{i-l} (x^l)^u)=0$, where we treat $(x^l)^u$ as a
$\Bbb Z/l\Bbb Z$-invariant. This gives the equation
$$
i-l+lu-\sum_j c_j\zeta^{ij}=0,
$$
i.e.
$$
u=z_i:=\frac{1}{l}(l-i+\sum_j c_j\zeta^{ij}).
$$
This yields the desired change of variable.

Now, ${\Bbb B}_l(c)$ preserves the subspaces
$X^{z_i}\Bbb C[X^{\pm 1}]$, so by Proposition \ref{1var} we have an inclusion
${\Bbb B}_l(c)\subset B_l(z)$.
To show that this is actually an equality, it suffices to observe that
the element $L$ of $B_l(z)$ is proportional to
$D_{\rm cyc}^l{\bold p}$.
\end{proof}

\subsection{Comparison to cyclotomic Cherednik algebra for general $N$}\label{compa}

Let us now extend the result of the previous subsection to general $N$.

\begin{definition} \label{cyclcher}
The cyclotomic rational Cherednik algebra for the group $S_N\ltimes (\Bbb Z/l\Bbb Z)^N$,
$\bDAHA_N^{l,{\rm cyc}}$, is the algebra generated over $\Bbb C[c_0,...,c_{l-1},\hbar,k]$ by the group $S_N\ltimes (\Bbb Z/l\Bbb Z)^N$, elements $x_i$, and
the cyclotomic Dunkl operators, also called Dunkl-Opdam operators, \cite[Definition 3.2]{DO}:
$$
D_{i,\rm cyc}=\hbar \partial_i-\frac{1}{x_i}\sum_{j=0}^{l-1}c_j\sigma_i^j-k\sum_{r\ne i,m}\frac{1}{x_i-\zeta^m x_r}(1-s_{ir}\sigma_i^m\sigma_r^{-m}),
$$
for $i=1,...,N$, where $\sigma_i$ is $\sigma$ acting in the $i$-th component.
\end{definition}

As in the rank 1 case, this algebra is bigraded by
$\deg(x_i)=(1,0)$, $\deg(D_{i,\rm cyc})=(-1,1)$, $\deg(\sigma_i)=(0,0)$, $\deg(\hbar)=\deg(c)=\deg(k)=(0,1)$, and by the PBW theorem (see \cite{DO}) is a free
bigraded module over $\Bbb C[c,\hbar,k]$. We also have the algebra $\bDAHA_N^{l,{\rm cyc}}(c,\hbar,k)$ with numerical parameters,
i.e., the specialization of $\bDAHA_N^{l,{\rm cyc}}$, and it carries a grading and a compatible filtration $F^\bullet$.

Let ${\bold p}$ be the symmetrizer of
the subgroup $(\Bbb Z/l\Bbb Z)^N$, and
$\bDAHA_N^{l,{\rm psc}}(c,\hbar,k)={\bold p}\bDAHA_N^{l,{\rm cyc}}(c,\hbar,k){\bold p}$ be
the corresponding partially spherical subalgebra.

\begin{example}\label{ratdaha} Let $l=1$. Then $\bDAHA_N^{l,{\rm psc}}(c,\hbar,k)$ does not depend on $c$ (up to a natural isomorphism), and
is the rational Cherednik algebra $\DAHA_N^{\rm rat}(\hbar,k)$, generated by $X_i,D_i,$ and $s\in S_N$ such that
$X_i,D_i$ are permuted by $S_N$ and satisfy the relations
$$
[X_i,X_j]=[D_i,D_j]=0,
$$
$$
[D_i,X_j]=ks_{ij},\ [D_i,X_i]=\hbar-k\sum_{j\ne i}s_{ij}.
$$
Also in this case $\bold p=1$.
\end{example}

\begin{theorem}\label{main} Suppose $z_i-z_j$ are not integers and $k$ is Weil generic. Then
we have a natural isomorphism $\bDAHA_N^{l,{\rm psc}}(c,1,k)\cong \DAHA_{N,{\rm deg}}^l(z,1,k)$, where $z_i$ are expressed via $c_j$ by formula \eqref{ziform}.
This isomorphism preserves the order filtration for differential operators.
\end{theorem}

\begin{proof} We have a natural faithful action of $\bDAHA_N^{l,{\rm psc}}(c,1,k)$ on $\bold P$.
Moreover, it is easy to see that $\bDAHA_N^{l,{\rm psc}}(c,1,k)$
satisfies the conclusion of Theorem \ref{charac}(iii) (this follows by taking formal completions
at generic points of reflection hyperplanes, as in \cite{E1}, and using Proposition \ref{compar}). Therefore, by Theorem \ref{charac},
$\bDAHA_N^{l,{\rm psc}}(c,1,k)=\DAHA_{N,{\rm deg}}^l(z,1,k)$.
\end{proof}

One of our main results is the following theorem.

\begin{theorem}\label{main1}
(i) We have a natural isomorphism of bigraded algebras $\theta: \bDAHA_N^{l,{\rm psc}}\cong \DAHA_{N,{\rm deg}}^l$, where $z_i$ are expressed via $c_j$ and $\hbar$ by the homogenization of
formula \eqref{ziform}:
\begin{equation}\label{ziform1}
z_i=\frac{1}{l}(\hbar(l-i)+\sum_j c_j\zeta^{ij}).
\end{equation}

(ii) $\DAHA_{N,{\rm deg}}^l$ is a free bigraded $\Bbb C[z_1,...,z_l,\hbar,k]$-module,
and $\DAHA_{N,{\rm deg}}^l(z,\hbar,k)$ are specializations of this algebra.

(iii) The isomorphism $\theta$ induces an isomorphism
$$
\theta_{z,\hbar,k}:\bDAHA_N^{l,{\rm psc}}(c,\hbar,k)\cong \DAHA_{N,{\rm deg}}^l(z,\hbar,k)
$$
for all $z_i,\hbar,k\in \Bbb C$, which is compatible with the grading and the filtration.
\end{theorem}

\begin{proof} Let ${\rm gr}(\DAHA_{N,{\rm deg}}^l(z,1,k))=\DAHA_{N,{\rm deg}}^l(0,0,0)$ be the associated graded algebra of
$\DAHA_{N,{\rm deg}}^l(z,1,k)$ with respect to the filtration by order of differential operators.
Then ${\rm gr}(\DAHA_{N,{\rm deg}}^l(z,1,k))$ contains $S_N\ltimes A^{\otimes N}$, where
$A$ is the algebra of functions on the $A_{l-1}$-singularity, generated by $X, XP$, and $X^{-1}(XP)^l$
(where $P$ is the symbol of $\partial$). On the other hand, ${\rm gr}(\bDAHA_N^{l,{\rm psc}}(c,1,k))$
clearly coincides with $S_N\ltimes A^{\otimes N}$. Since by Theorem \ref{main},
$\bDAHA_N^{l,{\rm psc}}(c,1,k)=\DAHA_{N,{\rm deg}}^l(z,1,k)$ for Weil generic parameters,
we conclude that ${\rm gr}(\DAHA_{N,{\rm deg}}^l(z,\hbar,k))=S_N\ltimes A^{\otimes N}$ for all $z_i,\hbar,k$. This implies Theorem \ref{main1}.
\end{proof}

\begin{remark} 1. Since the degenerate DAHA has a $\Bbb G_a$-action given by $y_i\mapsto y_i+a$ and trivial on other generators, the algebra $\DAHA_{N,{\rm deg}}^l(z,1,k)$
does not change under the transformation $z_i\to z_i+a$ (i.e., it depends only on the differences $z_i-z_{i+1}$).
Under the isomorphism of Theorem \ref{main1}, this symmetry transforms into the symmetry $c_0\mapsto c_0+la$.

2. Another proof of Theorem \ref{main1}(ii) is given in the next subsection. 
\end{remark}

\begin{remark}\label{KNW} The isomorphism of spherical subalgebras of $\bDAHA_N^{l,{\rm psc}}$ and $\DAHA_{N,{\rm deg}}^l$
is due to Kodera and Nakajima \cite{KN}. More precisely, the parameters of the cyclotomic rational Cherednik algebra
in \cite{KN} are related to ours by the formula $\hbar^{KN}=-\hbar$, $c_m^{KN}=c_m(1-\zeta^m)$ for $m\ne 0$ (and $c_0^{KN}$ is not used in \cite{KN}).
Also in \cite{KN}, one has $\sum_{m=0}^{l-1} c_m=0$ and thus $z_l=0$, which is not really a restriction due to the symmetry
$z_i\mapsto z_i+a$, $c_0\mapsto c_0+la$.

Also, this isomorphism is closely related to the new presentation of the full 
cyclotomic DAHA $\bDAHA_N^{l,{\rm cyc}}(c,\hbar,k)$ given in~\cite[3.3]{G} 
and~\cite[Section~2]{W}. Namely, the element $e'x_1^{\ell-1}\sigma e'$ 
in~\cite[(4.1)]{W} (which in terms of the pictures is wrapping around the cylinder $\ell-1$ times, and a little further so as to cross the seam $\ell$ times) matches up with our element $\pi_-$, and similarly $e'y_n^{\ell-1}\tau e'$ matches with our element $\pi$ (with $\ell$ in \cite{W} equal to our $l$).
\end{remark}

\subsection{A presentation of $\DAHA_{N ,\rm deg}^l$ by generators and relations} 

Let us give a presentation of $\DAHA_{N ,\rm deg}^l$ by generators and relations. As generators we will use 
the elements of $S_N$, $y_1,...,y_N$, $X_1,...,X_N$, and the elements 
$$
D_i^{(l)}:=s_{1i}D_1^{(l)}s_{1i},\text{ where }D_1^{(l)}:=X_1^{-1}(y_1-z_1)...(y_1-z_l). 
$$
Obviously, these elements generate $\DAHA_{N ,\rm deg}^l$, so we only need to write down the relations. 

First of all, the elements $s\in S_N$, $y_i$ and $X_i$ satisfy the relations of Proposition \ref{trigDAHA}, except that $X_i$ are no longer invertible. 

We also claim that 
\begin{equation}\label{dunklcomm}
[D_i^{(l)},D_j^{(l)}]=0.
\end{equation}

Indeed, it suffices to check it for $i=1,j=2$. Then we have 
$$
D_1^{(l)}D_2^{(l)}=X_1^{-1}(y_1-z_1)...(y_1-z_l)s_{12}X_1^{-1}(y_1-z_1)...(y_1-z_l)s_{12}=
$$
$$
X_1^{-1}s_{12}(y_2+ks_{12}-z_1)....(y_2+ks_{12}-z_l)X_1^{-1}(y_1-z_1)...(y_1-z_l)s_{12}=
$$
$$
X_1^{-1}X_2^{-1}s_{12}(y_2-z_1)...(y_2-z_l)(y_1-z_1)...(y_1-z_l)s_{12}.
$$
This expression commutes with $s_{12}$, which implies the statement. 

We also see by direct computation that 
$$
[y_1,D_1^{(l)}]=-\hbar D_1^{(l)}+k\sum_{i>1}s_{1i}D_1^{(l)}
$$
and 
$$
[y_j,D_1^{(l)}]=-ks_{1j}D_1^{(l)},\ j>1.
$$

Finally, we write down the commutation relations between $D_i^{(l)}$ and $X_j$. 

\begin{lemma}\label{commuu}
We have 
$$
[D_1^{(l)},X_1]=\sum_{r=1}^{l}\prod_{i=1}^{r-1}(y_1-z_i+\hbar-k\sum_{j>1}s_{1j})(\hbar-k\sum_{j>1}s_{1j})\prod_{i=r+1}^l (y_1-z_i),
$$
and for $m>1$
$$
[D_1^{(l)},X_m]=k\sum_{r=1}^{l}\prod_{i=1}^{r-1}(y_1-z_i+\hbar-k\sum_{j>1}s_{1j})s_{1m}\prod_{i=r+1}^l (y_1-z_i). 
$$
\end{lemma}

\begin{proof}
The first relation holds because 
$$
[D_1^{(l)},X_1]=P(y_1+\hbar-k\sum_{j>1}s_{1j})-P(y_1),
$$
where $P(y)=(y-z_1)...(y-z_l)$. 
To prove the second relation, note that 
$$
[D_1^{(l)},X_m]=X_1^{-1}(P(y_1)-P(y_1-ks_{1m}))X_m=
$$
$$
k\sum_{r=1}^l X_1^{-1}\prod_{i=1}^{r-1}(y_1-z_i)s_{1m}\prod_{i=r+1}^l(y_1-z_i-ks_{1m})X_m=
$$  
$$
k\sum_{r=1}^l X_1^{-1}\prod_{i=1}^{r-1}(y_1-z_i)X_1s_{1m}\prod_{i=r+1}^l(y_1-z_i)=
$$
$$
k\sum_{r=1}^l \prod_{i=1}^{r-1}(y_1-z_i+\hbar-k\sum_{j>1}s_{1j})s_{1m}\prod_{i=r+1}^l(y_1-z_i),
$$
as desired. 
\end{proof} 

Note that the commutation relations between $D_i^{(l)}$ and $X_m$ for $i>1$ can now be obtained by conjugating 
the relations of Lemma \ref{commuu} by $S_N$. 

\begin{proposition}\label{basi} Let $l\ge 1$. 
Let $M_X$ be a monomial in $X_i$, $M_D$ a monomial in $D_i^{(l)}$, $M_y$ a monomial in $y_i$ with degrees of all the $y_i$ at most $l-1$, and $s\in S_N$. Then the elements of the form $M_XM_ysM_D$ form a basis in $\DAHA_{N,{\rm deg}}^l$; in particular, 
$\DAHA_{N,{\rm deg}}^l$ is a free $\Bbb C[\hbar,k,z_1,...,z_l]$-module. 
\end{proposition} 

\begin{proof} It is easy to see by looking at the polynomial representation that these elements are linearly independent, so we only need to establish the spanning property. Since the generators are monomials of this form, it suffices to show that any (unordered) monomial 
in $s$, $X_i$, $y_i$, $D_i^{(l)}$ can be reduced to a linear combination of such standard monomials. 

Let us introduce a filtration by setting $\deg(S_N)=0$, $\deg (y_i)=2$, $\deg(X_i)=\deg(D_i^{(l)})=l$. 
Using the above commutation relations, we can reduce any monomial to ordered form by adding corrections of lower degree. 
Further, by using the relation 
$$
X_1D_1^{(l)}=P(y_1)
$$
and its conjugates, we can reduce powers of $y_i$ to $0,...,l-1$. This implies the statement.  
\end{proof} 

Thus, we obtain the following proposition.

\begin{proposition}\label{relacycdaha}
The degenerate cyclotomic DAHA $\DAHA_{N,\rm deg}^l$ is generated by 
$S_N$ and elements $y_i,X_i,D_i$, $i=1,...,N$, with the following defining relations:

\begin{gather*}
s_iy_i=y_{i+1}s_i+k, s_iy_j=y_js_i, j\ne i,i+1,\\
[y_i,y_j]=0,\\
sX_i=X_{s(i)}s,\ s\in S_N,\ [X_i,X_j]=0,\\
[y_i,X_1]=kX_1s_{1i},\ i>1,\\
[y_1,X_1]=\hbar X_1-k\sum_{i>1}X_1s_{1i},\\
sD_i=D_{s(i)}s,\ [D_i,D_j]=0,\\
[y_j,D_1]=-ks_{1j}D_1,\ j>1,\\
[y_1,D_1]=-\hbar D_1+k\sum_{i>1}s_{1i}D_1,\\
[D_1,X_1]=\sum_{r=1}^{l}\prod_{i=1}^{r-1}(y_1-z_i+\hbar-k\sum_{j>1}s_{1j})(\hbar-k\sum_{j>1}s_{1j})\prod_{i=r+1}^l (y_1-z_i),\\
[D_1,X_m]=k\sum_{r=1}^{l}\prod_{i=1}^{r-1}(y_1-z_i+\hbar-k\sum_{j>1}s_{1j})s_{1m}\prod_{i=r+1}^l (y_1-z_i), m>1,\\
X_1D_1=(y_1-z_1)...(y_1-z_l).
\end{gather*}
\end{proposition} 

\begin{proof} This follows from Proposition \ref{basi} (with $D_i=D_i^{(l)}$). 
\end{proof} 

\begin{remark}\label{invoo} It is easy to check using this presentation that we have an involution $\phi$ on $\DAHA_{N,\rm deg}^l$ given by 
$$
\phi(\hbar)=-\hbar,\ \phi(k)=-k,\ \phi(X_i)=D_i,\ \phi(D_i)=X_i,\ \phi(s_{ij})=s_{ij},
$$
$$
\phi(y_i)=y_i+\hbar-k\sum_{j\ne i}s_{ij}.
$$ 
The existence of this involution is also clear from the isomorphism of $\DAHA_{N,\rm deg}^l$ with $\bDAHA_N^{l,{\rm psc}}$, since the latter algebra is well known to have such an involution (coming from the corresponding involution of the cyclotomic rational Cherednik algebra exchanging the coordinates with the Dunkl operators). 
\end{remark}

The proof of Proposition \ref{basi} in fact shows that for any $l\ge 0$ 
ordered products of $M_X,M_y,s,M_D$ in {\it any} of the 24 possible orders 
are a spanning set for $\DAHA_{N,\rm deg}^l$, and those of them with degrees of $y_i$ at most $l-1$ are a basis for $l\ge 1$.   
This implies that we also have another basis of this algebra, formed by monomials 
$M_XM_DsM_y$ without restriction on the degree of $y_i$, but with the restriction that for each $i$ either $X_i$ is missing in $M_X$ or $D_i$ is missing in $M_D$. Indeed, if this restriction is not satisfied, we may use the relation $X_1D_1=P(y_1)$ and its permutations to lower the number of $X_i$ and $D_i$, and it is easy to see by looking at the polynomial representation that monomials with this restriction are linearly independent. 
Thus we obtain the following proposition.

\begin{proposition}\label{basi1} The elements $M_XM_D$ which miss either $X_i$ or $D_i$ for each $i$ form a basis of $\DAHA_{N,\rm deg}^l$ as a left or right module over the degenerate affine Hecke algebra $H_{N,{\rm deg}}$ generated by $s\in S_N$ and $y_i$, $1\le i\le n$; in particular, 
$\DAHA_{N,\rm deg}^l$ is a free module over this subalgebra. 
\end{proposition} 

Note that the basis of Proposition \ref{basi1} is labeled by $N$-tuples of integers, $(m_1,...,m_N)$. Namely, if $M_XM_D$ contains $X_i^p$ then we set $m_i=p$, and if it contains $D_i^p$ then we set $m_i=-p$.    

We note that Propositions \ref{basi} and \ref{basi1} also follow from Theorem \ref{main1}.
Yet another, geometric proof of Proposition \ref{basi1} will be given in 
Section~\ref{sec5}.

\section{Cyclotomic DAHA}
\label{sec3}

\subsection{DAHA and formal DAHA}\label{dfd}
Recall the definition of Cherednik's double affine Hecke algebra (DAHA), \cite{Ch1}. Let $q,\bold t\in \Bbb C^*$, and $t={\bold t}^2$.

\begin{definition}\label{DAHAdef} The DAHA $\DAHA_N(q,t)$ is generated by invertible elements
$X_i,Y_i$, $i=1,...,N$, and $T_i$, $i=1,...,N-1$, with relations\footnote{This algebra really depends on ${\bold t}$ rather than $t={\bold t}^2$,
but it is traditional to use the parameter $t$, implying that a square root of this parameter has been chosen, see \cite{Ch1}.
While somewhat clumsy, this convention turns out to be more natural from the viewpoint of Macdonald theory.}
\begin{gather*}
(T_i-{\bold t})(T_i+{\bold t}^{-1})=0,\quad (R1)\\
T_iT_{i+1}T_i=T_{i+1}T_iT_{i+1},\quad (R2)\\
T_iT_j=T_jT_i\ (|i-j|\ge 2),\quad (R3)\\
T_iX_iT_i=X_{i+1},\quad (R4)\\
T_iX_j=X_jT_i\ (j\ne i,i+1),\quad (R5)\\
T_iY_iT_i=Y_{i+1},\quad (R6)\\
T_iY_j=Y_jT_i\ (j\ne i,i+1),\quad (R7)\\
X_1^{-1}Y_2^{-1}X_1Y_2=T_1^2,\quad (R8)\\
Y_i\tilde X=q\tilde XY_i,\quad (R9)\\
X_i\tilde Y=q^{-1}\tilde YX_i,\quad (R10)\\
[X_i,X_j]=0,\quad (R11)\\
[Y_i,Y_j]=0.\quad (R12)
\end{gather*}
where $\tilde X:=\prod_i X_i$ and $\tilde Y=\prod_i Y_i$.
\end{definition}

We can define the element
$$
T_0:=T_1^{-1}...T_{N-1}^{-1}...T_1^{-1}X_1^{-1}X_N
$$
which together with $T_i$, $i=1,...,N-1$ generates the affine Hecke algebra of type $A_{N-1}$ in the Coxeter presentation (i.e., relations (R1),(R2) are satisfied for all
$i,j\in \Bbb Z/N\Bbb Z$).

Similarly one defines the algebra $\DAHA_N$ over $\Bbb C[q^{\pm 1},\bold t^{\pm 1}]$.

We will also consider a formal version of DAHA over $\Bbb C[[\varepsilon]]$, in which $q=e^{\varepsilon \hbar}$ and $t=e^{-\varepsilon k}$ for $k\in \Bbb C$.
Namely, set $T_i=s_ie^{-\varepsilon ks_i/2}$, $Y_i=e^{\varepsilon y_i}$, and
let $\DAHA_N^{\rm formal}(\hbar,k)$ be the $\varepsilon$-adically complete algebra generated over $\Bbb C[[\varepsilon]]$
by $s_i$, $X_i$ and $y_i$ with the relations of Definition \ref{DAHAdef}. We can also treat $\hbar, k$ as indeterminates, working over $\Bbb C[\hbar,k]$.

Note that using (R6), relation (R8) can be written as
$$
X_1^{-1}T_1^{-1}Y_1^{-1}T_1^{-1}X_1T_1Y_1T_1=T_1^2,
$$
or
$$
X_1T_1Y_1=T_1Y_1T_1X_1T_1.\quad (R8a)
$$
This shows that $\DAHA_N$ has the {\it Cherednik involution} $\varphi$ defined by
$\varphi(q)=q^{-1}$, $\varphi(\bold t)=\bold t^{-1}$,
$\varphi(X_i)=Y_i^{-1}$, $\varphi(Y_i)=X_i^{-1}$,
$\varphi(T_i)=T_i^{-1}$.

\subsection{The quasiclassical limit of the formal DAHA}

The following proposition is well known; for instance, in the rank 1 case it appears in \cite{Ch1}. 

\begin{proposition}\label{qcl1} The algebra ${\DAHA}_N^{\rm formal}(\hbar,k)/(\varepsilon)$ is isomorphic to the trigonometric DAHA $\DAHA_{N,{\rm deg}}(\hbar,k)$,
and ${\DAHA}_N^{\rm formal}(\hbar,k)$ is a flat deformation of $\DAHA_{N,{\rm deg}}(\hbar,k)$.
\end{proposition}

\begin{proof} To prove the first statement, we need to show that the DAHA relations of Definition \ref{DAHAdef} degenerate to the relations of
the trigonometric DAHA.

Clearly, relation (R1) yields $s_i^2=1$, and relations (R2,R3) yield $s_is_{i+1}s_i=s_{i+1}s_is_{i+1}$ and $s_is_j=s_js_i$ for $1\le i,j\le N-1$ and $|i-j|\ge 2$.
Relations (R4,R5) give $s_iX_j=X_js_i$ if $j\ne i,i+1$, and $s_iX_i=X_{i+1}s_i$.
Relation (R6) gives a trivial relation in zeroth order, but in the first order it gives
$$
-ks_i+s_iy_is_i=y_{i+1},
$$
which yields
$$
s_iy_i=y_{i+1}s_i+k.
$$
Relation (R7) gives
$$
[s_i,y_j]=0,\ j\ne i,i+1.
$$
Relation (R8) yields
$$
y_2-X_1^{-1}y_2X_1=-ks_1,
$$
which is equivalent to
$$
[y_2,X_1]=kX_1s_1.
$$
Relations (R9,R10) yield
$$
[y_i,\prod_j X_j]=\hbar \prod_j X_j,\ [\sum_i y_i,X_j]=\hbar X_j.
$$
Finally, relations (R11, R12) yield
$$
[X_i,X_j]=0, [y_i,y_j]=0.
$$
It is easy to see that these relations are exactly the relations of $\DAHA_{N,{\rm deg}}(\hbar,k)$ given in Proposition \ref{trigDAHA} (see Remark \ref{trigDAHA1}).

The second statement of the Proposition follows from the first one and the PBW theorem for DAHA (\cite{Ch1}).
\end{proof}

\subsection{The polynomial representation of DAHA}
\begin{proposition}\label{polyre} (\cite{Ch1}) We have an action of ${\DAHA}_N(q,t)$ on $\bold P$ given by
\begin{gather*}
\rho(X_i)=X_i,\\
\rho(T_i)={\bold t}s_i+\frac{{\bold t}-{\bold t}^{-1}}{X_i/X_{i+1}-1}(s_i-1),\\
\rho(Y_i)={\bold t}^{N-1}\rho(T_i^{-1}...T_{N-1}^{-1})\omega \rho(T_1...T_{i-1}),
\end{gather*}
where $(\omega f)(X_1,...,X_N):=f(qX_N,...,X_{N-1})$.
\end{proposition}

The same formulas define a representation of $\DAHA_N^{\rm formal}(\hbar,k)$.
The representation $\rho$ of DAHA on
$\bold P$ is called the {\it polynomial representation} of DAHA.

\begin{proposition} \label{qcl2}
The quasiclassical limit (i.e., reduction modulo $\varepsilon$) of the polynomial representation of $\DAHA_N^{\rm formal}(\hbar,k)$ coincides
with the polynomial representation of the trigonometric DAHA given by Proposition \ref{chere}.
\end{proposition}

\begin{proof}
The proof is by a direct calculation.
\end{proof}

\subsection{The cyclotomic DAHA}\label{cyd}

Let $\pi\in {\DAHA}_N(q,t)$ be the element given by the formula
$$
\pi=X_1T_1...T_{N-1}.
$$
Let $Z_1,..,Z_l\in \Bbb C^*$, and $Z=(Z_1,...,Z_l)$.

\begin{definition} The subalgebra ${\DAHA}_N^l(Z,q,t)$ of ${\DAHA}_N(q,t)$ is
generated by $T_i$, $i=1,...,N-1$,
$Y_i$, $i=1,...,N$, $\pi$, and the element
$$
\pi_-:=\pi^{-1}\prod_{i=1}^l (Y_1-Z_i).
$$
\end{definition}

Let $z_1,...,z_l\in \Bbb C$, and $Z_i=q^{z_i}$ (for some choice of branches).

\begin{proposition}\label{preser}  The algebra ${\DAHA}_N^l(Z,q,t)$ preserves the space
$(\prod_j X_j)^{z_i}\bold P_+$ for all $i$.
\end{proposition}

\begin{proof} We only need to check that $\pi_-$ preserves this space.
For this, it is enough to prove that for any $u\in \Bbb C$,
the element $X_1^{-1}(Y_1-q^u)$ preserves the space $(\prod_j X_j)^u\bold P_+$.
To this end, note that
$$
\rho(T_i^{-1})={\bold t}^{-1}s_i+\frac{({\bold t}^{-1}-{\bold t})X_i}{X_{i+1}-X_i}(s_i-1).
$$
Now consider
$$
\rho(Y_1)={\bold t}^{N-1}\rho(T_1^{-1})...\rho(T_{N-1}^{-1})\omega=
$$
\begin{equation}\label{prodd}
(1+\frac{(1-t)X_1}{X_2-X_1}(1-s_{12}))...(1+\frac{(1-t)X_1}{X_N-X_1}(1-s_{1N}))\tau_1,
\end{equation}
where $\tau_j$ replaces $X_j$ with $qX_j$ and keeps $X_i$ fixed for $i\ne j$. By opening the brackets, this product can be written as a sum of $2^N$ terms
(as in each factor, we can take the first or the second summand). If we take the first summand from all factors, we get
$\tau_1$, and $X_1^{-1}(\tau_1-q^u)$ clearly preserves $(\prod_j X_j)^u\bold P_+$. So it suffices
to show that for each of the remaining $2^N-1$ terms $T$, the operator $X_1^{-1}T$ preserves
$(\prod_j X_j)^u\bold P_+$. But all of these terms have a factor $X_1$ on the left (as so does the second summand in each factor in \eqref{prodd}),
which implies the desired statement.
\end{proof}

Let ${\DAHA}_N^{l,{\rm formal}}(z,\hbar, k)$ be the formal version of $\DAHA_N^l(Z,q,t)$,
namely the subalgebra of $\DAHA_N^{\rm formal}(\hbar,k)$ generated by $T_i$, $i=1,...,N-1$,
$y_i$, $i=1,...,N$, $\pi$, and the element
$$
\varepsilon^{-l}\pi^{-1}\prod_{i=1}^l (Y_1-e^{\varepsilon z_i})=\pi^{-1}\prod_{i=1}^l\frac{e^{\varepsilon y_1}-e^{\varepsilon z_i}}{\varepsilon}.
$$

\begin{corollary}\label{coro1} Let $z_i-z_j\notin \Bbb Z$, and $k$ be Weil generic. Then the algebra ${\DAHA}_N^{l,{\rm formal}}(z,1,k)$ is the algebra of all elements
of ${\DAHA}_N^{\rm formal}(1,k)$ which preserve $(\prod X_j)^{z_i}\bold P_+$ for all $i$.
\end{corollary}

\begin{proof} Recall that by Proposition \ref{qcl1}, ${\DAHA}_N^{\rm formal}(\hbar,k)$ is a flat deformation of $\DAHA_{N,{\rm deg}}(\hbar,k)$, and by Proposition \ref{qcl2},
the same applies to the polynomial representations of these algebras. Thus the result follows by a deformation argument
from Proposition \ref{preser} and the fact that a similar statement holds in the trigonometric case (Theorem \ref{charac}(ii)).
Namely, the algebra ${\DAHA}_N^{l,{\rm formal}}(z,1,k)$ is a priori ``at least as big" as $\DAHA_{N,{\rm deg}}^l(z,1,k)$, as its generators are deformations of generators of
$\DAHA_{N,{\rm deg}}^l(z,1,k)$.
At the same time, the subalgebra of elements of $\DAHA_N^{l,{\rm formal}}(z,1,k)$
preserving $(\prod_j X_j)^{z_i}\bold P_+$ for all $i$ is ``at most as big" as $\DAHA_{N,{\rm deg}}^l(z,1,k)$, as
by Theorem \ref{charac}(ii), this condition cuts out $\DAHA_{N,{\rm deg}}^l(z,1,k)$ inside $\DAHA_{N,{\rm deg}}(1,k)$.
But by Theorem \ref{preser}, the former subalgebra is contained in the latter one. This implies the corollary.
\end{proof}

\begin{theorem}\label{main2} For any $z_1,...,z_l$, the algebra $\DAHA_N^{l,{\rm formal}}(z,\hbar,k)$ is a flat deformation of $\DAHA_{N,{\rm deg}}^l(z,\hbar,k)$.
\end{theorem}

\begin{proof} Since  $\DAHA_N^{l,{\rm formal}}(z,\hbar,k)$ is generated by deformations of generators of $\DAHA_{N,{\rm deg}}^l(z,\hbar,k)$,
it suffices to prove this statement for Weil generic $z_i$, $\hbar$, $k$, but in this case it follows from Corollary \ref{coro1}.
\end{proof}

Another proof of Theorem \ref{main2} is obtained from the presentation of $\DAHA_N^l$ given below. 

\begin{definition}\label{cyclotomDAHA} The algebra $\DAHA_N^l(Z,q,t)$ is called the {\it cyclotomic DAHA}
and ${\DAHA}_N^{l,{\rm formal}}(z,\hbar,k)$ is called the {\it formal cyclotomic DAHA.}
\end{definition}

As usual, one can similarly define versions of these algebras where the parameters are indeterminates.
Note also that by this definition $\DAHA_N^0(q,t)=\DAHA_N(q,t)$, and $\DAHA_N^{l'}(Z',q,t)\subset \DAHA_N^l(Z,q,t)$ if $l'\ge l$ and $Z'\supset Z$ as a multiset.

Thus, the cyclotomic DAHA is a $q$-deformation of the partly spherical cyclotomic rational Cherednik algebra. More precisely,
it follows from Theorem \ref{main1} and Theorem \ref{main2} that for any $z_1,..,z_l,\hbar,k$, the formal cyclotomic DAHA $\DAHA_N^{l,{\rm formal}}(z,\hbar,k)$
is a flat deformation of $\bDAHA_N^{l,{\rm psc}}(c,\hbar,k)$, where $c$ is related to $z$ by equation \eqref{ziform1}.
In particular, the cyclotomic DAHA is interesting already for $l=1$, as it provides a $q$-deformation of the rational Cherednik algebra
$\bDAHA_N^{\rm rat}(\hbar,k)$ attached to $S_N$ and its permutation representation.

\begin{remark}  Since the DAHA has a $\Bbb G_m$-action given by $Y_i\mapsto aY_i$ and trivial on other generators, the algebra $\DAHA_N^l(Z,q,t)$
does not change under the transformation $Z_i\to aZ_i$ (i.e., it depends only on the ratios $Z_i/Z_{i+1}$).
\end{remark}

\begin{example}\label{Neq1} Let $N=1$. Then there is no dependence on $t$, and $\DAHA_1(q,t)=\DAHA_1(q)$ is the quantum torus algebra
with invertible generators $X,Y$ and relation $YX=qXY$. Let $q$ be not a root of unity. The polynomial representation is $\bold P=\Bbb C[X^{\pm 1}]$ with $Y$ acting by shift, $(Yf)(X)=f(qX)$, so that $\DAHA_1(q)$ is the algebra of polynomial $q$-difference operators.
The subalgebra $\DAHA_1^l(Z,q,t)=\DAHA_1^l(Z,q)$ inside $\DAHA_1(q)$ is generated by $X,Y^{\pm 1}$, and
$$
L:=X^{-1}(Y-q^{z_1})...(Y-q^{z_l}),
$$
 where $Z_i=q^{z_i}$. We claim that if $Z_i/Z_j\notin{q^{\Bbb Z}}$ then $\DAHA_1^l(Z,q)$ is exactly the subalgebra of all difference operators preserving $X^{z_i}\Bbb C[X]$
for all $i$. Indeed, if we set $\deg(Y)=0$, $\deg(X)=1$, then any difference operator of nonnegative degree is in both subalgebras, while an
operator $M$ of degree $-d<0$ has the form $X^{-d}g(Y)$, where $g$ is a Laurent polynomial, and applying this operator to $X^{z_i+j}$, $j<d$,
we must get zero, so we get $g(q^{z_i+j})=0$, $i=1,...,l$, $j=0,...,d-1$. Thus $M=L^dh(Y)$ for some polynomial $h$.
\end{example}

\begin{remark}\label{rai} We expect that Corollary \ref{coro1} holds in the non-formal setting, i.e., $\DAHA_N^l(Z,q,t)$ can be characterized as the algebra of elements of $\DAHA_N(q,t)$ preserving the spaces $(\prod X_j)^{z_i}\bold P_+$, as in Proposition \ref{preser}. For $N=1$ this is demonstrated in~Example~\ref{Neq1}. Moreover, recall that Ginzburg, Kapranov, and Vasserot (\cite{GKV}) characterized DAHA (for any Weyl group $W$) as the algebra 
of difference-reflection operators $L=\sum_{w\in W}L_ww$ (where $L_w$ are difference operators) 
satisfying some residue conditions. These conditions are equivalent to the conditions that $L$ preserves $\bold P$ and $\Delta_{q,t}\bold P$, where 
$\Delta_{q,t}$ is an appropriate meromorphic function. Therefore, we expect that $\DAHA_N^l(Z,q,t)$ 
can be characterized as the algebra of difference operators preserving the spaces $\bold P$, $\Delta_{q,t}\bold P$,
and $(\prod X_j)^{z_i}\bold P_+$ for $i=1,...,l$. 

We note that this approach to DAHA-type algebras in the more general elliptic setting is developed in the ongoing work \cite{R2}. 
In the one-variable case $N=1$ such (spherical) algebras generated by a given set of difference-reflection operators
have been studied in \cite{R1}.   
\end{remark} 

\subsection{The case $l=1$}
Let us study the algebra $\DAHA_N^l$ for $l=1$ and give its presentation. These results can be derived from the
case of general $l$ considered below, but the special case $l=1$ is especially nice, and it is instructive to do it separately first. 

By rescaling $Y_i$ by the same scalar, we may assume without loss of generality that $Z_1=1$.
Then the algebra $\DAHA_N^1(q,t):=\DAHA_N^1(1,q,t)$ is generated inside $\DAHA_N(q,t)$ by $T_i$, $X_i$, $Y_i^{\pm 1}$, and
$X_1^{-1}(Y_1-1)$.

This algebra actually appeared a long time ago in the paper \cite{BF}. Let us describe it in more detail, following \cite{BF}.
(We note that our conventions are slightly different from those of \cite{BF}).

Define the Dunkl elements $D_i:=T_{i-1}^{-1}...T_1^{-1}D_1T_1^{-1}...T_{i-1}^{-1}\in \DAHA_N^1(q,t)$, where $D_1=X_1^{-1}(Y_1-1)$.
It is easy to check that
$$
D_i=T_iD_{i+1}T_i,\ [T_i,D_j]=0\text{ for }j\ne i,i+1.
$$

\begin{lemma} \label{D12}
We have $[D_1,D_2]=0$.
\end{lemma}

\begin{proof} By definition, we have $D_2=T_1^{-1}D_1T_1^{-1}$, so our job is to show that
$$
T_1^{-1}D_1T_1^{-1}D_1=D_1T_1^{-1}D_1T_1^{-1}.
$$
In other words, we must show that
$$
T_1^{-1}X_1^{-1}(Y_1-1)T_1^{-1}X_1^{-1}(Y_1-1)=X_1^{-1}(Y_1-1)T_1^{-1}X_1^{-1}(Y_1-1)T_1^{-1}.
$$
Since $T_1X_1T_1X_1=X_2X_1=X_1X_2=X_1T_1X_1T_1$, it suffices to prove two identities in $\DAHA_N(q,t)$:
\begin{equation}\label{id1}
T_1^{-1}X_1^{-1}Y_1T_1^{-1}X_1^{-1}Y_1=X_1^{-1}Y_1T_1^{-1}X_1^{-1}Y_1T_1^{-1}
\end{equation}
and
\begin{multline}\label{id2}
T_1^{-1}X_1^{-1}T_1^{-1}X_1^{-1}Y_1+T_1^{-1}X_1^{-1}Y_1T_1^{-1}X_1^{-1}=\\
=X_1^{-1}T_1^{-1}X_1^{-1}Y_1T_1^{-1}+X_1^{-1}Y_1T_1^{-1}X_1^{-1}T_1^{-1}.
\end{multline}
Identity \eqref{id1} actually holds already in the braid group. Indeed, we have
$$
T_1^{-1}X_1^{-1}Y_1T_1^{-1}X_1^{-1}Y_1=
T_1^{-1}X_1^{-1}T_1^{-1}Y_2T_1^{-2}X_1^{-1}Y_1=
T_1^{-1}X_1^{-1}T_1^{-1}X_1^{-1}Y_2Y_1=
$$
$$
X_1^{-1}T_1^{-1}X_1^{-1}T_1^{-1}Y_2Y_1=
X_1^{-1}T_1^{-1}X_1^{-1}Y_2Y_1T_1^{-1}=
$$
$$
X_1^{-1}T_1^{-1}Y_2T_1^{-2}X_1^{-1}Y_1T_1^{-1}=
X_1^{-1}Y_1T_1^{-1}X_1^{-1}Y_1T_1^{-1}.
$$
It remains to establish \eqref{id2}. Note that because $T_1^{-1}=T_1-{\bold t}+{\bold t}^{-1}$, \eqref{id2} is equivalent to
\begin{multline}\label{id3}
T_1X_1^{-1}T_1^{-1}X_1^{-1}Y_1+T_1^{-1}X_1^{-1}Y_1T_1^{-1}X_1^{-1}=\\
=X_1^{-1}T_1^{-1}X_1^{-1}Y_1T_1+X_1^{-1}Y_1T_1^{-1}X_1^{-1}T_1^{-1}.
\end{multline}
On the other hand, \eqref{id3} holds already in the group algebra of the braid group (i.e., termwise). Indeed, we have
$$
T_1^{-1}X_1^{-1}Y_1T_1^{-1}X_1^{-1}=T_1^{-1}X_1^{-1}T_1^{-1}Y_2T_1^{-2}X_1^{-1}=
$$
$$
T_1^{-1}X_1^{-1}T_1^{-1}X_1^{-1}Y_2=X_1^{-1}T_1^{-1}X_1^{-1}T_1^{-1}Y_2=X_1^{-1}T_1^{-1}X_1^{-1}Y_1T_1,
$$
and
$$
T_1X_1^{-1}T_1^{-1}X_1^{-1}Y_1=X_1^{-1}T_1^{-1}X_1^{-1}T_1Y_1=X_1^{-1}T_1^{-1}X_1^{-1}Y_2T_1^{-1}=
$$
$$
X_1^{-1}T_1^{-1}Y_2T_1^{-2}X_1^{-1}T_1^{-1}=X_1^{-1}Y_1T_1^{-1}X_1^{-1}T_1^{-1}.
$$
\end{proof}

\begin{corollary}\label{Dij}
One has $[D_i,D_j]=0$ for all $i,j$.
\end{corollary}

\begin{proof} For any $j>1$, we have
$$
[D_1,D_j]=[D_1,T_{j-1}^{-1}...T_2^{-1}D_2T_2^{-1}...T_{j-1}^{-1}]=0,
$$ since $D_1$ commutes with every factor by Lemma \ref{D12}. Hence for $i<j$,
$$
[D_i,D_j]=[T_{i-1}^{-1}...T_1^{-1}D_1T_1^{-1}...T_{i-1}^{-1},D_j]=0,
$$
again because $D_j$ commutes with every factor.
\end{proof}

Thus, $T_i$ and $D_i$ generate the ``positive part" of an affine Hecke algebra.

When $q\to 1$ and $t=q^{-k}$, where $k$ is fixed, the algebra $\DAHA_N^1(q,t)$ degenerates to the rational Cherednik algebra $\DAHA_N^{\rm rat}(1,k)$ for $S_N$ (associated to the permutation representation), and $X_i,\frac{D_i}{q-1}$ degenerate to the standard generators of $\DAHA_N^{\rm rat}(1,k)$. Thus, let us compute the commutation relations between $D_i$ and $X_j$
which deform the corresponding relations of $\DAHA_N^{\rm rat}(1,k)$.

\begin{lemma}\label{XD} One has
$$
X_1D_2=D_2T_1^2X_1+({\bold t}-{\bold t}^{-1})T_1^{-1}
$$
and
$$
D_1X_2=X_2T_1^{-2}D_1-({\bold t}-{\bold t}^{-1})T_1.
$$
\end{lemma}

\begin{proof}
We have
$$
D_2X_1^{-1}=T_1^{-1}X_1^{-1}Y_1T_1^{-1}X_1^{-1}-T_1^{-1}X_1^{-1}T_1^{-1}X_1^{-1}
$$
Thus, by the proof of Lemma \ref{D12}, we obtain
$$
D_2X_1^{-1}=X_1^{-1}T_1^{-1}X_1^{-1}Y_1T_1-X_1^{-1}T_1^{-1}X_1^{-1}T_1^{-1}=
$$
$$
X_1^{-1} D_2T_1^2+X_1^{-1}T_1^{-1}X_1^{-1}T_1-X_1^{-1}T_1^{-1}X_1^{-1}T_1^{-1}=
$$
$$
X_1^{-1}D_2T_1^2+({\bold t}-{\bold t}^{-1})X_1^{-1}T_1^{-1}X_1^{-1}.
$$
Now the first relation of the lemma is obtained by multiplying both sides by $X_1$ on the left and on the right.
To obtain the second relation of the lemma from the first one, it suffices to multiply the first relation by $T_1$ on both sides, and apply the commutation relations between $D_i,X_i$ and $T_j$.
\end{proof}

\begin{corollary}\label{DXij} If $i<j$, one has
$$
X_iD_j=D_jT_{j-1}...T_{i+1}T_i^2T_{i+1}^{-1}...T_{j-1}^{-1}X_i+({\bold t}-{\bold t}^{-1})T_{j-1}^{-1}...T_i^{-1}...T_{j-1}^{-1}.
$$
If $i>j$, one has
$$
D_jX_i=X_iT_{i-1}^{-1}...T_{j+1}^{-1}T_j^{-2}T_{j+1}...T_{i-1}D_j-({\bold t}-{\bold t}^{-1})T_{i-1}...T_j...T_{i-1}.
$$
\end{corollary}

\begin{proof} First consider the case $i=1$. Then
$$
X_1D_j=X_1T_{j-1}^{-1}...T_2^{-1}D_2T_2^{-1}...T_{j-1}^{-1}=
T_{j-1}^{-1}...T_2^{-1}X_1D_2T_2^{-1}...T_{j-1}^{-1}.
$$
By the first relation of Lemma \ref{XD}, this implies
$$
X_1D_j=T_{j-1}^{-1}...T_2^{-1}D_2T_1^2X_1T_2^{-1}...T_{j-1}^{-1}+({\bold t}-{\bold t}^{-1})T_{j-1}^{-1}...T_2^{-1}T_1^{-1}T_2^{-1}...T_{j-1}^{-1}=
$$

$$
D_jT_{j-1}...T_2T_1^2T_2^{-1}...T_{j-1}^{-1}X_1+({\bold t}-{\bold t}^{-1})T_{j-1}^{-1}...T_2^{-1}T_1^{-1}T_2^{-1}...T_{j-1}^{-1},
$$
as claimed. Now consider the general case. We have
$$
X_iD_j=T_{i-1}...T_1X_1T_1...T_{i-1}D_j=T_{i-1}...T_1X_1D_jT_1...T_{i-1}=
$$
$$
T_{i-1}...T_1D_jT_{j-1}...T_2T_1^2T_2^{-1}...T_{j-1}^{-1}X_1T_1...T_{i-1}+
$$
$$
({\bold t}-{\bold t}^{-1})T_{i-1}...T_1T_{j-1}^{-1}...T_2^{-1}T_1^{-1}T_2^{-1}...T_{j-1}^{-1}
T_1...T_{i-1}=
$$
$$
D_jT_{j-1}...T_{i+1}T_i^2T_{i+1}^{-1}...T_{j-1}^{-1}X_1+({\bold t}-{\bold t}^{-1})T_{j-1}^{-1}...T_{i+1}^{-1}T_i^{-1}T_{i+1}^{-1}...T_{j-1}^{-1},
$$
(here we repeatedly used the braid relations between the $T_i$). This proves the first relation of the corollary.

The second relation is proved similarly, using the second relation of Lemma \ref{XD}.
\end{proof}

Finally, let us generalize the commutation relations between $D_i$ and $X_i$ (which are the only relations containing $q$).
It turns out that it is convenient to write instead the commutation relations between $D_i$ and $\sum_{j=1}^N X_j$.

\begin{lemma}\label{D1S} We have
$$
[D_i,\sum_{j=1}^N X_j]=(1-q^{-1})(D_iX_i+1)T_{i-1}^{-1}...T_1^{-2}...T_{i-1}^{-1}.
$$
\end{lemma}

\begin{proof} First assume $i=1$. Looking at the polynomial representation and using that $\sum_{j=1}^N X_j$ is symmetric, we
find that
$$
[D_1,\sum_{j=1}^N X_j]=(1-q^{-1})(D_1X_1+1).
$$
The general case is now obtained by multiplying both sides by $T_{i-1}^{-1}...T_1^{-1}$ on the left and by
$T_1^{-1}...T_{i-1}^{-1}$ on the right.
\end{proof}

We obtain the following theorem.

\begin{theorem}\label{genrel} (i) Let $\DAHA_N^{1,+}(q,t)$ be the subalgebra of $\DAHA_N^1(q,t)$ generated by $T_i$, $1\le i\le N-1$, $X_i,D_i$, $1\le i\le N$.
Then the defining relations for $\DAHA_N^{1,+}(q,t)$ are
\begin{gather*}
(T_i-{\bold t})(T_i+{\bold t}^{-1})=0,\\
T_iT_{i+1}T_i=T_{i+1}T_iT_{i+1},\\
T_iT_j=T_jT_i\ (|i-j|\ge 2),\\
T_iX_iT_i=X_{i+1},\ [T_i,X_j]=0\text{ for }j\ne i,i+1.\\
D_i=T_iD_{i+1}T_i,\ [T_i,D_j]=0\text{ for }j\ne i,i+1.\\
[X_i,X_j]=0,\\
[D_i,D_j]=0,\\
X_iD_j=D_jT_{j-1}...T_{i+1}T_i^2T_{i+1}^{-1}...T_{j-1}^{-1}X_i+({\bold t}-{\bold t}^{-1})T_{j-1}^{-1}...T_i^{-1}...T_{j-1}^{-1},\ i<j,\\
D_jX_i=X_iT_{i-1}^{-1}...T_{j+1}^{-1}T_j^{-2}T_{j+1}...T_{i-1}D_j-({\bold t}-{\bold t}^{-1})T_{i-1}...T_j...T_{i-1},\ i>j,\\
[D_i,\sum_{j=1}^N X_j]=(1-q^{-1})(D_iX_i+1)T_{i-1}^{-1}...T_1^{-2}...T_{i-1}^{-1}.
\end{gather*}

(ii) (the PBW theorem) For any values of parameters, the elements $\prod_i X_i^{m_i}\cdot T_w\cdot \prod_i D_i^{n_i}$
form a basis of $\DAHA_N^{1,+}(q,t)$.

(iii) The algebra $\DAHA_N^1(q,t)$ is obtained from $\DAHA_N^{1,+}(q,t)$ by inverting the element $Y_1:=1+X_1D_1$.
\end{theorem}

\begin{proof} We have shown in Lemma \ref{Dij}, Corollary \ref{DXij}, Lemma \ref{D1S} that the claimed relations between $T_i,X_i,D_i$ are satisfied.
These relations allow us to order any monomial as claimed in (ii). Since $\DAHA_N^{1,{\rm formal}}$ is a flat deformation of the rational
Cherednik algebra $\DAHA_N^{\rm rat}(\hbar,k)$, the relations in (i) are defining, and the monomials in (ii) are linearly independent. This implies both (i) and (ii).

Now, by definition, $\DAHA_N^1(q,t)$ is generated inside $\DAHA_N(q,t)$ by $\DAHA_N^{1,+}(q,t)$ and $Y_1^{-1}$, where $Y_1=1+X_1D_1$. This implies (iii).
\end{proof}

The following proposition shows that there is a symmetry between $X_i$ and $D_i$. Let us regard $q$ and $\bold t$ as variables.

\begin{proposition}\label{phiauto} There is an involutive automorphism $\phi$ of $\DAHA_N^1$ such that $\phi(X_i)=D_i$, $\phi(D_i)=X_i$, $\phi(T_i)=T_i^{-1}$, $\phi(\bold t)=\bold t^{-1}$, $\phi(q)=q^{-1}$.
This automorphisms preserves the subalgebra $\DAHA_N^{1,+}$.
\end{proposition}

\begin{proof} It is easy to see that all the relations of Theorem \ref{genrel}
except the last one are invariant under $\phi$. It is sufficient to impose the last relation
for $i=1$, so let us consider this case. The relation has the form
\begin{gather*}
q^{-1}(D_1X_1+1)-(X_1D_1+1)=\\
({\bold t}-{\bold t}^{-1})\sum_{i>1}(X_iT_{i-1}^{-1}...T_2^{-1}T_1^{-1}T_2...T_{i-1}D_1+T_{i-1}...T_1...T_{i-1}).\\
=({\bold t}-{\bold t}^{-1})(\sum_{i>1}T_{i-1}...T_1...T_{i-1})(X_1D_1+1),
\end{gather*}
i.e.,
$$
D_1X_1+1=q(1+({\bold t}-{\bold t}^{-1})\sum_{i>1}T_{i-1}...T_1...T_{i-1})(X_1D_1+1).
$$
Using braid relations, we see that $T_i...T_1...T_i=T_1...T_i...T_1$, so this can be also written as
$$
D_1X_1+1=q(1+({\bold t}-{\bold t}^{-1})\sum_{i>1}T_1...T_{i-1}...T_1)(X_1D_1+1).
$$
Now let $J_N:=T_1...T_{N-1}^2...T_1$ be the braid-like Jucys-Murphy element.
Then, using the quadratic relation for $T_{N-1}$, we get
$$
J_N=({\bold t}-{\bold t}^{-1})T_1...T_{N-1}...T_1+J_{N-1}.
$$
This yields
$$
J_N=1+({\bold t}-{\bold t}^{-1})\sum_{i>1}T_1...T_{i-1}...T_1,
$$
so our relation can be simplified to
\begin{equation}\label{lastrel}
D_1X_1+1=qJ_N(X_1D_1+1).
\end{equation}
The invariance of this relation under $\phi$
reduces to the identity $J_N\phi(J_N)=1$, which is obvious from the definitions of $\phi$ and $J_N$.
This shows that $\phi$ is a well defined automorphism of $\DAHA_N^{1,+}$, which is obviously involutive.

We also see from \eqref{lastrel} that $\phi(Y_1)=qJ_NY_1$, which implies that $\phi$ extends to $\DAHA_N^1$.
\end{proof}

\begin{corollary} (i)  The last defining relation of $\DAHA_N^{1,+}(q,t)$ may be replaced with
$$
[X_i,\sum_{j=1}^N D_j]=(1-q)(X_iD_i+1)T_{i-1}...T_1^2...T_{i-1}.
$$

(ii) The last defining relation of $\DAHA_N^{1,+}(q,t)$ may be replaced with \eqref{lastrel}.
\end{corollary}

\begin{proof} (i) follows immediately from Theorem \ref{genrel}, and (ii) from its proof.
\end{proof}

\subsection{A commutative subalgebra in $\DAHA_N^l$}\label{insys} 

In this subsection, we will construct a commutative subalgebra
inside $\DAHA_N^l$.
In the case $l=1$, this subalgebra will reduce to the subalgebra $\Bbb C[D_1,...,D_N]$ constructed in the
previous subsection (so we will obtain an alternative construction of this subalgebra).

Let $f\in \Bbb C[X]$ be any polynomial.
Define the elements
$$
Y_i(f):=Y_iT_{i-1}^{-1}...T_1^{-1}f(X_1^{-1})T_1...T_{i-1}\in \DAHA_N.
$$
Also let $\bold e$ be the symmetrizer of the finite Hecke algebra generated by the $T_i$.
(To define $\bold e$, we need to invert $[N]_t!$).  

\begin{lemma}\label{commu}
(i) The elements $Y_i(f)$ are pairwise commuting:
$$
[Y_i(f),Y_j(f)]=0.
$$

(ii) For $r\ge 1$ the element ${\mathcal M}_r(f)=(\sum_{i=1}^N e_r(Y_1(f),...,Y_N(f)))\bold e$ (where $e_r$ is the $r$-th elementary symmetric function)
commutes with $T_i$, i.e. ${\mathcal M}_r(f)=\bold e {\mathcal M}_r(f)$.
\end{lemma}

\begin{proof} (i) We may specialize the variables to numerical values.
Let us compute the action of $Y_i(f)$ in the polynomial representation. Using Proposition \ref{polyre}, we get
$$
\rho(Y_i(f))={\bold t}^{N-1}\rho(T_i^{-1}...T_{N-1}^{-1})\omega f(X_1)\rho(T_1...T_{i-1}).
$$
Let $g(X)$ be a meromorphic function such that $g(q^{-1}X)=g(X)f(X^{-1})$.
Let us extend the polynomial representations to meromorphic functions,
and conjugate it by the function $G(X_1,...,X_N):=g(X_1)...g(X_N)$
This gives a representation $\rho_G$ such that
$$
\rho_G(Y_i)=G\rho(Y_i)G^{-1}={\bold t}^{N-1}\rho(T_i^{-1}...T_{N-1}^{-1})G\omega G^{-1} \rho(T_1...T_{i-1})=
$$
$$
{\bold t}^{N-1}\rho(T_i^{-1}...T_{N-1}^{-1})\omega \frac{G(q^{-1}X_1)}{G(X_1)} \rho(T_1...T_{i-1})=
$$
$$
{\bold t}^{N-1}\rho(T_i^{-1}...T_{N-1}^{-1})\omega f(X_1^{-1}) \rho(T_1...T_{i-1})=\rho(Y_i(f)).
$$
Thus $\rho(Y_i(f))$ are pairwise commuting, hence so are $Y_i(f)$.

(ii) It suffices to show that ${\mathcal M}_r(f)$ maps symmetric polynomials to symmetric ones.
But this follows from the fact that $\rho({\mathcal M}_r(f))=\rho_G({\mathcal M}_r)$,
where ${\mathcal M}_r=\sum_{i=1}^Ne_r(Y_1,...,Y_N)$, and ${\mathcal M}_r$ act
in the polynomial representation by Macdonald difference operators, \cite{Ch1}.
\end{proof}

Now let $f(X)=(X-Z_1)...(X-Z_l)$. Let us apply the Cherednik involution $\varphi$ to the elements $Y_i(f)$.
Namely, let
$$
D_i^{(l)}:=\varphi(Y_i(f))=T_{i-1}^{-1}...T_1^{-1}D_1^{(l)}T_1^{-1}...T_{i-1}^{-1}\in \DAHA_N,
$$
where, $D_1^{(l)}:=X_1^{-1}(Y_1-Z_1)...(Y_1-Z_l)$.

\begin{corollary}\label{commu1} (i) The elements $D_i^{(l)}$, $i=1,...,N$,
are pairwise commuting.

(ii) The elements ${\bold M}_r(f):=\varphi({\mathcal M}_r(f))$ commute with $T_i$.
\end{corollary}

\begin{proof} This follows from Lemma \ref{commu}.
\end{proof}

\begin{example} Let $l=1$ and $f(X)=X-1$. Then we get
$D_i^{(1)}=D_i$. Thus, we recover the $q$-deformed Dunkl operators
of \cite{BF} described in the previous subsection and therefore
obtain another proof of Corollary \ref{Dij}.
\end{example}

Corollary \ref{commu1}(ii) implies that the elements ${\bold M}_r(f)$ act on symmetric functions by some commuting symmetric difference operators
$M_r(f)=:M_r^{(l)}$. Thus for each $l$ we obtain a family of quantum integrable systems $\lbrace{ M_1^{(l)},...,M_N^{(l)}\rbrace}$
depending on $l$ parameters $Z_1,...,Z_l$ (and also $q$, $t$). This system is a $q$-deformation
of the cyclotomic Calogero-Moser system.

\begin{example}\label{bfpaper} (\cite{BF}, Lemma 5.3) We have 
$$
M_1^{(1)}=\sum_{j=1}^N \left(\prod_{i\ne j}\frac{X_i-tX_j}{X_i-X_j}\right)\frac{1}{X_j}(\tau_j-1).
$$
Thus, this operator defines a quantum integrable system. 
\end{example}

\begin{remark}\label{ruij} J. F. Van Diejen and S. Ruijsenaars have explained to us that the system defined by the operator $M_1:=M_1^{(1)}$ from 
Example \ref{bfpaper} may be obtained as a limit of the system from \cite{DE} defined by the Hamiltonian (3.13a). Namely, conjugating by 
the Gaussian $\exp(\sum_i \frac{(\log X_i)^2}{2\log q})$ and rescaling $X_i$ and $M_1$, we can reduce the operator $M_1$ to the form
$$
M_1'= \sum_{j=1}^N \left(\prod_{i\ne j}\frac{X_i-tX_j}{X_i-X_j}\right)\tau_j-\sum_{j=1}^N X_j^{-1}=M-\sum_{j=1}^N X_j^{-1},
$$
where $M$ is the first Macdonald operator. (Here we use the identity at the beginning of~\cite[p.~1621]{DE}). 
On the other hand, let us multiply the Hamiltonian~\cite[(3.13a)]{DE} by $\hat{t}_0$ and then
put $\hat{t}_1 =  \hat{t}_0^{-1}$ and send $\hat{t}_0$ to $0$. Conjugating the resulting operator by an appropriate function, and 
again using the identity at the beginning of~\cite[p.~1621]{DE}, one obtains the operator $M_1'$. 
\end{remark}

\begin{remark}\label{cha1} 
It is interesting to compute the joint eigenfunctions of $D_1^{(l)},...D_N^{(l)}$ and symmetric joint eigenfunctions of $M_1^{(l)},...,M_N^{(l)}$. For
$l=1$, this is done in \cite{BF}. Let us sketch how this can be done
for any $l$ in the nonsymmetric case (without working out any details).

For simplicity, assume that $Z_1...Z_l=(-1)^l$ (this can be assumed without loss of generality, as we can simultaneously rescale the $Z_i$).
Then $f(0)=1$, so for $q>1$ we can set
$g(X)=\prod_{m=1}^\infty f(q^{-m}X^{-1})$.
Given a collection $\Lambda=(\lambda_1,...,\lambda_N)$ of
eigenvalues, let $F(X_1,...,X_N,\Lambda)$ be the joint eigenfunction
of $Y_i$ (i.e., the nonsymmetric Macdonald function).
Then
$$
\widetilde{F}(X_1,...,X_N,\Lambda):=g(X_1)...g(X_N)F(X_1,...,X_N,\Lambda)
$$
is a joint eigenfunction of $Y_i(f)$. So the joint eigenfunctions of $D_i^{(l)}$ can be obtained by applying Cherednik's difference Fourier
transform (\cite{Ch1}) to $\widetilde{F}(X_1,...,X_N,\Lambda)$:
\begin{equation}\label{fourier} 
\widehat{F}(X_1,...,X_N,\Lambda)={\mathcal F}_{\rm Cherednik}(\widetilde{F}(X_1,...,X_N,\Lambda)).
\end{equation}
For $l=1$ this should recover the formulas of \cite{BF}.

Also, the following observation was made by O. Chalykh when we sent him a preliminary version of this paper. 
Consider the case where $Z_1=...=Z_l=0$ (this violates our restriction that $Z_i\in \Bbb C^*$, but this restriction is not essential here). 
In this case, $f(X)=X^l$, and one can take $g(X)$ to be the $l$-th power of the Gaussian 
$g(X)=\exp(\frac{(\log X)^2}{2\log q})$ (this function is not single-valued but this is not important if we restrict to the locus $q,X\in \Bbb R_+$). 
Therefore, the symmetrized version of formula \eqref{fourier} turns into formula (7.5) in O. Chalykh's appendix to \cite{CE} (up to changing $q$ to $q^{-1}$). 
This shows that the Hamiltonians of the twisted Macdonald-Ruijsenaars model of Theorem 7.1 in the appendix to \cite{CE} (for type $A_{N-1}$) 
are a special case of the commuting Hamiltonians $e_r(D_1^{(l)},...,D_N^{(l)})$ when $Z_1=...=Z_l=0$. 
\end{remark}

\subsection{Presentation of cyclotomic DAHA by generators and relations}
Let us give a presentation of $\DAHA_N^l(Z,q,t)$ by generators and relations. 
We will use as generators the elements $T_i$, $i=1,...,N-1$, and $X_i,Y_i^{\pm 1},D_i^{(l)}$, $i=1,...,N$. 

First of all, the elements $T_i,X_i,Y_i^{\pm 1}$ satisfy the relations of DAHA, (R1-R12) from Definition \ref{DAHAdef}. 
More precisely, we need to rewrite relations (R8-R10) to account for non-invertibility of $X_i$. 
Namely, from (R8-R10) we get
\begin{equation}\label{cd1}
X_iY_j=Y_jX_iT_{j-1}^{-1}...T_{i+1}^{-1}T_i^2T_{i+1}...T_{j-1},\ i<j. 
\end{equation}
Similarly, we have 
\begin{equation}\label{cd2}
Y_iX_j=T_{j-1}^{-1}...T_{i+1}^{-1}T_i^2T_{i+1}...T_{j-1}X_jY_i,\ i<j. 
\end{equation}
Finally, we have 
\begin{equation}\label{cd3}
Y_iT_{i-1}^{-1}...T_1^{-2}...T_{i-1}^{-1}X_i=qX_iT_i...T_{N-1}^2...T_iY_i.
\end{equation} 
Secondly, we have similar relations between $D_j^{(l)}$ (instead of $X_j$) and $T_i,Y_i$. 
Namely, we know from Corollary \ref{commu1} that 
\begin{equation}\label{cd4}
[D_i^{(l)},D_j^{(l)}]=0, 
\end{equation} 
and it follows from the definition that 
\begin{equation}\label{cd5}
T_i^{-1}D_i^{(l)}T_i^{-1}=D_{i+1}^{(l)},\ [T_j,D_i^{(l)}]=0\text{ for }|i-j|\ge 2.
\end{equation} 
Also since $X_1^{-1}Y_2=Y_2T_1^{-2}X_1^{-1}$, we have 
$$
D_1^{(l)}Y_2=Y_2T_1^{-2}D_1^{(l)}.
$$
This implies that 
\begin{equation}\label{cd6}
D_i^{(l)}Y_j=Y_jT_{j-1}^{-1}....T_{i+1}^{-1}T_i^{-2}T_{i+1}...T_{j-1}D_i^{(l)},\ i<j.
\end{equation} 
Similarly, we have 
$$
D_2^{(l)}T_1^2Y_1=Y_1D_2^{(l)},
$$
which implies that 
\begin{equation}\label{cd7}
D_j^{(l)}T_{j-1}...T_{i+1}T_i^2T_{i+1}^{-1}...T_{j-1}^{-1}Y_i=Y_iD_j^{(l)},\ i<j. 
\end{equation} 
Finally, we have 
$$
D_1^{(l)}Y_1=qT_1...T_{N-1}^2...T_1Y_1D_1^{(l)},
$$
which gives
\begin{equation}\label{cd8} 
D_i^{(l)}Y_iT_{i-1}^{-1}...T_1^{-2}...T_{i-1}^{-1}=T_i...T_{N-1}^2...T_iY_iD_i^{(l)}.
\end{equation} 

Finally, we write commutation relations between $X_i$ and $D_j^{(l)}$. 
First of all, we have 
\begin{equation}\label{cd9}
X_1D_1^{(l)}=(Y_1-Z_1)...(Y_1-Z_l),\ D_1^{(l)}X_1=(qJ_NY_1-Z_1)...(qJ_NY_1-Z_l), 
\end{equation} 
where $J_N=T_1...T_{N-1}^2...T_1$. Also we have 
$$
[D_1^{(l)},X_2]=\sum_{r=1}^l X_1^{-1}(Y_1-Z_1)...(Y_1-Z_{r-1})[Y_1,X_2](Y_1-Z_{r+1})...(Y_1-Z_l)=
$$
$$
\sum_{r=1}^l X_1^{-1}(Y_1-Z_1)...(Y_1-Z_{r-1})(T_1^{-2}-1)X_2Y_1(Y_1-Z_{r+1})...(Y_1-Z_l)=
$$
$$
(\bold t^{-1}-\bold t)\sum_{r=1}^l X_1^{-1}(Y_1-Z_1)...(Y_1-Z_{r-1})T_1^{-1}X_2Y_1(Y_1-Z_{r+1})...(Y_1-Z_l)=
$$
$$
(\bold t^{-1}-\bold t)\sum_{r=1}^l X_1^{-1}(Y_1-Z_1)...(Y_1-Z_{r-1})X_1T_1Y_1(Y_1-Z_{r+1})...(Y_1-Z_l)=
$$
$$
=(\bold t^{-1}-\bold t)\sum_{r=1}^l (qJ_NY_1-Z_1)...(qJ_NY_1-Z_{r-1})T_1Y_1(Y_1-Z_{r+1})...(Y_1-Z_l).
$$
Therefore, we have 
\begin{multline}\label{cd10}  
[D_1^{(l)},X_2]=\\
(\bold t^{-1}-\bold t)\sum_{r=1}^l (qJ_NY_1-Z_1)...(qJ_NY_1-Z_{r-1})Y_2T_1^{-2}(Y_2T_1^{-2}-Z_{r+1})...(Y_2T_1^{-2}-Z_l)T_1.
\end{multline} 
Now the commutation relations between $D_i^{(l)}$ and $X_j$ can be obtained from this by applying the elements 
$T_m$. 

Thus, we have the following proposition. Let $\DAHA_N^{l,+}$ be the 
``unlocalized" version of the cyclotomic DAHA, 
generated by $T_i$, $X_i$, $D_i^{(l)}$, and $Y_i$ (without $Y_i^{-1}$). 

\begin{proposition}\label{cycDAHAbas} Let $M_X$ be a monomial in $X_i$, $M_D$ a monomial in $D_i^{(l)}$, 
$M_Y$ a monomial in $Y_i$ with powers of $Y_i$ being $\le l-1$. Then the monomials of the form $M_XM_YT_sM_D$,
$s\in S_N$ form a basis in $\DAHA_N^{l,+}$. 
\end{proposition} 

\begin{proof} It is easy to see from considering the polynomial representation that these monomials are linearly independent. 
Therefore, it remains to establish the spanning property. We consider 
a filtration with $\deg(T_i)=0$, $\deg(X_i)=\deg(D_i^{(l)})=l$, 
$\deg(Y_i)=2$. We can use the above commutation relations to order any monomial without increasing its degree. 
Then we can reduce the degrees of $Y_i$ below $l$ by using the relation $X_1D_1^{(l)}=(Y_1-Z_1)...(Y_1-Z_l)$
and its conjugates by the $T_i$. Namely, we have the relation 
$$
X_iD_i^{(l)}=(Y_iT_{i-1}^{-1}...T_1^{-2}...T_{i-1}^{-1}-Z_1)...(Y_iT_{i-1}^{-1}...T_1^{-2}...T_{i-1}^{-1}-Z_l),
$$
which can be used to express the monomial $Y_i^l$ as 
$$
Y_i^l=X_iD_i^{(l)}T_{i-1}...T_1^2...T_{i-1}+...
$$
where ... is a linear combination of ordered monomials of degrees $<l$ in $Y_i,T_i$ and monomials of degree $l$ 
involving $Y_i$ in degree $<l$ and some $Y_j$ with $j<i$. For example, for $i=2$, $l=2$ we get 
$$
X_2D_2^{(l)}=(Y_2T_1^{-2}-Z_1)(Y_2T_1^{-2}-Z_2),
$$
so we get 
$$
Y_2T_1^{-2}Y_2=X_2D_2^{(l)}T_1^2+\text{ lower degree terms }.
$$
But 
$$
Y_2T_1^{-2}Y_2=Y_2^2+(\bold t^{-1}-\bold t)Y_2Y_1T_1,
$$
so 
$$
Y_2^2=X_2D_2^{(l)}T_1^2+(\bold t-\bold t^{-1})Y_2Y_1T_1+\text{ lower degree terms }.
$$
This implies the required spanning property. 
\end{proof} 

\begin{theorem}\label{relcycDAHA} (i) The algebra $\DAHA_N^{l,+}$ is generated by 
$T_i$, $X_i$, $D_i:=D_i^{(l)}$, and $Y_i$ with the following defining relations:

(1) relations (R1-R7), (R11), (R12) of DAHA; 

(2) the relations \eqref{cd1}---\eqref{cd10}.

The algebra $\DAHA_N^l$ is defined by the same generators and relations, adding the condition that $Y_i$ are invertible. 
\end{theorem} 

\begin{proof} We have shown that these relations hold. Moreover, it was shown in the proof of Proposition \ref{cycDAHAbas} that 
using these relations, we can reduce every monomial to a linear combination of basis monomials. 
This implies that the relations are defining.  
\end{proof}  

\begin{proposition}\label{invo} There is an involutive automorphism $\phi$ of $\DAHA_N^l$ such that $\phi(X_i)=D_i$, $\phi(D_i)=X_i$, $\phi(T_i)=T_i^{-1}$, $\phi(\bold t)=\bold t^{-1}$, $\phi(q)=q^{-1}$ and 
$$
\phi(Y_i)=qT_i...T_{N-1}^2...T_iY_iT_{i-1}^{-1}...T_1^{-2}...T_{i-1}^{-1}
$$
This automorphism preserves the subalgebra $\DAHA_N^{l,+}$.
\end{proposition} 

\begin{proof} We check that $\phi$ preserves the defining relations. 
Relations (R1-R3) go to themselves. Relations (R4,R5) get exchanged with \eqref{cd5}.
Relations (R6,R7) go to themselves once (R1-R3) are imposed. 
Relations (R11) get exchanged with \eqref{cd4} and (R12) with themselves once (R1-R3,R6,R7) are imposed. 
Relation \eqref{cd1} gets exchanged with \eqref{cd6}, \eqref{cd2} with \eqref{cd7}, \eqref{cd3} with \eqref{cd8}. Relations \eqref{cd9} go to themselves. So it remains to see that relation \eqref{cd10}is preserved. 

Let $K_N=T_2...T_{N-1}^2...T_2$. We have 
$$
[D_2,X_1]=
$$
$$
\sum_{r=1}^l X_2^{-1}(Y_2T_1^{-2}-Z_l)...(Y_2T_1^{-2}-Z_{r+1})[Y_2T_1^{-2},X_1](Y_2T_1^{-2}-Z_{r-1})...(Y_2T_1^{-2}-Z_1)=
$$
$$
\sum_{r=1}^l X_2^{-1}(Y_2T_1^{-2}-Z_l)...(Y_2T_1^{-2}-Z_{r+1})Y_2(T_1^{-2}-1)X_1(Y_2T_1^{-2}-Z_{r-1})...(Y_2T_1^{-2}-Z_1)=
$$
$$
(\bold t^{-1}-\bold t)\sum_{r=1}^l X_2^{-1}(Y_2T_1^{-2}-Z_l)...(Y_2T_1^{-2}-Z_{r+1})Y_2T_1^{-1}X_1(Y_2T_1^{-2}-Z_{r-1})...(Y_2T_1^{-2}-Z_1)=
$$
$$
(\bold t^{-1}-\bold t)\sum_{r=1}^l X_2^{-1}(Y_2T_1^{-2}-Z_l)...(Y_2T_1^{-2}-Z_{r+1})Y_2T_1^{-2}X_2T_1^{-1}(Y_2T_1^{-2}-Z_{r-1})...(Y_2T_1^{-2}-Z_1).
$$
Thus, 
\begin{multline}\label{cd11}
[D_2,X_1]=\\
(\bold t^{-1}-\bold t)\sum_{r=1}^l (qK_NY_2-Z_l)...(qK_NY_2-Z_{r+1})qK_NY_2(Y_1-Z_{r-1})...(Y_1-Z_1)T_1^{-1}.
\end{multline} 
Now it is easy to see that $\phi$ maps \eqref{cd10} to \eqref{cd11} (as $Y_1$ commutes with $K_N$ and $Y_2$, hence with $K_NY_2$).
\end{proof} 

\begin{remark} It is easy to see that as $q\to 1$, the involution $\phi$ degenerates to the involution $\phi$ on $H_{N,{\rm deg}}^l$ constructed in Remark \ref{invoo}. 
\end{remark} 

\begin{corollary}\label{freeness} The algebra $\DAHA_N^{l,+}(Z_1,...,Z_l,q,t)$ is a free module over $\Bbb C[X_1,...,X_N]\otimes \Bbb C[D_1,...,D_N]$ of rank 
$N!\cdot l^N$, where the first factor acts by left multiplication and the second one by right multiplication. 
\end{corollary} 

\begin{proof} This follows immediately from Proposition \ref{cycDAHAbas}.  
\end{proof}

The proof of Proposition \ref{cycDAHAbas} in fact shows that for any $l\ge 0$ 
ordered products of $M_X,M_Y,T_s,M_D$ in {\it any} of the 24 possible orders 
are a spanning set for $\DAHA_N^{l,+}$, and those of them with degrees of $Y_i$ at most $l-1$ are a basis for $l\ge 1$.   
This implies that we also have another basis of this algebra, formed by monomials 
$M_XM_DT_sM_Y$ without restriction on the degree of $Y_i$, but with the restriction that for each $i$ either $X_i$ is missing in $M_X$ or $D_i$ is missing in $M_D$. Indeed, if this restriction is not satisfied, we may use the relation $X_1D_1=f(Y_1)$ and its permutations to lower the number of $X_i$ and $D_i$, and it is easy to see by looking at the polynomial representation that monomials with this restriction are linearly independent. 
Thus we obtain the following proposition.

\begin{proposition}\label{basi2} The elements $M_XM_D$ which miss either $X_i$ or $D_i$ for each $i$ form a basis 
of $\DAHA_{N}^{l,+}$ as a left or right module over the positive part of the affine Hecke algebra $H_N^+$ generated by $T_s,s\in S_N$ and $Y_i$, and a basis of 
$\DAHA_{N}^l$ as a left or right module over the affine Hecke algebra $H_N$ generated by $T_s,s\in S_N$ and $Y_i^{\pm 1}$; in particular, $\DAHA_N^{l,+}$ is a free module over $H_N^+$ and $\DAHA_N^l$ is a free module over $H_N$. 
 \end{proposition} 

Note that the basis of Proposition \ref{basi2} is labeled by $N$-tuples of integers, $(m_1,...,m_N)$. Namely, if $M_XM_D$ contains $X_i^p$ then we set $m_i=p$, and if it contains $D_i^p$ then we set $m_i=-p$.    

Another, geometric proof of Proposition \ref{basi2} (for a geometric version of $\DAHA_N^l$) will be given in Section~\ref{sec5}.

\section{Geometric realization}
\label{sec5}

\subsection{A variety of triples}
\label{triples}
We consider a quiver with a set of vertices $I$ and a set of arrows $\Omega$.
Let $V=\bigoplus_{i\in I}V_i,\ W=\bigoplus_{i\in I}W_i$ be $I$-graded finite
dimensional $\BC$-vector spaces; $d_i:=\dim V_i$.
Given a length $\ell$ sequence $\bi=(i_1,\ldots,i_\ell)\in I^\ell$ and a
length $\ell$ sequence $\ba=(a_1,\ldots,a_\ell)$ of positive integers such that
$\sum_{n : i_n=i}a_n=d_i$ for any $i\in I$, we choose an $I$-graded flag
in $V\colon V=V^0\supset V^1\supset\ldots\supset V^\ell=0$ such that
$V^{n-1}/V^n$ is an $a_n$-dimensional vector space supported at the vertex $i_n$
for any $n=1,\ldots,\ell$.

We set $\CK=\BC((z))\supset\BC[[z]]=\CO$. We consider the following flag
of $I$-graded lattices in
$V_\CK=V\otimes\CK\colon \ldots\supset L_{-1}\supset L_0\supset L_1\supset\ldots$,
where $L_{r+\ell}=zL_r$ for any
$r\in\BZ;\ L_0=V_\CO;\ L_n/L_\ell=V^n\subset V=L_0/L_\ell$ for any
$n=1,\ldots,\ell$. Let $GL(V):=\prod_{i\in I}GL(V_i)$, and let
$\CP\subset GL(V)_\CO\subset GL(V)_\CK$ be the stabilizer of the flag $L_\bullet$.
Then $GL(V)_\CK/\CP$ is the set of points of the ind-projective moduli space
$\Fl$ of flags of $I$-graded lattices
$\ldots\supset M_{-1}\supset M_0\supset M_1\supset\ldots$ in $V_\CK$ such that
$M_{r+\ell}=zM_r$ for any $r\in\BZ$, and $M_{n-1}/M_n$ is an $a_n$-dimensional
vector space supported at the vertex $i_n$ for any $n=1,\ldots,\ell$.
This is a partial affine flag variety of the reductive group $GL(V)$.
Note that the set of connected components $\pi_0(\Fl)$ is naturally identified
with $\BZ^I$ (the virtual graded dimension of $M_0$).

Let $\CR$ be the moduli space of the following data (cf.~\cite[Section~1]{l91}):

(a) $M_\bullet\in\Fl$;

(b) a $\CK$-linear homomorphism $p_i\colon W_{i,\CK}\to V_{i,\CK}$ for any
$i\in I$;

(c) a $\CK$-linear homomorphism $b_{i\to j}\colon V_{i,\CK}\to V_{j,\CK}$ for
any $i\to j\in\Omega$; such that

(1) $b:=\sum_{i\to j\in\Omega}b_{i\to j}$ takes $L_r$ to $L_{r+1}$ and $M_r$ to
$M_{r+1}$ for any $r\in\BZ$;

(2) $p:=\sum_{i\in I} p_i$ takes $W_\CO$ to $L_0\cap M_0$.

Note that when $\ell=1$, we have $\CP=GL(V)_\CO$,
and $\CR$ is nothing but the variety
of triples $\CR_{GL(V),\bN}$ associated in~\cite[2(i)]{BFN2} to a $GL(V)$-module
$\bN=\bigoplus_{i\to j\in\Omega}\Hom(V_i,V_j)\oplus\bigoplus_{i\in I}\Hom(W_i,V_i)$.
The definition of equivariant Borel-Moore homology
$H^{GL(V)_\CO\rtimes\BC^\times}_\bullet(\CR_{GL(V),\bN})$
(respectively, equivariant $K$-theory $K^{GL(V)_\CO\rtimes\BC^\times}(\CR_{GL(V),\bN})$)
and the construction of convolution product on it
in~\cite[Sections~2,3]{BFN2} work without any changes in our situation, and
produce the convolution algebras
$H^{\CP\rtimes\BC^\times}_\bullet(\CR)$ and $K^{\CP\rtimes\BC^\times}(\CR)$.
Moreover, if we choose a Cartan torus $T(W_i)\subset GL(W_i)$ and set
$T(W):=\prod_{i\in I}T(W_i)\subset\prod_{i\in I}GL(W_i)=:GL(W)$ (a flavor symmetry
group), we obtain the convolution algebras
$H^{\BC^\times\times T(W)_\CO\times\CP\rtimes\BC^\times}_\bullet(\CR)$ and
$K^{\BC^\times\times T(W)_\CO\times\CP\rtimes\BC^\times}(\CR)$. Here the first factor
$\BC^\times$ acts by dilations in the fibers of the projection
\begin{equation}
\label{varpi}
\varpi\colon \CR\to\Fl.
\end{equation}

The following easy result will be important in the future.
\begin{lemma}\label{free}
The algebra $H^{\BC^\times\times T(W)_\CO\times\CP\rtimes\BC^\times}_\bullet(\CR)$ is free as a module over the equivariant point cohomology
$H_{\BC^\times\times T(W)_\CO\times\CP\rtimes\BC^\times}^\bullet(pt)$. Similarly, the algebra $K^{\BC^\times\times T(W)_\CO\times\CP\rtimes\BC^\times}(\CR)$ is free as a module over $K_{\BC^\times\times T(W)_\CO\times\CP\rtimes\BC^\times}(pt)$.
\end{lemma}
\begin{proof}
It is enough to show that $\CR$ has an algebraic cell decomposition, which is invariant under the maximal torus of the group $\BC^\times\times T(W)_\CO\times\CP\rtimes\BC^\times$. For this we need to choose an Iwahori subgroup $\mathcal I$ of $GL(V)_{\CK}$ which is contained in $\CP$. Then the $\mathcal I$-orbits on
$\Fl$ are affine spaces; hence their preimages in $\CR$ are (infinite-dimensional) affine spaces as well, which are clearly invariant under $\BC^\times\times T(W)_\CO\times\CI\rtimes\BC^\times$, hence under some maximal torus of $\BC^\times\times T(W)_\CO\times\CP\rtimes\BC^\times$.
\end{proof}
In case $W=W'\oplus W''$, we denote the variety of triples corresponding to
$W'$ (resp.\ $W''$) by $\CR'$ (resp.\ $\CR''$), and we have an evident
closed embedding $\bz\colon \CR'\hookrightarrow\CR$.
The argument of~\cite[Lemma~5.11]{BFN2} goes through word for word in our
situation and proves that
\begin{equation}
\label{zH}
\bz^*\colon
H^{\BC^\times\times T(W)_\CO\times\CP\rtimes\BC^\times}_\bullet(\CR)\hookrightarrow
H^{\BC^\times\times T(W')_\CO\times\CP\rtimes\BC^\times}_\bullet(\CR')\otimes
H^\bullet_{T(W'')}(pt)
\end{equation} and
\begin{equation}
\label{zK}
\bz^*\colon K^{\BC^\times\times T(W)_\CO\times\CP\rtimes\BC^\times}(\CR)\hookrightarrow
K^{\BC^\times\times T(W')_\CO\times\CP\rtimes\BC^\times}(\CR')\otimes K_{T(W'')}(pt)
\end{equation}
are the convolution algebra homomorphisms.

\subsection{Jordan quiver}
\label{Jordan}
In what follows we consider a special case of the construction
of~Section~\ref{triples} where $I$ consists of a single vertex, and
$\Omega$ consists of a single loop, $W=\BC^l$ with a basis $e_1,\ldots,e_l$
(hence the diagonal torus $T(W)\subset GL(W)$), and
$V=\BC^N$ with a basis $v_1,\ldots,v_N$. Moreover, $\ell=N,\ a_n=1$ for any
$n=1,\ldots,N$, and $V^\bullet$ is a complete flag
$V^n:=\BC v_1\oplus\ldots\oplus\BC v_{N-n}$.
We fix a flag of lattices $L_i\subset V((z)),\ i\in \Bbb Z$, such that 
$L_0=V[[z]]$, $L_{j+N}=zL_j$, and 
$L_j=zV[[z]]\oplus \Bbb Cv_1\oplus...\oplus \Bbb Cv_{N-j}$, $j=0,...,N-1$.

The space of triples of~Section~\ref{triples} is the moduli space of the 
following data: 

(a) a sequence of $\Bbb C[[z]]$-lattices $M_i\subset V((z)),\ i\in \Bbb Z$,  
such that $M_i\supsetneq M_{i+1}$ and $M_{j+N}=zM_j$; 

(b) a $\Bbb C((z))$-linear map $b\colon V((z))\to V((z))$; and

(c) a $\Bbb C((z))$-linear map $p\colon W((z))\to V((z))$; 

such that 

(1) $b$ strongly preserves $L$ and $M$, i.e., $bL_i\subset L_{i+1}$ and 
$bM_i\subset M_{i+1}$; and 

(2) $pW[[z]]\subset L_0\cap M_0$.

The basis of $V$ gives rise to the diagonal torus $T(V)\subset GL(V)$,
and we denote by $y_1,\ldots,y_N$ the generators of $H^\bullet_{T(V)}(pt)$.
Also, we denote by $\sz_1,\ldots,\sz_l$ (resp.\ $\hbar,-k$)
the generators of $H^\bullet_{T(W)}(pt)$ (resp.\ $H^\bullet_{\BC^\times}(pt)$ for
the loop rotation $\BC^\times$, $H^\bullet_{\BC^\times}(pt)$ for the dilation
$\BC^\times$).
We shall denote the corresponding generators of $K_{T(V)}(pt),\ K_{T(W)}(pt),\ K_{\BC^\times}(pt),\
K_{\BC^\times}(pt)$ by $Y_1,\ldots,Y_N;\ \sZ_1,\ldots,\sZ_l;\ q;t$.
The algebra
$H^{\BC^\times\times T(W)_\CO\times\CP\rtimes\BC^\times}_\bullet(\CR)$
(resp.\ $K^{\BC^\times\times T(W)_\CO\times\CP\rtimes\BC^\times}(\CR)$)
will be denoted $\CHH^l_{N,{\rm deg}}$ (resp.\ $\CHH^l_N$).
We shall also denote $\CHH^0_{N,{\rm deg}}$ by $\CHH_{N,{\rm deg}}$, and $\CHH^0_N$ by $\CHH_N$.
According to~(\ref{zH}) and~(\ref{zK}) we have algebra embeddings
$\bz^*\colon \CHH^l_{N,{\rm deg}}\hookrightarrow\CHH_{N,{\rm deg}}[\sz_1,\ldots,\sz_l]$ and
$\bz^*\colon \CHH^l_N\hookrightarrow\CHH_N[\sZ_1^{\pm1},\ldots,\sZ_l^{\pm1}]$.

Note that for $W=0$ the variety of triples $\CR$ is nothing but the affine
Steinberg variety of $GL(V)$, and $\Fl$ is nothing but the affine flag
variety of $GL(V)$. It is well known that $\DAHA_{N,{\rm deg}}\simeq\CHH_{N,{\rm deg}}$
(see e.g.~\cite{OY}), and one can check that
$\DAHA_N\simeq\CHH_N$, cf.~\cite{VV} (see~\cite[Remark~3.9(2)]{BFN2}).
More precisely, for $n=0,\ldots,N-1$, we denote by $\BP^1_n\subset\Fl$
the projective line formed by all the flags $M_\bullet$ of lattices
(see~Section~\ref{triples}) such that $M_m=L_m$ for $m\ne n\pmod{N}$.
Each $\BP^1_n$ contains the base point $L_\bullet\in\Fl$, and we denote
by $\BA^1_n\subset\BP^1_n$ the complement. The restriction of the projection
$\varpi\colon\CR\to\Fl$ to $\BA^1_n$ is a (profinite dimensional) vector
bundle, and the closure of $\varpi^{-1}(\BA^1_n)$ in $\CR$ is still a
vector bundle over $\BP^1_n$, to be denoted by $\widetilde\BP{}^1_n$.
We define $s_n\in\CH_N$ as $1+[\widetilde\BP{}^1_n]$ for $n=0,\ldots,N-1$
(the fundamental cycle of $\widetilde\BP{}^1_n$).
Also, let $\ol\pi$ (resp.\ $\ol\pi{}^{-1}$) be
a point-orbit of the Iwahori group $\CP$ in $\Fl$ consisting of the flag
$M_\bullet$ such that $M_n=L_{n+1}$ (resp.\ $M_n=L_{n-1}$) for any $n\in\BZ$.
Finally, we define $\pi^{\pm1}\in\CHH_{N,{\rm deg}}$ as $[\varpi^{-1}(\ol\pi{}^{\pm1})]$.
Now the desired isomorphism $\DAHA_{N,{\rm deg}}\iso\CHH_{N,{\rm deg}}$ takes the generators
of $\DAHA_{N,{\rm deg}}$  to the same named elements of $\CHH_{N,{\rm deg}}$, except $\hbar\mapsto -\hbar$.

Furthermore, we define $T_n\in\CHH_N$ as $-1-[\CO_{\widetilde\BP{}^1_n}(-2)]$
for $n=0,\ldots,N-1$, and we define $\pi^{\pm1}\in\CHH_N$ as
$[\CO_{\varpi^{-1}(\ol\pi{}^{\pm1})}]$. The desired isomorphism $\DAHA_N\iso\CHH_N$
takes the generators of $\DAHA_N$ to the same named elements of $\CHH_N$, except $q\mapsto q^{-1}$.

For arbitrary $W=\BC^l$, in order to distinguish from the $l=0$ case,
we will denote by $\varpi_l\colon \CR^l\to\Fl$ the corresponding projection.
Let us set
$$
\CHH_{N,{\rm deg}}^l \ni\pi_-=
[\varpi_l^{-1}(\ol\pi{}^{-1})];\quad \CHH_N^l\ni\pi_-=
(-q^{-1}t^{-1})^l\sZ_1\cdots\sZ_l[\CO_{\varpi_l^{-1}(\ol\pi{}^{-1})}].
$$
Then we have
\begin{equation}
\label{piCH}
\bz^*\pi_-=\prod_{m=1}^l(y_N-\sz_m-k)\cdot\pi^{-1}=
\pi^{-1}\prod_{m=1}^l(y_1-\sz_m-k+\hbar)\in\CHH_{N,{\rm deg}},
\end{equation}
and
\begin{multline}
\label{piCHH}
\bz^*\pi_-=(-q^{-1}t^{-1})^l\sZ_1\cdots\sZ_l\prod_{m=1}^l(1-Y_N\sZ^{-1}_mt)\cdot
\pi^{-1}=\\
=\pi^{-1}\prod_{m=1}^l(Y_1-\sZ_mq^{-1}t^{-1})\in\CHH_N.
\end{multline}

Indeed, $\bz^*\pi_-$ is obtained from $\pi^{-1}$ by multiplication with the Euler class
of the finite dimensional quotient space $\CT^l_{\ol\pi{}^{-1}}/\CR^l_{\ol\pi{}^{-1}}$, where
$\CR^l_{\ol\pi{}^{-1}}$ is the fiber of $\CR^l$ at the point $\ol\pi{}^{-1}$, and
$\CT^l_{\ol\pi{}^{-1}}$ is defined similarly, but the condition~(2) above: $pW[[z]]\subset
L_0\cap M_0=L_0\cap L_{-1}=L_0$ is relaxed to the condition $pW[[z]]\subset M_0=L_{-1}$,
cf.~\cite[4(vi)]{BFN2}.

We preserve the name $\pi$ for $[\varpi_l^{-1}(\ol\pi)]\in\CHH^l_{N,{\rm deg}}$
(resp.\ $[\CO_{\varpi_l^{-1}(\ol\pi{})}]\in\CHH^l_N$) since $\bz^*$ takes this
$\pi$ to the one in $\CHH_{N,{\rm deg}}$ (resp.\ in $\CHH_N$). For the same reason we
preserve the names $s_n$ (resp.\ $T_n$), $n=1,\ldots,N-1$,
for the corresponding elements of $\CHH_N^{l,{\rm deg}}$ (resp.\ $\CHH_N^l$).
Finally, we set $z_m=\sz_m+k-\hbar$, and $Z_m=\sZ_mq^{-1}t^{-1}$.
So the following diagrams commute:
\begin{equation}
\label{commH}
\begin{CD}
\DAHA_{N,{\rm deg}}^l @>>> \CHH_{N,{\rm deg}}^l\\
@VVV @V{\bz^*}VV\\
\DAHA_{N,\rm deg}[z_1,...,z_l] @>\sim>> \CHH_{N,{\rm deg}}[z_1,...,z_l]
\end{CD},
\end{equation}
\begin{equation}
\label{commHH}
\begin{CD}
\DAHA_N^l @>>> \CHH_N^l\\
@VVV @V{\bz^*}VV\\
\DAHA_N[Z_1^{\pm 1},...,Z_l^{\pm 1}] @>\sim>> \CHH_N[Z_1^{\pm 1},...,Z_l^{\pm 1}]
\end{CD}.
\end{equation}

Note that in diagram \eqref{commH} we have bigradings defined by
$\deg(X_i)=(1,0)$, $\deg(y_i)=\deg(z,\hbar,k)=(0,1)$,
which are preserved by all the maps.

The following theorem is independently obtained in \cite[Lemma~4.2]{W}. 

\begin{theorem}\label{geomiso} The map $\xi: \DAHA_{N,{\rm deg}}^l \to \CHH_{N,{\rm deg}}^l$ in diagram \eqref{commH} is an isomorphism of bigraded algebras.
\end{theorem}

\begin{proof} We may assume that $l>0$. Recall from Lemma \ref{free} that $\CHH_{N,{\rm deg}}^l$ is a free $\Bbb C[z_1,...,z_l,\hbar,k]$-module.
Thus it suffices to show that $\xi_0$ is an isomorphism, where $\xi_0$ is the specialization of $\xi$ at $\hbar=k=0,z=0$. It is clear that $\xi_0$ is injective, hence so are $\xi$ and all its specializations $\xi_{z,\hbar,k}$. Since the bigraded components $\DAHA_{N,{\rm deg}}^l[r,s]$ are finite dimensional (as $l>0$), it suffices to show that the specialization $\xi_{z,\hbar,k}$ is an isomorphism for {\it Weil generic} $z,\hbar,k$.

We will now use the following easy lemma about unital rings.

\begin{lemma}\label{algle} Let $B$
be a unital ring and $A\subset B$ a unital subring.
If $\bold e\in A$ is an idempotent such that $A\bold e A=A$ and $\bold e A\bold e=\bold e B\bold e$, then $A=B$.
 \end{lemma}

\begin{proof} Since $A\bold e A=A$, we have $A\bold e B=A\bold e AB=AB=B$. Similarly $B\bold e A=B$. Thus $A\bold e B\bold e A=B\bold e A=B$.
But since $\bold e A\bold e=\bold e B\bold e$ and $A\bold e A=A$, we have $A\bold e B\bold e A=A\bold e A\bold e A=A\bold e A=A$. Thus $A=B$.
\end{proof}

Now let $\bold e$ be the symmetrizer of $S_N$, and let us apply Lemma~\ref{algle} to $A=\DAHA_{N,{\rm deg}}^l(z,1,k)$ and $B=\CHH_{N,{\rm deg}}^l(z,-1,k)$.
Using Theorem~\ref{main1}, we see that the condition $A\bold e A=A$ is satisfied for generic $z,k$, namely when the corresponding parameters of
the cyclotomic rational Cherednik algebras are not aspherical (see e.g. \cite{BE}, Subsection 4.1). In fact,
it suffices to consider the case $z=0,k=0$, when this is easy, since $\DAHA_{N,{\rm deg}}^l(0,1,0)=S_n\ltimes ({\mathcal D}(\Bbb C)^{\Bbb Z/l\Bbb Z})^{\otimes N}$
is a simple algebra (by \cite[Theorem 2.3]{Mo}). Also, 
by~\cite[Proposition~3.24]{BFN3}, the theorem
holds for spherical subalgebras, i.e., $\bold e\xi_0 \bold e$ is an isomorphism. Hence so are $\bold e\xi\bold e$ and all its specializations $\bold e\xi_{z,\hbar,k}\bold e$, which yields the condition $\bold e A\bold e=\bold e B\bold e$ for all $z,\hbar,k$. Thus Lemma \ref{algle} applies and Theorem \ref{geomiso} follows.
\end{proof}

We expect that an analog of Theorem \ref{geomiso} also holds in the K-theoretic setting, i.e., for diagram \eqref{commHH}.
Let us prove a formal version of this statement. Let $q=e^{\varepsilon \hbar}$, $t=e^{-\varepsilon  k}$, $Z_i=e^{\varepsilon z_i}$.
Let $\CHH_N^{l,{\rm formal}}$ be the corresponding formal completion of $\CHH_N^l$
(obtained by viewing equivariant K-theory as a formal deformation of equivariant Borel-Moore
homology\footnote{Indeed, the Borel-Moore homology is the associated graded of the
  $\gamma$-filtration on the $K$-theory. On the equivariant $K$-theory of the point, this
  is the filtration by the order of vanishing at the neutral element. See~\cite{AS} for the
topological situation, and~\cite{To} for the comparison with the algebraic situation.}).
Diagram \eqref{commHH} furnishes a map $\widehat{\xi}: \DAHA_N^{l,{\rm formal}}\to \CHH_N^{l,{\rm formal}}$.

\begin{corollary} The map $\widehat{\xi}$ is an isomorphism.
\end{corollary}

\begin{proof} By Lemma \ref{free},
$\CHH_N^{l,{\rm formal}}$ is a flat formal deformation of $\CHH_{N,{\rm deg}}^l$ over $\Bbb C[[\varepsilon]]$.
Finally, $\widehat\xi|_{\varepsilon=0}=\xi$. This implies the corollary.
\end{proof}

In particular, this gives another, geometric proof of the facts that
$\DAHA_{N,{\rm deg}}^l$ is a free bigraded $\Bbb C[z_1,...,z_l,\hbar,k]$-module,
and the algebra $\DAHA_N^{l,{\rm formal}}(z,\hbar,k)$ is its flat formal deformation.

We also have

\begin{theorem}\label{spherr} Let $\bold e$ be the symmetrizer of the finite Hecke algebra generated by $T_i$, $1\le i\le N-1$. Then
the natural map
$$
\bold e \xi_{Z,q,t}\bold e: \bold e\DAHA_N^l(Z,q,t)\bold e \to \bold e\CHH_N^l(Z,q,t)\bold e
$$
is an isomorphism when $q=t=1$ and $Z_i=1$ for all $i$.
\end{theorem}

\begin{proof}
Same as the proof of~\cite[Proposition~3.24]{BFN3}.
\end{proof}

\begin{theorem} Let $\widehat{\DAHA_N^l}$, $\widehat{\CHH_N^l}$ be the completions of $\DAHA_N^l,\CHH_N^l$ at $q=t=Z_i=1$ (as modules over $\Bbb C[Z_1^{\pm 1},...,Z_l^{\pm 1},
q^{\pm 1},\bold t^{\pm 1}]$). Then the map
$$
\widehat{\xi}: \widehat{\DAHA_N^l} \to \widehat{\CHH_N^l}
$$
is an isomorphism.
\end{theorem}

\begin{proof} The proof is analogous to the proof of Theorem \ref{geomiso}, using Theorem \ref{spherr}.
Namely, the identity $A\bold e A=A$
for $A=\DAHA_N^l(Z,q,t)$ is established for generic $q$ and $Z=1$, $t=1$
using the fact that in this case $A$ is a simple algebra by
\cite[Theorem 2.3]{Mo}.
\end{proof}

\section{Cyclotomic DAHA and multiplicative quiver and bow varieties}
\label{sec4}

\subsection{Multiplicative quiver varieties} 
\label{quiver}
Let $t\in \Bbb C^*$ be not a root of unity, and $Z_1,...,Z_l\in \Bbb C^*$ be such that $Z_i/Z_j$ is not an integer power of $t$ for $i\ne j$.
Let $Q_l$ be the cyclic quiver $\hat A_{l-1}$ with vertices $1,...,l$ and an additional ``Calogero-Moser vertex'' $0$ attached to the vertex $1$.
Let ${\mathcal M}_N^l(Z,t)$ be the {\it multiplicative quiver variety}
for $Q_l$ with dimension vector $d_1=...=d_l=N$ and $d_0=1$, see \cite{CBS}.
Namely, given complex vector spaces $V_i$, $i=1,...,l$, with $\dim V_i=N$,
${\mathcal M}_N^l(Z,t)$ is the variety of collections
of linear maps $\X_i: V_{i+1}\to V_i$ and $\D_i: V_i\to V_{i+1}$
(where addition is mod $l$) satisfying the equations
\begin{equation}\label{rela1}
Z_{i}(1+\X_{i} \D_{i})=Z_{i-1}(1+\D_{i-1}\X_{i-1}), 2\le i\le l
\end{equation}
and
\begin{equation}\label{rela2}
Z_1(1+\X_1 \D_1)T=Z_l(1+\D_l\X_l),
\end{equation}
where $T: V_1\to V_1$ is an operator conjugate to ${\rm diag}(t^{-1},...,t^{-1},t^{n-1})$,
modulo simultaneous conjugation (i.e., the corresponding categorical quotient).

\begin{example} Let $l=1$. Then there is no dependence on $Z_1$, and
${\mathcal M}_N^1(t)$ is the variety of pairs $(\X,\D)$ of
$N$ by $N$ matrices such that
$$
(1+\X \D)T=(1+\D \X), 
$$
where $1+\X \D$ is invertible, and $T$ is as above, modulo simultaneous conjugation.
\end{example}

Let $\X:=\X_1...\X_l$, $\D:=\D_l...\D_1$, $\Y:=Z_1(1+\X_1\D_1)$.
Consider the operators $L_+:=Z_1...Z_l \X \D$,
$L_-:=Z_1...Z_l \D \X$.

\begin{lemma}\label{prodfor}
We have
$$
L_+=(\Y-Z_1)...(\Y-Z_l),
$$
$$
L_-=(\Y T-Z_1)...(\Y T-Z_l).
$$
\end{lemma}

\begin{proof} We prove the formula for $L_+$;
the formula for $L_-$ is proved in a similar way.

It suffices to prove by induction in $r$ that
$$
Z_1...Z_r \X_1...\X_r\D_r...\D_1=(\Y-Z_1)...(\Y-Z_r).
$$
The base $r=0$ is obvious. For $r>0$, we have, using the induction assumption:
$$
Z_1...Z_r \X_1...\X_r\D_r...\D_1=\X_1...\X_{r-1}(Z_r\X_r\D_r)\X_{r-1}^{-1}...\X_1^{-1}(\Y-Z_1)...(\Y-Z_{r-1}).
$$
But
$$
\X_1...\X_{r-1}(Z_r\X_r\D_r)\X_{r-1}^{-1}...\X_1^{-1}=$$
$$\X_1...\X_{r-1}(Z_{r-1}\D_{r-1}\X_{r-1}+Z_{r-1}-Z_r)\X_{r-1}^{-1}...\X_1^{-1}=
$$
$$
=Z_{r-1}-Z_r+\X_1...\X_{r-2}(Z_{r-1}\X_{r-1}\D_{r-1})\X_{r-2}^{-1}...\X_1^{-1}.
$$
Thus,
$$
\X_1...\X_{r-1}(Z_r\X_r\D_r)\X_{r-1}^{-1}...\X_1^{-1}=Z_1-Z_r+Z_1\X_1\D_1=\Y-Z_r.
$$
This implies the induction step.
\end{proof}

Thus we have
\begin{equation}\label{eqnsvar1}
\X \D=(Z_1^{-1}\Y-1)...(Z_l^{-1}\Y-1),\ \D \X=(Z_1^{-1}\Y T-1)...(Z_l^{-1}\Y T-1),
\end{equation}
\begin{equation}\label{eqnsvar2}
\Y \X =\X \Y T, \Y T \D=\D \Y.
\end{equation}

\begin{lemma}\label{invo1} 
We have an isomorphism 
$$
\Phi: {\mathcal M}_N^l(Z_1,...,Z_l,t)\to {\mathcal M}_N^l(Z_l,...,Z_1,t^{-1})
$$ 
given by $\Phi(\X_i)=\D_{l+1-i}$, $\Phi(\D_i)=\X_{l+1-i}$, $\Phi(Z_i)=Z_{l+1-i}$. 
\end{lemma} 

\begin{proof} It is easy to check that the relations defining these varieties are matched by $\Phi$.  
\end{proof}

\subsection{Quadruple varieties and their connection to multiplicative quiver varieties} 
Let ${\bold M}_N^l(Z,t)$ be the variety of quadruples of matrices $(\X,\D,\Y,T)$
satisfying \eqref{eqnsvar1},\eqref{eqnsvar2} such that $\Y$ is invertible, modulo simultaneous conjugation (i.e., the categorical quotient).
We have a natural map 
$$
\psi: {\mathcal M}_N^l(Z,t)\to {\bold M}_N^l(Z,t)
$$ 
sending
$\X_i,\D_i,1\le i\le l$ to $(\X,\D,\Y,T)$.

\begin{proposition}\label{irre} Any collection $(\X,\D,\Y,T)\in {\bold M}_N^l(Z,t)$ acts irreducibly on $\Bbb C^N$.
\end{proposition}

\begin{proof} Assume the contrary. Then there is an invariant subspace or quotient $V$ for $(\X,\D,\Y,T)$
of dimension $1\le n\le N-1$ on which $T$ acts by $t^{-1}$. So by Lemma
\ref{prodfor}, on $V$ we have
$$
\X \D=(Z_1^{-1}\Y-1)...(Z_l^{-1}\Y-1),\ \D \X=((tZ_1)^{-1}\Y -1)...((tZ_l)^{-1}\Y -1),
$$
$$
t\Y \X =\X \Y ,
$$
and
\begin{equation}\label{dy}
\Y \D=t\D \Y.
\end{equation}
Equation \eqref{dy} implies that $\D$ cannot be invertible (otherwise taking determinants of both sides gives a contradiction).
Thus there is a nonzero vector $v\in V$ such that $\D v=0$. Since (again by \eqref{dy}) ${\rm Ker}\D$ is $\Y$-stable,
we can choose $v$ so that $\Y v=\lambda v$ for some $\lambda\ne 0$. Since $\X \D v=0$, we have $\lambda=Z_j$ for some $j$.

We may assume that $V$ is irreducible for the action of $(\X,\D,\Y,T)$.
Then $V$ has a basis $v,\X v,...,\X^{n-1}v$ with $\Y \X^iv=t^{-i}\lambda \X^i v$,
and
$$
\D \X^i v=(Z_1^{-1}\lambda t^{-i}-1)...(Z_l^{-1}\lambda t^{-i} -1)\X^i v.
$$
In particular, since $\X^n v=0$, we must have $\lambda=t^nZ_m$ for some $m$. Thus, we have $t^nZ_m=Z_j$, which contradicts our assumption on the $Z_i$.
\end{proof}

By Schur's lemma, Proposition \ref{irre} implies that any operator $A$ commuting with $\X,\D,\Y,T$ has to be a scalar, so ${\bold M}_N^l(Z,t)$ is, in fact, the ordinary quotient
of the set of quadruples $(\X,\D,\Y,T)$ satisfying \eqref{eqnsvar1},\eqref{eqnsvar2} by the free action of $PGL_N(\Bbb C)$.

\begin{corollary}\label{scalar} Every collection of endomorphisms $A_i: V_i\to V_i$ which commute with
$(\X_1,...,\X_l,\D_1,...,\D_l)$ satisfying \eqref{rela1},\eqref{rela2} is a scalar
(the same at all vertices).
\end{corollary}

\begin{proof} Suppose we have such a collection.
Then $A_1$ commutes with $(\X,\D,\Y,T)$, so by Proposition \ref{irre}, it has to be a scalar.
So by shifting $A_i$ by the same scalar we may assume that $A_1=0$.
Our job is to show that $A_i=0$ for all $i$. Assume the contrary, i.e.
that $A_s\ne 0$ for some $s$.

Let $V_i'={\rm Im}A_i$. Then $V_1'=0$, so
there exist $1\le i<j\le l$ such that $V_{i+1}'=V_{i+2}'=...=V_j'\ne 0$,
while $V_i'=V_{j+1}'=0$. The collection of nonzero spaces $V_s'$, $i<s\le j$
is invariant under the operators $\X_s,\D_s$, and these operators satisfy on $V_s'$ the equations
$$
Z_{i+1}(1+\X_{i+1}\D_{i+1})=Z_i,\ Z_{i+2}(1+\X_{i+2}\D_{i+2})=Z_{i+1}(1+\D_{i+1}\X_{i+1}),...,
$$
$$
Z_j=Z_{j-1}(1+\D_{j-1}\X_{j-1}).
$$
(if $j=i+1$, then we get just one equation $Z_i=Z_{i+1}$). If $j\ge i+2$, this implies that any nonzero vector $v\in V_{i+1}'$ is an eigenvector of
the operator $Z_{i+1}(1+\X_{i+1}\D_{i+1})$ with eigenvalue $Z_i$. Since $Z_{i+1}\ne Z_i$, this implies that $\D_{i+1}v$ is an eigenvector of the operator
$$
Z_{i+1}(1+\D_{i+1}\X_{i+1})=Z_{i+2}(1+\X_{i+2}\D_{i+2})
$$
with eigenvalue $Z_i$. Continuing like this, we find that
$Z_{j-1}(1+\D_{j-1}\X_{j-1})$ has eigenvector $\D_{j-1}...\D_{i+1}v$ with eigenvalue $Z_i$, hence $Z_i=Z_j$. This is a contradiction, which
proves the corollary.
\end{proof}

Corollary \ref{scalar} implies that ${\mathcal M}_N^l(Z,t)$ is also an ordinary quotient.
A similar argument shows that equations \eqref{rela1},\eqref{rela2}
define a smooth complete intersection, i.e., the multiplicative quiver variety ${\mathcal M}_N^l(Z,t)$
is smooth (in fact, both of these statements follow from the results of \cite{CBS}).

Let ${\bold M}_N^l(Z,t)^\circ$ be the open subset of ${\bold M}_N^l(Z,t)$
where $\X$ is invertible. On this set, $\D$ is redundant (i.e., expresses in terms of $\X$ and $\Y$), 
and the only equation we are left with is
$$
\Y \X =\X \Y T.
$$
Thus, ${\bold M}_N^l(t)^\circ={\bold M}_N^l(Z,t)^\circ$ is independent of the $Z_i$
and is the multiplicative Calogero-Moser space considered in \cite{Ob1} (the phase space of the
Ruijsenaars integrable system). In particular, as explained in \cite{Ob1}, ${\bold M}_N^l(t)^\circ$ is
smooth and connected.

Now let us study the properties of the map $\psi$. Note that if $l=1$, $\psi$ is tautologically an
isomorphism. Moreover, we claim that $\psi$ is an isomorphism
$\psi^{-1}({\bold M}_N^l(t)^\circ)\to {\bold M}_N^l(t)^\circ$. Indeed, if
$\X$ is invertible then we can set $V_i=\Bbb C^N$ and $\X_i=1$ for $i=1,...,l-1$, while $\X_l=\X$.
Then we get that $Z_i(1+\D_i)=\Y =Z_l(1+\D_l \X_l)T^{-1}$ for $i=2,...,l$.
Thus $\D_i=Z_i^{-1}\Y -1$ for $i=1,...,l-1$, and $\D_l=\X^{-1}(Z_1^{-1}\Y-1)$.

\begin{proposition}\label{cloem} $\psi$ is a closed embedding.
\end{proposition}

\begin{proof} Our job is to show that the map
$$
\psi^*: {\mathcal O}({\bold M}_N^l(Z,t))\to {\mathcal O}({\mathcal M}_N^l(Z,t))
$$
is surjective. By the Fundamental Theorem of invariant theory, ${\mathcal O}({\bold M}_N^l(Z,t))$ is generated by
the elements ${\rm Tr}(w)$, where $w$ are words in $\X,\D,\Y^{\pm 1}$, and $T^{\pm 1}$. So it suffices to show that
${\mathcal O}({\mathcal M}_N^l(Z,t))$ is generated by the elements $\psi^*{\rm Tr}(w)$.
We know that this algebra is generated by expressions $\Tr(u)$, where $u$ is any cyclic word
(i.e., closed path) consisting of $\D_i$, $\X_i$, $(1+\X_i\D_i)^{\pm 1}$ and $T^{\pm 1}$,
so it suffices to show that any such cyclic word can be expressed
as a cyclic word $w$ in $\X,\D,\Y^{\pm 1}$, and $T^{\pm 1}$.

Let $\Lambda$ be the deformed multiplicative preprojective algebra of the Calogero-Moser quiver $Q_l$ defined in \cite{CBS}; thus,
${\mathcal O}({\mathcal M}_N^l(Z,t))$ is the representation variety
of $\Lambda$ for the dimension vector $(d_i)$. Let $\bold e_1\in \Lambda$ be the
idempotent of the vertex 1.

\begin{lemma}\label{lel} One has $\Lambda=\Lambda \bold e_1 \Lambda$.
\end{lemma}

\begin{proof} Assume the contrary. Then we have a nonzero $\Lambda$-module $V$ such that $V_1=0$ (namely, any nonzero
module over $\Lambda/\Lambda \bold e_1\Lambda$), and the argument in the proof of Corollary
\ref{scalar} gives a contradiction with the condition $Z_i\ne Z_j$.
\end{proof}

Lemma \ref{lel} implies that the closed path
$u$ in question may be assumed to pass through the vertex 1.
So it remains to show that using the relations of $\Lambda$, one
may reduce $u$ to a product $w$ of $\X,\D,\Y^{\pm 1}, T^{\pm 1}$.

To this end, we may assume that the path $u$ begins and ends at the vertex $1$. The proof is by induction in the length $\ell$ of $u$. 
The base of induction ($\ell=0$) is clear. Let us make the induction step from $\ell-1$ to $\ell$. Suppose that $u$ 
ends with $\bold D_1$ (the case of $\bold X_l$ is similar). If all the factors in $u$ are $\bold D_i$ then $u$ is a power of 
$\bold D$, and we are done. So let $m$ be the number of factors $\bold D_i$ in $u$ until the first $\bold X_i$, counting from the end.
We may assume that $m<l$, since otherwise we can split away the factor $\bold D$ at the end of $u$ and pass to smaller length. 
If $m=1$, we can split away the factor $\bold X_1\bold D_1=Z_1^{-1}\bold Y-1$ at the end and reduce to smaller length.
On the other hand, if $m>1$, then  
$$
u=\bar u \bold X_m\bold D_m\bold D_{m-1}\dots \bold D_1
$$
for some word $\bar u$. By adding smaller length words, we may replace $u$ with 
$$
u':=\bar u (1+\bold X_m\bold D_m)\bold D_{m-1}\dots \bold D_1=Z_{m-1}Z_m^{-1}\bar u (1+\bold D_{m-1}\bold X_{m-1})\bold D_{m-1}\dots \bold D_1
$$ 
By adding words of smaller length and rescaling we can replace $u'$ by 
$$
u'':=\bar u\bold D_{m-1}\bold X_{m-1}\bold D_{m-1}\dots \bold D_1,
$$ 
which is a word of the same type as $u$ but with $m$ replaced by $m-1$. Proceeding in this way, we will eventually reach the case 
$m=1$, which has already been considered. This implies the claim. 
\end{proof}

Let ${\mathcal M}_N^l(Z,t)_*$ be the closure of
$\psi^{-1}({\bold M}_N^l(t)^\circ)$ and
${\bold M}_N^l(Z,t)_*$ be the closure of
${\bold M}_N^l(t)^\circ$.

\begin{proposition}\label{isomo} The map
$\psi: {\mathcal M}_N^l(Z,t)_*\to {\bold M}_N^l(Z,t)_*$
is an isomorphism of smooth connected affine varieties.
\end{proposition}

\begin{proof} This follows from Proposition \ref{cloem}.
\end{proof}

We will also see that the multiplicative quiver variety ${\mathcal M}_N^l(Z,t)$ is connected, i.e., 
${\mathcal M}_N^l(Z,t)_*={\mathcal M}_N^l(Z,t)$ (Theorem \ref{conne}). 

\subsection{Connection to cyclotomic DAHA}
Now let us connect multiplicative quiver varieties with cyclotomic DAHA.
Let $\bold e_N$ be the symmetrizer of the finite Hecke algebra
of $S_N$ generated by $T_i$, and
consider the spherical subalgebra $\bold e_N\DAHA_N^l(Z,1,t)\bold e_N$.
This is a subalgebra of the commutative domain
$\bold e_N\DAHA_N(1,t)\bold e_N$ (see~\cite[Theorem~5.1(1),(2)]{Ob1}), so it is also a commutative domain.
Consider the module $\DAHA_N^l(Z,1,t)\bold e_N$ over this algebra.
Let $\Bbb M_N^l(Z,t)={\rm Specm}(\bold e_N\DAHA_N^l(Z,1,t)\bold e_N)$.

\begin{proposition}\label{cmac} For any $Z_i$ the algebra $\bold e_N\DAHA_N^l(Z,1,t)\bold e_N$ is finitely generated and Cohen-Macaulay
(i.e., $\Bbb M_N^l(Z,t)$ is an irreducible Cohen-Macaulay variety) and the module $\DAHA_N^l(Z,1,t)\bold e_N$ is Cohen-Macaulay.  
In particular, $\DAHA_N^l(Z,1,t)\bold e_N$ is projective of rank $N!$ on the smooth locus
$\Bbb M_N^l(Z,t)_{\rm smooth}$ of $\Bbb M_N^l(Z,t)$. 
\end{proposition} 

\begin{proof} The proof is analogous to the proof 
of~\cite[Theorem~5.1(2),(3)]{Ob1}. Namely, the statements 
follow from the fact that by Proposition~\ref{freeness}, $\DAHA_N^{l,+}(Z,1,t)\bold e_N$ and $\bold e_N\DAHA_N^{l,+}(Z,1,t)\bold e_N$ 
are free modules of finite rank over the subalgebra 
$\Bbb C[X_1,...,X_N]^{S_N}\otimes \Bbb C[D_1,...,D_N]^{S_N}$, and $\DAHA_N^l(Z,1,t)\bold e_N$, $e_N\DAHA_N^l(Z,1,t)\bold e_N$
are obtained from these by inverting the element $Y_1...Y_N\bold e_N$.  
\end{proof} 

\begin{proposition}\label{smooth}  
For any $Z_i$, the variety $\Bbb M_N^l(Z,t)$ is smooth outside of a set of codimension two.  
\end{proposition} 

\begin{proof} Consider the open set $\Bbb M_N^l(Z,t)_X=\Bbb M_N^l(t)_X\subset \Bbb M_N^l(Z,t)$
where $\prod_i X_i$ is invertible. On this set, the localization of the spherical cyclotomic DAHA to $\Bbb M_N^l(t)_X$ is the usual spherical DAHA with $q=1$, so $\Bbb M_N^l(t)_X$ is smooth by the result of \cite{Ob1}. Similarly, consider the open set 
$\Bbb M_N^l(Z,t)_D=\Bbb M_N^l(t)_D\subset \Bbb M_N^l(Z,t)$ where $D_i$ are invertible. 
By Proposition \ref{invo}, the localization of the spherical cyclotomic DAHA to $\Bbb M_N^l(t)_D$
is also isomorphic to the usual spherical DAHA, as there is an involution $\phi$ of the cyclotomic DAHA exchanging $X_i$ and $D_i$. 
Thus  $\Bbb M_N^l(t)_D$ is also smooth, by \cite{Ob1}. But it is easy to see that the complement of $\Bbb M_N^l(t)_X\cup \Bbb M_N^l(t)_D$ 
has codimension at least two. 
\end{proof} 

\begin{corollary}\label{normali} The variety $\Bbb M_N^l(Z,t)$ is normal. 
\end{corollary} 

\begin{proof} The proof is similar to the proof of~\cite[Theorem~5.1(2)]{Ob1}. 
Namely, the statement follows from Proposition~\ref{cmac} and 
Proposition~\ref{smooth}, since by the Serre criterion, 
a Cohen-Macaulay variety smooth 
outside of a set of codimension $2$ is normal.  
\end{proof} 

Let ${\mathcal Z}(\DAHA_N^l(Z,1,t))$ be the center of $\DAHA_N^l(Z,1,t))$. 

\begin{proposition}\label{endo} For any $Z_i$: 
(i) The natural map $\DAHA_N^l(Z,1,t)\to {\rm End}_{\bold e_N\DAHA_N^l(Z,1,t)\bold e_N}(\DAHA_N^l(Z,1,t)\bold e_N)$ is an isomorphism. 

(ii) The natural map ${\mathcal Z}(\DAHA_N^l(Z,1,t))\to \bold e_N\DAHA_N^l(Z,1,t)\bold e_N$ given by $\bold z\mapsto \bold z\bold e_N$ is an isomorphism. 
\end{proposition} 

\begin{proof} The proof is analogous to the proof 
of~\cite[Theorem~5.1(4),(5)]{Ob1}, replacing $Y_i$ with $D_i$ and 
the Cherednik involution with the involution $\phi$, and 
using~\cite[Theorem~5.1]{Ob1}. 
\end{proof} 

We can now define a regular map $\xi: \Bbb M_N^l(Z,t)\to {\bold M}_N^l(Z,t)_*$
as follows. Given $\chi\in \Bbb M_N^l(Z,t)_{\rm smooth}$, consider the representation
$I(\chi):=\DAHA_N^l(Z,1,t)\bold e_N\otimes_{\bold e_N\DAHA_N^l(Z,1,t)\bold e_N}\chi$.
By Proposition \ref{cmac}, this representation has dimension $N!$ and is the regular representation of the finite Hecke algebra
generated by $T_i$. Let $\bold e_{N-1}$ be the symmetrizer of the subalgebra generated by $T_2,...,T_{N-1}$, and
let $V(\chi):=\bold e_{N-1}I(\chi)$. This is an $N$-dimensional space, and
it carries an action of the operators $\bold X:=X_1$,
$\bold D=D_1=X_1^{-1}(Y_1-Z_1)...(Y_1-Z_l)$, $\bold Y:=Y_1$ and $T:=T_1....T_{N-1}^2...T_1$.

\begin{proposition}\label{rels} The operators $\bold X,\bold D,\bold Y,T$ satisfy equations \eqref{eqnsvar1},\eqref{eqnsvar2}, i.e.,
define a point of ${\bold M}_N^l(Z,t)$.
\end{proposition}

\begin{proof} Relations \eqref{eqnsvar2} follow from relation
\eqref{lastrel}. The first relation of \eqref{eqnsvar1} is easy, and the second one
follows from the first one and \eqref{lastrel}.
\end{proof}

\begin{remark} 
On the open set where $X_1$ is invertible, Proposition~\ref{rels} reduces to the result of \cite{Ob1}.
\end{remark} 

Proposition \ref{rels} allows us to set $\xi(\chi):=(\bold X,\bold D,\bold Y,T)$, which defines the map $\xi$ on the smooth locus $\Bbb M_N^l(Z,t)_{\rm smooth}$.
By Corollary \ref{normali}, $\xi$ then uniquely extends from the smooth locus to the whole variety $\Bbb M_N^l(Z,t)$. 
Also it is clear that this map lands in ${\bold M}_N^l(Z,t)_*$ (as $\Bbb M_N^l(Z,t)$ is irreducible).

Thus, altogether we obtain a map $\kappa:=\psi^{-1}\circ \xi: \Bbb M_N^l(Z,t)\to {\mathcal M}_N^l(Z,t)_*$.

\begin{proposition}\label{isomor} $\kappa$ is an isomorphism. 
\end{proposition} 

\begin{proof} Consider the restriction $\kappa_X$ of $\kappa$ to the open set $\Bbb M_N^l(t)_X\subset \Bbb M_N^l(Z,t)$
where $X_i$ are invertible. As shown above, $\kappa_X$ is an isomorphism $\Bbb M_N^l(t)_X\to {\mathcal M}_N^l(Z,t)_X$
onto the open set ${\mathcal M}_N^l(Z,t)_X\subset {\mathcal M}_N^l(Z,t)_*$ where $\bold X$ is invertible. 
Similarly, by using the involution $\phi$ of Proposition \ref{invo} and involution $\Phi$ of Lemma \ref{invo1} and the fact that $\kappa\circ \phi=\Phi\circ \kappa$,
we see that the restriction $\kappa_D$ of $\kappa$ to $\Bbb M_N^l(t)_D$ is an isomorphism $\Bbb M_N^l(t)_D\to {\mathcal M}_N^l(Z,t)_D$
onto the open set ${\mathcal M}_N^l(Z,t)_D\subset {\mathcal M}_N^l(Z,t)_*$ where $\bold D$ is invertible. 

Consider the morphism $\kappa^*: {\mathcal O}({\mathcal M}_N^l(Z,t)_*)\to {\mathcal O}({\Bbb M}_N^l(Z,t))$. Obviously, it is an inclusion
which becomes an isomorphism after passing to the fields of fractions. 
Let $F\in {\mathcal O}({\Bbb M}_N^l(Z,t))$. Then $F$ is a rational function 
on ${\mathcal M}_N^l(Z,t)_*$. As shown above, this function is regular on ${\mathcal M}_N^l(Z,t)_X\cup {\mathcal M}_N^l(Z,t)_D\subset {\mathcal M}_N^l(Z,t)$, an open subset whose 
complement has codimension $\ge 2$. But we know that ${\mathcal M}_N^l(Z,t)_*$ is smooth, hence normal. Thus, $F$ extends to a regular function on the whole 
${\mathcal M}_N^l(Z,t)_*$. This implies that $\kappa^*$ and hence $\kappa$ is an isomorphism. 
\end{proof} 

Thus, we obtain the following theorem. 

\begin{theorem} Under the above assumptions on the $Z_i$, the following statements hold. 

(i) The variety $\Bbb M_N^l(Z,t)$ is smooth.

(ii) The module $\DAHA_N^l(Z,1,t)\bold e_N$ over ${\mathcal Z}(\DAHA_N^l(Z,1,t))\cong \bold e_N\DAHA_N^l(Z,1,t)\bold e_N$ is projective of rank $N!$. 

(iii) $\DAHA_N^l(Z,1,t)$ is a split Azumaya algebra over ${\mathcal Z}(\DAHA_N^l(Z,1,t))$ of rank $N!$, namely the endomorphism algebra of the vector bundle
$\DAHA_N^l(Z,1,t)\bold e_N$. Thus, all irreducible representations of $\DAHA_N^l(Z,1,t)$ have dimension $N!$ and are parametrized by points of $\Bbb M_N^l(Z,t)$. 
\end{theorem} 

Thus, we see that irreducible representations of $\DAHA_N^l(Z,1,t)$ are parametrized by points of a connected component of the multiplicative quiver variety.
In fact, it turns out that this is the only connected component. Namely, we have 

\begin{theorem}\label{conne} The variety ${\mathcal M}_N^l(Z,t)$ is connected, i.e., ${\mathcal M}_N^l(Z,t)={\mathcal M}_N^l(Z,t)_*$.  
Thus, $\kappa: {\Bbb M}_N^l(Z,t)\to {\mathcal M}_N^l(Z,t)$ is an isomorphism. 
\end{theorem} 

\begin{proof}
See Subsection \ref{exaaa} below. 
\end{proof} 

\begin{remark} 1. Recall that the multiplicative quiver variety
${\mathcal M}_N^l(Z,t)$ carries a Poisson structure (symplectic for generic parameters and generically symplectic for any parameters), 
coming from the quasi-Hamiltonian reduction procedure (\cite{VdB}).  
Our results imply that the algebra $\bold e_N\DAHA_N^l(Z,q,t)\bold e_N$ (where $q=e^\varepsilon$ and $\varepsilon$ is a formal parameter) 
is a deformation quantization of this Poisson variety (namely, the matching of Poisson brackets may be checked 
on the open set where $\bold X$ is invertible, using the results of \cite{Ob1}).  

2. We expect that the results of this subsection can be lifted to the quantum level. Namely, we expect that
the algebra $\bold e_N\DAHA_N^l(Z,q,t)\bold e_N$ is isomorphic to the quantization of the multiplicative quiver variety
${\mathcal M}_N^l(Z,t)$ defined by D.~Jordan in \cite{J}, via a quantization of the map $\kappa$.\footnote{We note that for $l=1$ we would need a slightly less localized version than that of \cite{J}, not requiring $\bold X$ to be invertible,
which should produce not the full DAHA but its subalgebra.} 
We note that this is known in the degenerate setting, see \cite{Ob2,EGGO}.
\end{remark} 

\begin{remark}\label{cha2} 
The functions ${\rm Tr}(\bold D^r)$, $r=1,...,N$, form 
a classical integrable system on the symplectic variety ${\mathcal M}_N^l(Z,t)$. This system is the classical limit of the quantum integrable system
$\lbrace{ D_1^r+...+D_N^r, r=1,...,N\rbrace}$ in the (spherical) cyclotomic DAHA discussed in Subsection \ref{insys}. 
These classical integrable systems have been studied independently by O. Chalykh and M. Fairon  (\cite{CF}). 
\end{remark} 

\subsection{Multiplicative bow varieties}
\label{bow}
In this section we follow the notations of~\cite{NT}. Given a bow diagram
with a balanced dimension vector as in~\cite[(6.1)]{NT} we consider the
auxiliary diagram of vector spaces and linear maps as in~\cite[6.1]{NT}.
Let $Z_{r,i},\ i=0,\ldots,n-1;\
1\leq r\leq\bw_i$, be a collection of nonzero complex numbers, and let
$t$ be another nonzero complex number.

Let $\BM^\times_{\unl{Z},t}(\unl\bv,\unl\bw)$ be the variety of collections
$(A_i,B_i,B'_i,a_i,b_i,D_{r,i},C_{r,i})$ such
that (cf.~\cite[(6.3),~2.2]{NT})

(i) $B_i,B'_i$ are invertible for all $i=0,\ldots,n-1$;

(a) $B'_iA_{i-1}-A_{i-1}B_{i-1}+a_ib_{i-1}=0$;

(b) $(B'_i)^{-1}=t^{-1}Z_{1,i}^{-1}(1+D_{1,i}C_{1,i})$;

(c) $(1+C_{k,i}D_{k,i})=Z_{k,i}Z_{k+1,i}^{-1}(1+D_{k+1,i}C_{k+1,i})$;

(d) 
$B_i^{-1}=Z_{\bw_i,i}^{-1}(1+C_{\bw_i,i}D_{\bw_i,i})$; 

(S1) There is no nonzero subspace $0\ne S\subset V_i^{\bw_i}$ with
$B_i(S)\subset S,\ A_i(S)=0=b_i(S)$;

(S2) There is no proper subspace $T\subsetneq V^0_{i+1}$ with
$B'_{i+1}(T)\subset T,\ \on{Im}A_i+\on{Im}a_{i+1}\subset T$.

Then $\BM^\times_{\unl{Z},t}(\unl\bv,\unl\bw)$ is an affine algebraic variety
acted upon by $G=\prod_{i=0}^{n-1}\prod_{k=0}^{\bw_i}GL(V_i^k)$,
cf.~\cite[Proposition~3.2]{NT}.

We define a multiplicative bow variety $\CM^\times_{\unl{Z},t}(\unl\bv,\unl\bw)$
as the categorical quotient $\BM^\times_{\unl{Z},t}(\unl\bv,\unl\bw)/\!\!/G$,
cf.~\cite[2.2]{NT}.

\subsection{$K$-theoretic Coulomb branch}
\label{Coul}
Recall~\cite[Section~2]{BFN2} that given a representation $\bN$ of a 
reductive group $G$ one can consider the variety of triples 
$\CR$.\footnote{It is different from the one considered 
in~Section~\ref{triples}.}
According to~\cite[Remark~3.9(3)]{BFN2}, the equivariant $K$-theory
$K^{G_\CO}(\CR)$ is a commutative ring with respect to convolution.
Moreover, if the $G$-action on $\bN$ extends to an action of a larger
group $\tilde G$ containing $G$ as a normal subgroup, then the equivariant
$K$-theory $K^{\tilde{G}_\CO}(\CR)$ is a commutative ring with respect to 
convolution, a deformation of $K^{G_\CO}(\CR)$ over $\on{Spec}(K^{G_F}(\on{pt}))$
where $G_F=\tilde{G}/G$ is the flavor symmetry group. The affine variety
$\on{Spec}(K^{\tilde{G}_\CO}(\CR))$ is denoted $\CM_C^\times(\tilde{G},\bN)$
and called the $K$-theoretic Coulomb branch.

A framed oriented quiver representation gives rise to a representation
$\bN=\bigoplus_{i\to j}\on{Hom}(V_i,V_j)\oplus\bigoplus_i\on{Hom}(W_i,V_i)$ of 
$G=GL(V):=\prod_iGL(V_i)$. Choosing a maximal torus $T(W_i)\subset GL(W_i)$,
we consider the natural action of 
$\hat{G}:=\BC^\times\times G\times\prod_iT(W_i)$ on $\bN$ where $\BC^\times$ 
acts by dilation on the component 
$\bN_{\on{hor}}:=\bigoplus_{i\to j}\on{Hom}(V_i,V_j)$ of $\bN$. 
If $W:=\oplus_iW_i\ne0$, the action of scalars 
$\BC^\times\subset T(W):=\prod_iT(W_i)$ coincides with the action of scalars
$\BC^\times\subset G$. Hence if $W\ne0$, the action of $\hat{G}$ on $\bN$
factors through the action of 
$\tilde{G}:=\BC^\times\times(G\times T(W))/\BC^\times$. If $W=0$, we denote
$\tilde{G}:=\hat{G}$.
The corresponding $K$-theoretic Coulomb branch
$\CM_C^\times(\tilde{G},\bN)$ of a framed quiver gauge theory will be denoted
by $\CM_C^\times$ for short.

Similarly to~\cite[Theorem~6.18]{NT} one can construct an isomorphism
$\CM_C^\times\stackrel{\sim}{\longrightarrow}\CM^\times_{\unl{Z},t}(\unl\bv,\unl\bw)$
from the $K$-theoretic Coulomb branch of a framed quiver gauge theory of
affine type $A_{n-1}$ with dimension vectors $\unl\bv,\unl\bw$.
Here $\unl{Z},t$ in
the RHS are parameters corresponding to the equivariant flavor symmetry
parameters in the LHS, cf.~\cite[6.8.2]{NT}.
The proof of the above isomorphism in particular shows that 
$\CM^\times_{\unl{Z},t}(\unl\bv,\unl\bw)$ is connected 
(and $\CM_C^\times$ is connected similarly to~\cite[Corollary~5.22]{BFN2}).


\subsection{Example}\label{exaaa}
We consider a special case when a bow diagram has 1 cross and $l$ circles
(we allow $l=0$),
that is $n=1,\ \bw_0=l$, and $\dim V_0^k=N$ for any $k=0,\ldots,l$.
Then according to~\cite[Lemma~3.1]{NT}, $A_0\colon V_0^l\to V_0^0$ is an
isomorphism. We identify $V_0^l\equiv V_0^0$ with the help of $A_0$, and then
the definition of $\CM^\times_{\unl{Z},t}(N,l)$ of~Section~\ref{bow} becomes
nothing but the definition of $\CM^l_N(Z,t)$ of~Section~\ref{quiver},
that is $\CM^\times_{\unl{Z},t}(N,l)\simeq\CM^l_N(Z,t)$. In particular,
we conclude that $\CM^l_N(Z,t)$ is connected, which proves Theorem \ref{conne}.

\begin{remark}
The relations between general multiplicative bow varieties and (various
versions of) multiplicative quiver varieties for a cyclic quiver are explained
in~Appendix~\ref{apend}.
\end{remark}

\section{Application to $q$-deformed $m$-quasiinvariants}
\label{sec6}

\subsection{$q$-deformed quasiinvariants}
Let $m$ be a nonnegative integer, and $q\in \Bbb C^*$.

\begin{definition}\label{quasii} (\cite{C})
We call $F\in \Bbb C[X_1^{\pm 1},...,X_N^{\pm 1}]$ a $q$-deformed $m$-quasiinvariant if $(1-s_{ij})F$ is divisible by $\prod_{p=-m}^m (X_i-q^pX_j)$ for any $i<j$.
\end{definition}

The algebra of $q$-deformed $m$-quasiinvariant Laurent (or trigonometric) polynomials will be denoted by $Q_{m,q}^{\rm trig}$.

Let $Q_{m,q}\subset Q_{m,q}^{\rm trig}$ be the graded
algebra of $q$-deformed quasiinvariants inside $\Bbb C[X_1,...,X_N]$. By the Hilbert basis theorem, $Q_{m,q}$ is a finitely generated module
over the ring of symmetric polynomials $\Bbb C[X_1,...,X_N]^{S_N}$. Note that $Q_{m,1}=Q_m$, the usual space of $m$-quasiinvariants defined by Chalykh and Veselov \cite{CV}, i.e., polynomials $F$ such that $(1-s_{ij})F$ is divisible by $(X_i-X_j)^{2m+1}$.

\begin{theorem}\label{quasiinth} For all except countably many values of $q$, the algebra $Q_{m,q}$ has the same Hilbert series as $Q_m$ and is
Cohen-Macaulay, i.e., a free module over $\Bbb C[X_1,...,X_N]^{S_N}$.
\end{theorem}

In other words, Theorem \ref{quasiinth} says that any quasiinvariant polynomial can be $q$-deformed.

Note that the Hilbert series of $Q_m$ is known (see \cite{FeV}, \cite{BEG}).

\begin{remark} Theorem \ref{quasiinth} was conjectured by P.E. and E. Rains on the basis of a computer calculation.
\end{remark}

Theorem \ref{quasiinth} is proved in the next subsection.

\begin{remark} The algebra of $q$-deformed trigonometric quasiinvariants $Q_{m,q}^{\rm trig}(R)$ may be defined for any reduced root system $R$ with Weyl group $W$ and a $W$-invariant multiplicity function $m: R\to \Bbb Z_+$, see \cite{C}. Namely, it is the algebra of regular functions $F$ on the corresponding torus $T$ such that for each $\alpha\in R_+$ the function $F(X)-F(s_\alpha X)$ is divisible by $\prod_{i=-m_\alpha}^{m_\alpha} (e^\alpha-q^i)$. 
Moreover, this algebra is Cohen-Macaulay for generic $q$. To see this, note that by using the exponential map ${\rm exp}: {\rm Lie}(T)\to T$ and rescaling in ${\rm Lie}(T)$, we may identify formal neighborhoods of closed points of ${\rm Spec}(Q_{m,q}^{\rm trig}(R))$ with those of ${\rm Spec}(Q_m^{\rm trig}(R))$,
the usual trigonometric (a.k.a. non-homogeneous) quasiinvariants for $R$, see e.g. \cite{ER}, Remark 6.4. But the algebra $Q_m^{\rm trig}(R)$ is Cohen-Macaulay, see \cite{ER}, Proposition 6.5. Hence so is $Q_{m,q}^{\rm trig}(R)$, as desired. Note that this result also holds for $q=1$ since in this case formal neighborhoods are the same as for usual (rational) quasiinvariants  $Q_m(R)$. 

Theorem \ref{quasiinth} is a refinement of this result for $R=A_{N-1}$, as $Q_{m,q}^{\rm trig}$ is a localization of $Q_{m,q}$. 
\end{remark} 

\subsection{Proof of Theorem \ref{quasiinth}}
The Cohen-Macaulayness statement of Theorem \ref{quasiinth} holds for $q=1$ (i.e., for $Q_m$)
by the results of \cite{EG2}, \cite{BEG} (conjectured earlier in \cite{FV}). Thus it suffices to show that $Q_{m,q}$ is a flat deformation of $Q_m$
when $q=e^\varepsilon$ and $\varepsilon$ is a formal parameter.

Let $\bold e$ be the symmetrizer of the finite Hecke algebra generated by $T_i$, and consider the action of the spherical subalgebra
$\bold e \DAHA_N(q,t)\bold e$ on $\Bbb C[X_1^{\pm 1},...,X_N^{\pm 1}]^{S_N}$. This action is by $q$-difference operators. Consider the element
$Y_1+...+Y_N$. It is a central element of the affine Hecke algebra generated by $T_i$ and $Y_i$, hence commutes with $\bold e$.
Thus
$$
(Y_1+...+Y_N)\bold e\in \bold e\DAHA_N(q,t)\bold e.
$$
We will need the following lemma, due to Cherednik (\cite{Ch1}).

\begin{lemma}\label{macd} The element ${\mathcal M}:=(Y_1+...+Y_N)\bold e$ acts on $\Bbb C[X_1^{\pm 1},...,X_N^{\pm 1}]^{S_N}$ by the first Macdonald operator
$$
M:=\sum_{j=1}^N \left(\prod_{i\ne j}\frac{X_i-tX_j}{X_i-X_j}\right)\tau_j.
$$
where $\tau_iX_j=q^{\delta_{ij}}X_j$.
\end{lemma}

\begin{proof} Denote the $q$-difference operator by which $(Y_1+...+Y_N)\bold e$ acts by $L$. Clearly, we have
$L=\sum_j f_j\tau_j$, where $f_j$ are rational functions such that $s_r f_j=f_j$ for $r\ne j,j+1$, and $s_jf_j=f_{j+1}$. Thus it suffices to show that
$$
f_1= \prod_{i\ne 1}\frac{X_i-tX_1}{X_i-X_1}.
$$
To this end, treat $n$ as a variable, and
note that
$$
L(X_1^n+...+X_N^n)=\sum_j f_j(X_1,...,X_N)(q^nX_j^n+\sum_{i\ne j}X_i^n),
$$
which we can treat as a polynomial in $q^n,X_1^n,...,X_N^n$ with coefficients in $\Bbb C(X_1,...,X_N)$.
Thus, $f_1$ may be obtained by extracting the coefficient of $q^nX_1^n$.

On the other hand, we have
$$
\rho(Y_i)=t^{i-1}\prod_{j=i+1}^N(1+\frac{(1-t)X_i}{X_j-X_i}(1-s_{ij}))\tau_i\prod_{j=1}^{i-1}(1+\frac{(1-t^{-1})X_i}{X_j-X_i}(1-s_{ij})).
$$
An easy direct computation using this formula shows that
$$
(Y_1+...+Y_n)(X_1^n+...+X_N^n)=
$$
$$
(\sum_{i=1}^N t^{i-1})(X_1^n+...+X_N^n)+(q^n-1)\sum_{i=1}^Nt^{i-1}\prod_{j=i+1}^N(1+\frac{(1-t)X_i}{X_j-X_i}(1-s_{ij}))X_i^n
$$
Thus, only the second summand contributes to the coefficient of $q^n$, and in the second summand
only $i=1$ contributes to the coefficient of $X_1^n$. Moreover, in this contribution, the transposition terms (involving $s_{ij}$) don't
contribute. Thus
$$
f_1(X_1,...,X_n)=\prod_{j=2}^N(1+\frac{(1-t)X_1}{X_j-X_1})=\prod_{j=2}^N\frac{X_j-tX_1}{X_j-X_1},
$$
as desired.
\end{proof}

We will also need the following lemma, which is a special case of Proposition 2.1 of \cite{C} (but we give a proof for reader's convenience).
Assume that $q$ is not a root of unity.

\begin{lemma}\label{preserq} If $t=q^{-m}$ then the operator $M$ preserves $Q_{m,q}^{\rm trig}$.
\end{lemma}

\begin{proof} Let $r>l$, and $F\in \Bbb C[X_1^{\pm 1},...,X_N^{\pm 1}]$. A direct computation shows that
$$
(1-s_{rl})(MF)(X_1,...,X_N)=
$$
$$
\prod_{j\ne r}\frac{X_j-q^{-m}X_r}{X_j-X_r}(F(...X_l...qX_r...)-F(...qX_r...X_l...))
$$
$$
-\prod_{j\ne l}\frac{X_j-q^{-m}X_l}{X_j-X_l}(F(...X_r...qX_l...)-F(...qX_l...X_r...)).
$$
If $F\in Q_{m,q}^{\rm trig}$ is a quasiinvariant then $F(...X_l...qX_r...)-F(...qX_r...X_l...)$ vanishes for $X_l=q^pX_r$ with $-m-1\le p\le m-1$,
while  $F(...X_r...qX_l...)-F(...qX_l...X_r...)$ vanishes for $X_l=q^pX_r$ with $-m+1\le p\le m+1$. Hence, both terms in the formula for
$(1-s_{rl})MF$ vanish when $X_l=q^pX_r$ with $-m\le p\le m$, $p\ne 0$ and are defined for $p=0$
(vanishing in the extremal cases $p=m,-m$ follows from vanishing of the prefactors).
Thus, $(1-s_{rl})MF(X_1,...,X_N)$ is a polynomial which is antisymmetric in $X_l,X_r$, so divisible by $X_r-X_l$.
This takes care of the case $p=0$.
\end{proof}

Now note that since the algebra $\bold e\DAHA_N(q,t)\bold e$ acts on $\Bbb C[X_1^{\pm 1},...,X_N^{\pm 1}]$ by $q$-difference operators,
this action can be extended tautologically to non-symmetric rational functions $\Bbb C(X_1^{\pm 1},...,X_N^{\pm 1})$ (by the same difference operators).

\begin{corollary}\label{preser2} If $t=q^{-m}$ for $m\in \Bbb Z_{\ge 0}$ then the spherical cyclotomic DAHA $\bold e\DAHA_N^1(q,t)\bold e$ (with $Z_1=1$) preserves the subspace of $q$-deformed
$m$-quasiinvariants $Q_{m,q}\subset  \Bbb C(X_1^{\pm 1},...,X_N^{\pm 1})$.
\end{corollary}

\begin{proof} By Proposition \ref{preser}, the difference-reflection operators
$\rho(L)$, $L\in \DAHA_N^1(q,t)$ do not create poles at $X_i=0$. Therefore, the action of $\bold e \DAHA_N^1(q,t)\bold e\subset \bold e\DAHA_N(q,t)\bold e$ on $\Bbb C(X_1,...,X_N)$ preserves the subspace $\Bbb C(X_1,...,X_N)_{\rm reg}$ of functions regular at the hyperplanes $X_i=0$, $i=1,...,N$.

We now claim that the algebra $\bold e \DAHA_N(q,t)\bold e$ preserves $Q_{m,q}^{\rm trig}$.
It suffices to prove this in the formal setting $q=e^\varepsilon$, $t=e^{-m\varepsilon}$. By Lemma \ref{macd} and Lemma \ref{preserq}, the element ${\mathcal M}$ preserves
$Q_{m,q}^{\rm trig}$. Also, $\Bbb C[X_1^{\pm 1},...,X_N^{\pm 1}]^{S_N}$ and $\tilde Y:=\prod_i Y_i$ preserve $Q_{m,q}^{\rm trig}$.
But we claim that the element $\widehat{H}:=\varepsilon^{-2}({\mathcal{M}}-N-\varepsilon(\tilde Y-1))$ and  $\Bbb C[X_1^{\pm 1},...,X_N^{\pm 1}]^{S_N}$ (topologically) generate 
$\bold e \DAHA_N^{\rm formal}(1,k)\bold e$. Indeed, the quasiclassical limit of $\widehat{H}$ is the trigonometric Calogero-Moser operator
$H=\bold e\sum_i y_i^2\bold e=\sum_i y_i^2\bold e$, and by Lemma \ref{generat} $H$ and $\Bbb C[X_1^{\pm 1},...,X_N^{\pm 1}]^{S_N}$ generate
$\bold e \DAHA_{N,{\rm deg}}(1,k)\bold e$. This implies the claim.

Thus, $\bold e\DAHA_N^1(q,t)\bold e$ preserves $\Bbb C(X_1,...,X_N)_{\rm reg}\cap Q_{m,q}^{\rm trig}=Q_{m,q}$, as desired.
\end{proof}

Now we can complete the proof of the theorem. Since every polynomial divisible by
$\prod_{i<j}\prod_{p=-m}^m (X_i-q^pX_j)$ is automatically a $q$-deformed $m$-quasiinvariant, the leading coefficient of the Hilbert polynomial
of $Q_{m,q}$ is the same as that of $Q_m$. On the other hand, working in the formal setting ($q=e^\varepsilon$) and reducing Corollary \ref{preser2} modulo $\varepsilon$, we find that
$Q_{m,q}/(\varepsilon)$ is a submodule of the module $Q_m$ over the spherical rational Cherednik algebra
$\bold e\DAHA_N^{\rm rat}(1,m)\bold e$. But as shown in \cite{BEG}, $Q_m=\oplus_\lambda \bold e M_m(\lambda)\otimes \pi_\lambda$,
where $\lambda$ runs over partitions of $N$, $\pi_\lambda$ is the Specht module for $S_N$ corresponding to $\lambda$,
and $M_m(\lambda)$ is the Verma module over the  rational Cherednik algebra $\DAHA_N^{\rm rat}(1,m)$ attached to $\lambda$.
Since the $\bold e\DAHA_N^{\rm rat}(1,m)\bold e$-modules $\bold e M_m(\lambda)$ are irreducible (see \cite{BEG}), 
any proper $\bold e\DAHA_N^{\rm rat}(1,m)\bold e$-submodule of $Q_m$ would have a strictly smaller leading coefficient of the Hilbert series than that of $Q_m$.
This implies that $Q_{m,q}/(\varepsilon)=Q_m$, i.e. $Q_{m,q}$ is a flat deformation of $Q_m$ (i.e., has the same Hilbert series).
By standard abstract nonsense, this applies also to numerical values of $q$, excluding a countable set. The theorem is proved.

\subsection{Generalization: $\bold q$-deformed cyclotomic quasiinvariants}

Let us now use the cyclotomic DAHA to generalize Theorem \ref{quasiinth} to the cyclotomic case.

Let $m\ge 0,l\ge 1, \bold q\in \Bbb C^*$, and $q=\bold q^l$.
Let us first define the algebra of $\bold q$-deformed cyclotomic trigonometric quasiinvariants $Q_{m,\bold q}^{l,{\rm trig}}$.
We introduce variables $x_i$ such that $X_i=x_i^l$. We define $Q_{m,\bold q}^{l,{\rm trig}}\subset \Bbb C[x_1^{\pm 1},...,x_N^{\pm 1}]$
to be the subalgebra of Laurent polynomials $F$ such that for every $i,j,r$, the Laurent polynomial
$$
F(...,x_i,...,x_j,...)-F(...,\zeta^r x_j,...,\zeta^{-r}x_i,...)
$$
is divisible by $\prod_{p=-m}^m(x_i-\zeta^r\bold q^px_j).$

Now let $m,m_1,...,m_{l-1}$ be nonnegative integers. For $0\le r\le l-1$ let $\bold p_r$ be the homogeneous projector
$\Bbb C[x]\to x^r\Bbb C[x^l]$, and $\bold p_r^{(i)}$ denotes $\bold p_r$
acting with respect to $x_i$. By analogy with \cite{BC}, let $Q_{m,m_1,...,m_{l-1},\bold q}^l\subset Q_{m,\bold q}^{l,{\rm trig}}$
be the graded space of all $\bold q$-deformed cyclotomic trigonometric quasiinvariants  $F$ inside
$\Bbb C[x_1,...,x_N]$ such that
\begin{equation}\label{cyccon}
\bold p_r^{(i)}F\text{ is divisible by }x_i^{r+m_rl}\text{ for }i=1,...,N\text{ and }r=1,...,l-1.
\end{equation}

\begin{definition} The space $Q_{m,m_1,...,m_{l-1},\bold q}^l$ is called the space of $\bold q$-deformed cyclotomic $(m,m_1,...,m_{l-1})$-quasiinvariants.
\end{definition}

By the Hilbert basis theorem, $Q_{m,m_1,...,m_{l-1},\bold q}^l$ is a finitely generated module
over the ring of symmetric polynomials $\Bbb C[X_1,...,X_N]^{S_N}$. Note that $Q_{m,m_1,...,m_{l-1}}^l:=Q_{m,m_1,...,m_{l-1},1}^l$ is the usual space of
$(m,m_1,...,m_{l-1})$-quasiinvariants for the complex reflection group $S_N\ltimes (\Bbb Z/l\Bbb Z)^N$ defined in \cite{BC}.
Note also that similarly to \cite{BC}, $Q_{m,m_1,...,m_{l-1},\bold q}^l$ is not necessarily an algebra.

Our main result about cyclotomic $q$-deformed quasiinvariants is the following theorem, which is a generalization\footnote{Theorem \ref{quasiinth} is a special case of Theorem \ref{quasiinth1} for $l=1$, but it is convenient for us to treat this special case first, and then pass to the general case.}
of Theorem \ref{quasiinth}.

\begin{theorem}\label{quasiinth1} For all except countably many values of $\bold q$, the space
$Q_{m,m_1,...,m_{l-1},\bold q}^l$ has the same Hilbert series as $Q_{m,m_1,...,m_{l-1}}^l$, and is a free module over $\Bbb C[X_1,...,X_N]^{S_N}$.
\end{theorem}

In other words, every quasiinvariant for $S_N\ltimes (\Bbb Z/l\Bbb Z)^N$ can be $\bold q$-deformed.

We note that the Hilbert series of $S_N\ltimes (\Bbb Z/l\Bbb Z)^N$ is computed in \cite{BC}.

Theorem \ref{quasiinth1} is proved in the next subsection.

\subsection{Proof of Theorem \ref{quasiinth1}}

The proof of Theorem \ref{quasiinth1} is parallel to the proof of Theorem \ref{quasiinth}, using the results of \cite{BC}. In fact, most of the technical statements we'll need have already been obtained in the proof of Theorem \ref{quasiinth}.

The freeness statement of Theorem \ref{quasiinth1} holds for $Q_{m,m_1,...,m_{l-1}}^l$ (i.e., for $\bold q=1$)
by the results of \cite{BC}. Thus it suffices to show that $Q_{m,m_1,...,m_{l-1},\bold q}^l$ is a flat deformation of $Q_{m,m_1,...,m_{l-1}}^l$
when $\bold q=e^{\varepsilon/l}$ and $\varepsilon$ is a formal parameter.

To this end, recall from \cite{BC} that the space $Q_{m,m_1,...,m_{l-1}}^l$ carries an action of the spherical cyclotomic Cherednik algebra
$\bold e\bDAHA_N^{l,{\rm psc}}(c,1,m)\bold e$, where $\bold e$ is the symmetrizer for $S_N$ and $c_i$ are certain linear functions of $m_j$
(in fact, this action is the main tool in \cite{BC} for proving that $Q_{m,m_1,...,m_{l-1}}^l$ is a free $\Bbb C[X_1,...,X_N]$-module).
Therefore, by Theorem \ref{main1}, this space carries an action of $\bold e\DAHA_N^l(z,1,m)\bold e$, and it is easy to compute that
$z_i=m_i/l$, $i=1,...,l-1$ and $z_l=0$.

The main idea of the proof is to show that this representation can be $\bold q$-deformed to a representation of
the spherical cyclotomic DAHA $\bold e\DAHA_N^l(Z,q,t)\bold e$ on $Q_{m,m_1,...,m_{l-1},\bold q}^l$, where $t=q^{-m}$
and $Z_i=q^{z_i}=\bold q^{m_i}$. Then, similarly to the proof of Theorem \ref{quasiinth}, the result will follow by looking at the leading coefficient of the Hilbert series and
using \cite[Theorem 8.2]{BC}, which gives a decomposition of the $\bold e\DAHA_N^l(z,1,m)\bold e$-module $Q_{m,m_1,...,m_{l-1}}^l$
into a direct sum of irreducible modules.

Finally, let us show that the representation of $\bold e\DAHA_N^l(z,1,m)\bold e$ on the space $Q_{m,m_1,...,m_{l-1}}^l$ admits a $\bold q$-deformation.
Recall from the proof of Theorem \ref{quasiinth} that the algebra $\bold e\DAHA_N(q,t)\bold e$ acts on $\Bbb C(x_1^l,...,x_N^l)$
by difference operators. This action can be straightforwardly extended to the field extension $\Bbb C(x_1,...,x_N)$ by using the same formulas,
where now $\tau_ix_j=\bold q^{\delta_{ij}}x_j$. Hence, the subalgebra $\bold e\DAHA_N^l(Z,q,t)\bold e\subset \bold e\DAHA_N(q,t)\bold e$
acts on $\Bbb C(x_1,...,x_N)$.

Since $X_i-q^pX_j=\prod_{r=0}^{l-1}(x_i-\zeta^r\bold q^px_j)$,
the argument in the proof of Theorem \ref{quasiinth} implies that $\bold e\DAHA_N(q,t)\bold e$ preserves the subspace
$Q_{m,\bold q}^{l,{\rm trig}}\subset \Bbb C(x_1,...,x_N)$. Hence, so does the subalgebra  $\bold e\DAHA_N^l(Z,q,t)\bold e$.
Also, by Proposition \ref{preser}, the algebra $\bold e\DAHA_N^l(Z,q,t)\bold e$ preserves the space
$(x_1...x_N)^{m_i}\Bbb C(X_1,...,X_N)_{\rm reg}$, where, as before,  the subscript 
``reg" means functions regular at the generic points
of the hyperplanes $X_i=0$. Therefore, the algebra $\bold e\DAHA_N^l(Z,q,t)\bold e$ preserves quasiinvariance conditions
\eqref{cyccon}.

Thus, we see that the algebra $\bold e\DAHA_N^l(Z,q,t)\bold e$ preserves the space $Q_{m,m_1,...,m_{l-1},\bold q}^l$.
Moreover, it is easy to see that the classical limit of this representation as $\bold q\to 1$ is exactly the
representation of $\bold e\DAHA_N^l(z,1,m)\bold e$ on the space $Q_{m,m_1,...,m_{l-1}}^l$ constructed in \cite{BC}.
This completes the proof of Theorem \ref{quasiinth1}.

\subsection{Twisted quasiinvariants}

Let $a_1,...,a_N\in \Bbb C$, $m\in \Bbb Z_+$. Let $Q_m(a_1,...,a_N)$ be the space of polynomials $F\in \Bbb C[X_1,...,X_N]$
such that the function
$$
\widetilde F(X_1,...,X_N):=(\prod_i X_i^{a_i})F(X_1,...,X_N)
$$
(regarded as a function of $X_i>0$) is $m$-quasiinvariant, in the sense that $(1-s_{ij})\widetilde{F}$ vanishes to order $2m+1$ at $X_i=X_j$ for all $i<j$.
Note that $Q_m(a_1,...,a_N)=Q_m(a_1-a,...,a_N-a)$, and $Q_m(a,...,a)=Q_m$ for all $a$. Obviously, $Q_m(a_1,...,a_N)$ is a graded $\Bbb C[X_1,...,X_N]^{S_N}$-module.
By the Hilbert basis theorem, this module is finitely generated.

\begin{definition} We will call $Q_m(a_1,...,a_N)$ the module of {\it twisted quasiinvariants}.
\end{definition}

\begin{theorem}\label{freenes} If $a_i-a_j\notin \Bbb Z\setminus \lbrace{0\rbrace}$ for all $i<j$, then $Q_m(a_1,...,a_N)$ is a free $\Bbb C[X_1,...,X_N]^{S_N}$-module
(of rank $N!$).
\end{theorem}

\begin{example}\label{exa} Let $N=2$. Then $Q_m(a,0)$ is a free module over $R:=\Bbb C[X_1,X_2]^{S_2}$ for any $a\in \Bbb C$. We show this by induction in $m$.
The base case $m=0$ is obvious. Let $m>0$. Then it is easy to show that the lowest degree $d$ of a nonzero element $P_{a,m}$ in $Q_m(a,0)$ is $a$ if $a=0,1,...,m-1$, and $m$ otherwise.
Moreover, we can uniquely choose $P_{a,m}$ so that $P_{a,m}(X,X)=X^d$; for example, if $0\le a\le m$ is an integer, then $P_{a,m}=X_2^a$, and
$$
P_{a,1}(X_1,X_2)=\frac{(a-1)X_1+(a+1)X_2}{2a}
$$
for $a\ne 0$.
Therefore,
\begin{equation}\label{gene}
Q_m(a,0)=RP_{a,m}+(X_1-X_2)^2Q_{m-1}(a,0).
\end{equation}
Indeed, given a nonzero homogeneous $F\in Q_m(a,0)$ such that $F(X,X)=\beta X^r$, consider $F':=F-2^{d-r}\beta(X_1+X_2)^{r-d}P_{a,m}$. Then
$F'\in Q_m(a,0)$ and is divisible by $X_1-X_2$, so it is in $(X_1-X_2)^2Q_{m-1}(a,0)$.

Consider first the case when $0\le a\le m-1$ is an integer. Then by the induction assumption, the module
$Q_{m-1}(a,0)$ is free (of rank 2), so it is generated
by $X_2^a$ and some homogeneous polynomial $f_{a,m}$ of degree $2m-1-a$. So by \eqref{gene}, $Q_m(a,0)$ is generated by $X_2^a$
and $f_{a,m+1}:=(X_1-X_2)^2f_{a,m}$, which validates the induction step.

Now consider the case $a\ne 0,...,m-1$. Then $Q_{m-1}(a,0)$ is free by the induction assumption, so it is generated by some homogeneous
polynomials of $P_{a,m-1}$, $T_{a,m-1}$ of degrees $m-1$ and $m$, respectively, such that $T_{a,m}(X,X)=0$ (as one can easily check that one always has such
generators). It is easy to see that there exists $\alpha_{a,m}\in \Bbb C$ such that
$$
P_{a,m}:=T_{a,m-1}+\frac{1}{2}\alpha_{a,m} (X_1+X_2)P_{a,m-1}\in Q_m(a,0),
$$
and $T_{a,m}:=(X_1-X_2)^2P_{a,m-1}\in Q_m(a,0)$. Moreover, by \eqref{gene}, these elements generate $Q_m(a,0)$.
This completes the induction step.

This argument also implies that the Hilbert series of $Q_m(a,0)$ is
$\frac{t^a+t^{2m+1-a}}{(1-t)(1-t^2)}$ if $0\le a\le m-1$ is an integer, and
$\frac{t^m}{(1-t)^2}$ otherwise.
\end{example}

\begin{example}\label{counter} In spite of Example \ref{exa}, for $N\ge 3$ the condition on the $a_i$ cannot be dropped. Indeed, for $N=3$, the computer calculation shows that the Hilbert series
of $Q_2(1,0,0)$ has the form
$$
h(t)=t^2(1+t+2t^2+3t^3+5t^4+7t^5+10t^6+15t^7+20t^8+26t^9+33t^{10}+...)=
$$
$$
\frac{t^2+t^6+t^7+2t^9+t^{10}-t^{12}+...}{(1-t)(1-t^2)(1-t^3)},
$$
and the minus sign in the numerator shows that $Q_2(1,0,0)$ cannot be a free module over symmetric polynomials.
Indeed, if this module were free, the numerator would have been the Hilbert polynomial of the generators.
\end{example}

\begin{proof} (of Theorem \ref{freenes}) By permuting $a_i$ we may assume that
$$
(a_1,...,a_N)=(z_1,...,z_1,...,z_l,...,z_l),
$$
where $z_i\ne z_j$ and $z_i$ occurs $N_i$ times, where $N=N_1+...+N_l$. By simultaneously shifting $a_i$ we may assume that $z_l=0$.

Assume first that $z_r=\frac{r}{l}+m_r$, $1\le r\le l-1$.
Let $\chi$ be the character of $(\Bbb Z/l\Bbb Z)^N$ given by $\chi(\sigma_i)=\zeta^{la_i}$. Denote by $Q^{l,\chi}_{m,m_1,...,m_{l-1}}$ the
 $\chi$-eigenspace of $(\Bbb Z/l\Bbb Z)^N$ in $Q^l_{m,m_1,...,m_{l-1}}$. Then it is easy to see that $F\in Q^{l,\chi}_{m,m_1,...,m_{l-1}}$ if and only if it has the form
 $$
 F(x_1,...,x_N)=X_1^{a_1}...,X_N^{a_N}f(X_1,...,X_N),
 $$
 where $X_i=x_i^l$, and satisfies the quasiinvariance condition saying that $(1-s_{ij})F$ vanishes to order $2m+1$ at $X_i=X_j$ (for $X_i>0$), i.e., if and only if
 $f\in Q_m(a_1,...,a_N)$.

 Let $G:=S_N\ltimes (\Bbb Z/l\Bbb Z)^N$.
 By \cite[Theorem 8.2]{BC},
\begin{equation}\label{charfor}
 Q^{l,\chi}_{m,m_1,...,m_{l-1}}=\oplus_{\tau\in {\rm Irrep}(G): \tau^\chi\ne 0} \bold e_G M_{c,k}(\tau^*)\otimes \tau^\chi,
\end{equation}
where $\bold e_G$ is the symmetrizer for  $G$,
$M_{c,k}(\tau)$ denotes the Verma module over the cyclotomic rational Cherednik algebra $\bDAHA_N^{l,{\rm cyc}}(c,1,k)$ for appropriate $c,k$,
and the superscript $\chi$ denotes the $\chi$-eigenspace. Here the grading on $Q^{l,\chi}_{m,m_1,...,m_{l-1}}$ is obtained from the grading on
the right hand side of \eqref{charfor} (given by the scaling element $\bold h$ of the rational Cherednik algebra) by dividing by $l$ and shifting by $m\frac{N(N-1)}{2}+\frac{N}{2l}$.

Thus, $Q_m(a_1,...,a_N)$ is a free module over $\Bbb C[X_1,...,X_N]^{S_N}$, with Hilbert series
independent of the numbers $m_r$ . Since this holds for a Zariski dense set of vectors $(z_1,...,z_{l-1})$ (namely, $z_r=\frac{r}{l}+m_r$, $m_r\in \Bbb Z_+$), this holds for Weil generic
$(z_1,...,z_{l-1})$.

It remains to show that the statement holds if $z_i-z_j$ is not a nonzero integer for $1\le i<j\le l$. Using Theorem \ref{main1} and formula \eqref{charfor}, we see
that $Q^{l,\chi}_{m,m_1,...,m_{l-1}}$ is a module over the spherical subalgebra $\bold e \DAHA^l_{N,{\rm deg}}(z,1,m)\bold e$, where $z_r=\frac{r}{l}+m_r$ for $1\le i\le l-1$ and $z_l=0$.
Interpolating to arbitrary complex values of $z_i$, we get that for any $z_1,...,z_{l-1}\in \Bbb C$, the algebra $\bold e \DAHA^l_{N,{\rm deg}}(z,1,m)\bold e$
with $z:=(z_1,...,z_{l-1},0)$ acts on $Q_m(a_1,...,a_N)$. If $z_i-z_j$ is not a nonzero integer, then the category ${\mathcal O}$ for the algebra $\bold e \DAHA^l_{N,{\rm deg}}(z,1,m)\bold e$
is semisimple (see \cite[Theorem 6.6]{BC}). Moreover, by a deformation argument, $Q_m(a_1,...,a_N)$ must contain representation \eqref{charfor}. Since every irreducible representation
in $\mathcal{O}$ has full support, this implies that $Q_m(a_1,...,a_N)$ coincides with \eqref{charfor}.
This proves the theorem.
\end{proof}

Formula \eqref{charfor} allows us to easily compute the character of $Q_m(a_1,...,a_N)$ as a graded $S_{N_1}\times...\times S_{N_{l}}$-module. Namely, given irreducible representations $\pi_r$ of $S_{N_r}$, we have from \eqref{charfor}:
$$
\Hom_{\prod_r S_{N_r}}(\pi_1\otimes...\otimes \pi_l,Q_m(a_1,...,a_N))=\bold e_G M_{c,k}(\tau(\pi_1,...,\pi_l,\chi)^*),
$$
where $\tau(\pi_1,...,\pi_l,\chi)={\rm Ind}_{(\Bbb Z/l\Bbb Z)^N\rtimes \prod_r S_{N_r}}^{(\Bbb Z/l\Bbb Z)^N\rtimes S_N}\chi$, and the grading is modified
as explained above. Therefore, we have

\begin{proposition}\label{hilser} If $(a_1,...,a_N)=(z_1,...,z_1,...,z_{l},...,z_{l})$ and $z_i-z_j$ are not nonzero integers, then
the Hilbert series of the graded vector space $\Hom_{\prod_r S_{N_r}}(\pi_1\otimes...\otimes \pi_l,Q_m(a_1,...,a_N))$ equals
$$
h_{\pi_1,...,\pi_l}(t)=t^{m(\frac{N(N-1)}{2}-\sum_{r=1}^l {\rm cont}(\pi_r))}h_{\pi_1}(t)...h_{\pi_l}(t),
$$
where  for an irreducible representation $\pi$
of $S_n$, ${\rm cont}(\pi)$ is the content of the Young diagram of $\pi$, and
$h_\pi(t)$ is the Hilbert series of the graded space
$(\pi\otimes \Bbb C[X_1,...,X_n])^{S_n}$, i.e.,
$$
h_\pi(t)=\frac{K_\pi(t)}{(1-t)...(1-t^n)},
$$
where $K_\pi(t)$ is the Kostka polynomial associated to $\pi$.
\end{proposition}

\begin{example}
1. If $l=1$ (i.e., $a_i=0$), we recover the standard formula for the character of $Q_m$, see e.g.
\cite{FeV}.

2. Let $N=2$, $l=2$, $N_1=N_2=1$.  We get from Proposition \ref{hilser} that the Hilbert series of $Q_m(a_1,a_2)$ is $h(t)=\frac{t^m}{(1-t)^2}$, i.e., we recover the formula of Example \ref{exa} for the case of generic $a$.

3. Let $N=3$, $l=3$, $N_1=N_2=N_3=1$. Then we get from Proposition \ref{hilser} that  the Hilbert series of $Q_m(a_1,a_2,a_3)$ is $h(t)=\frac{t^{3m}}{(1-t)^3}$.

 4. Let $N=3$, $l=2$, $N_1=1$, $N_2=2$. The space $Q_m(a,0,0)$ splits into the direct sum
 $$
 Q_m(a,0,0)=Q_m(a,0,0)_+\oplus Q_m(a,0,0)_-,
$$
the symmetric and antisymmetric part under $s_{23}$. Denoting the Hilbert series of these spaces by $h_+$ and $h_-$, we get
from Proposition \ref{hilser}:
$$
h_+(t)=\frac{t^{2m}}{(1-t)^2(1-t^2)},
$$
$$
h_-(t)=\frac{t^{4m+1}}{(1-t)^2(1-t^2)}.
$$
\end{example}

\subsection{$q$-deformed twisted quasiinvariants}

Keep the notation of the previous subsection, and let $q>0$. Define the module of {\it $q$-deformed twisted quasiinvariants} $Q_{m,q}(a_1,...,a_N)$
to be the space of polynomials $F\in \Bbb C[X_1,...,X_N]$
such that the function
$$
\widetilde F(X_1,...,X_N):=(\prod_i X_i^{a_i})F(X_1,...,X_N)
$$
(regarded as a function of $X_i>0$) is a $q$-deformed $m$-quasiinvariant, in the sense that $(1-s_{ij})\widetilde{F}$ is divisible by
$\prod_{p=-m}^m (X_i-q^pX_j)$ in the ring of analytic functions. If $q\ne 1$, this is equivalent to saying that
$(1-s_{ij})\widetilde{F}$ vanishes if $X_i=q^pX_j$, $-m\le p\le m$. Note that 
$$
Q_{m,1}(a_1,...,a_N)=Q_m(a_1,...,a_N),\
Q_{m,q}(a_1,...,a_N)=Q_{m,q}(a_1-a,...,a_N-a),$$
$$\text{ and }Q_{m,q}(a,...,a)=Q_{m,q}\text{ for all}\ a.
$$
Obviously, $Q_{m,q}(a_1,...,a_N)$ is a graded $\Bbb C[X_1,...,X_N]^{S_N}$-module.
By the Hilbert basis theorem, this module is finitely generated.

\begin{theorem}\label{qdef} If $a_i-a_j$ are not nonzero integers then for all but countably many $q$, the Hilbert series of $Q_{m,q}(a_1,...,a_N)$
coincides with the one for $Q_m(a_1,...,a_N)$ (i.e., $Q_{m,q}(a_1,...,a_N)$ is a flat $q$-deformation of $Q_m(a_1,...,a_N)$). Moreover,
$Q_{m,q}(a_1,...,a_N)$ is a free $\Bbb C[X_1,...,X_N]^{S_N}$-module (of rank $N!$).
\end{theorem}

\begin{proof} The proof is parallel to the proof of Theorem \ref{quasiinth1}, using Theorem \ref{freenes}.
Namely, the second statement follows from the first one, and it suffices to prove the first statement
in the formal setting, $q=e^\varepsilon$. In this case, one shows as in the proof of Theorem \ref{quasiinth1} that
the action of $\bold e\DAHA_{N,{\rm deg}}^l(z,1,m)\bold e$ on $Q_m(a_1,...,a_N)$ constructed in the proof of Theorem \ref{freenes}
can be $q$-deformed to an action of $\bold e\DAHA_N^{l,{\rm formal}}(z,1,m)\bold e$ on $Q_{m,q}(a_1,...,a_N)$.
The rest of the argument is the same as in the proof of Theorem \ref{quasiinth1}.
\end{proof}

\begin{remark} We expect that $Q_{m,q}(a_1,...,a_N)$ is a flat deformation of $Q_m(a_1,...,a_N)$ for all values of $a_i$,
in particular when the module $Q_m(a_1,...,a_N)$ is not necessarily free. This has been confirmed by a computer calculation in low degrees
in Example \ref{counter}.

Also, we expect that Theorems \ref{quasiinth}, \ref{quasiinth1}, \ref{qdef} hold for all nonzero values of $q$.
\end{remark}

\appendix
\newenvironment{NB}{
\color{red}{\bf NB}. \footnotesize
}{}
\newenvironment{NB2}{
\color{blue}{\bf NB2}. \footnotesize
}{}
\excludeversion{NB}
\excludeversion{NB2}

\newcommand{\RR}{\mathbb R}
\newcommand{\tn}{{\text{n}}}
\newcommand{\Coker}{\operatorname{Coker}}
\newcommand{\Ima}{\operatorname{Im}}
\newcommand{\xl}{{
x}}
\newcommand{\xp}{{
\sigma}}

\newcommand{\fus}{\mathrm{fus}}
\newcommand{\qQ}{\mathsf{Q}}
\newcommand{\qs}{\mathsf{s}}
\newcommand{\qt}{\mathsf{t}}
\newcommand{\qa}{h}
\newcommand{\qv}{i}
\newcommand{\vbv}{\underline{\mathbf{v}}}
\newcommand{\vbw}{\underline{\mathbf{w}}}

\newcommand{\MM}{\mathcal{M}^\times}
\newcommand{\MMreg}{\mathcal{M}^{\times\mathrm{reg}}}

\section{More general multiplicative quiver and bow varieties
  \\
by Hiraku Nakajima and Daisuke Yamakawa}
\label{apend}

In this appendix, we study more general multiplicative quiver and bow
varieties, as multiplicative analog of results in \cite{NT}.

Multiplicative bow varieties are defined as in Section~\ref{bow}
corresponding to more general bow diagrams with dimension vectors not
necessarily satisfying the balanced condition:
\begin{equation}\label{eq:App3}
\begin{xy}
(5,3)*{\bv_{i-1}},
(10,0)*{\boldsymbol\times},
(10,-4)*{\xl_{i}},
(15,3)*{\bv_i},
(20,0)*{\boldsymbol\medcirc},
(20,-4)*{h_{j+1}},
(25,3)*{\bv_{i}'},
(30,0)*{\boldsymbol\medcirc},
(30,-4)*{h_{j}},
(35,3)*{\bv_{i}''},
(40,0)*{\boldsymbol\times},
(40,-4)*{\xl_{i+1}},
(45,3)*{\bv_{i+1}},
\ar @{-} (5,0);(45,0),
\ar @{.} (0,0);(4,0),
\ar @{.} (46,0);(50,0)
\end{xy}
\end{equation}
Recall that the balanced condition is $\bv_i = \bv_i' = \bv_i''$ for
all $i$. We put $A$, $B$, $a$, $b$ at $\boldsymbol\times$, and $C$,
$D$ at $\boldsymbol\medcirc$ as
\begin{equation*}
  \begin{tikzcd}[column sep=small]
    \cdots \ 
    \CC^{\bv_{i-1}} \arrow[rr, "A_i"] \arrow[rd, "b_{i-1}"']
    \arrow[out=120,in=60,loop,looseness=3, "B_{i-1}"] &&
    \CC^{\bv_i} \arrow[shift left=1,rr,"C_{j+1}"]
    \arrow[out=120,in=60,loop,looseness=3, "B_{i}"] &&
    \CC^{\bv_i'} \arrow[shift left=1, ll, "D_{j+1}"]
    \arrow[shift left=1, rr, "C_j"] && 
    \CC^{\bv_i''} \arrow[shift left=1,ll ,"D_{j}"] \arrow[rd, "b_i"']
    \arrow[out=120,in=60,loop,looseness=3, "B''_{i}"] 
    \arrow[rr, "A_{i+1}"] &&
    \CC^{\bv_{i+1}}\arrow[out=120,in=60,loop,looseness=3, "B_{i+1}"] 
    \ \cdots \\
    & \CC\arrow[ru, "a_i"'] &&&&&& \CC\arrow[ru, "a_{i+1}"'] &
  \end{tikzcd}
\end{equation*}
We only consider the stability parameter $\nu^\RR = 0$ for brevity, as
generalization to arbitrary $\nu^\RR$ is straightforward.

The definition of multiplicative quiver varieties is more delicate
when $\dim W \neq 1$, hence will be given in Section~\ref{subsec:mq}.

\subsection{Hanany-Witten transition}

We first introduce Hanany-Witten transition for multiplicative bow
varieties.

Consider the following part of bow data:
\begin{equation*}
\xymatrix@C=1.2em{ V_1 \ar@(ur,ul)_{B_1} \ar@<-.5ex>[rr]_{C} 
  && V_2 \ar@(ur,ul)_{B_2}
  \ar@<-.5ex>[ll]_{D} \ar[rr]^{A} \ar[dr]_{b} && V_3 \ar@(ur,ul)_{B_3}
  \\
  &&& \CC \ar[ur]_a &}
\qquad
\begin{aligned}[m]
  & B_1 = t Z(1 + DC)^{-1} = tZ (1 - D(1+CD)^{-1}C),
  \\
  & B_2 = Z^{\prime}(1 + CD)^{-1} = Z'(1 - C(1+DC)^{-1}D),
  \\
  & B_3A - AB_2 + ab = 0,
\end{aligned}
\end{equation*}
where $t$, $Z$, $Z'$ are fixed nonzero scalars.
\begin{NB}
  $(1+DC)(1 - D(1+CD)^{-1}C) = 1 - D(1+CD)^{-1}C + DC - DCD(1+CD)^{-1} C
  = 1 + DC - D(1+CD)(1+CD)^{-1}C = 1$.
\end{NB}%
We consider another bow data with the same $V_1$, $V_3$, $B_1$, $B_3$:
\begin{equation*}
\xymatrix@C=1.2em{ V_1 \ar@(ur,ul)_{B_1} \ar[rr]^{A^\tn} \ar[dr]_{b^\tn} &&
  V_2^\tn
  \ar@(ur,ul)_{B_2^\tn} \ar@<-.5ex>[rr]_{C^\tn} && V_3 \ar@(ur,ul)_{B_3}
  \ar@<-.5ex>[ll]_{D^\tn}
  \\
  & \CC \ar[ur]_{a^\tn} &&&}
\qquad
\begin{aligned}[m]
    & B_2^\tn = tZ(1+D^\tn C^\tn)^{-1} = tZ(1 - D^\tn(1 + C^\tn
    D^\tn)^{-1} C^\tn),
    \\
    & B_3 = Z'(1+C^\tn D^\tn)^{-1} = Z'(1 - C^\tn(1 + D^\tn
    C^\tn)^{-1} D^\tn),
    \\
    & B_2^\tn A^\tn - A^\tn B_1 + a^\tn b^\tn = 0.
\end{aligned}
\end{equation*}
We take quotients by $\GL(V_2)$, $\GL(V_2^\tn)$ respectively. (But not
by $\GL(V_1)$, $\GL(V_3)$.)

\begin{proposition}\label{prop:HW}
  There exists a $\GL(V_1)\times \GL(V_3)$-equivariant isomorphism
  between two varieties above.
\end{proposition}

The isomorphism is explicitly given during the proof.

\begin{proof}
We consider a three term complex
\begin{equation*}
    \begin{CD}
        V_2 @>\alpha = \left[
          \begin{smallmatrix}
              D \\ A \\ b
          \end{smallmatrix}\right]>> 
        V_1\oplus V_3 \oplus \CC @>{\beta = \left[
          \begin{smallmatrix}
              AB_2C & (B_3 - Z')& a 
          \end{smallmatrix}\right]}>>
        V_3,
    \end{CD}
\end{equation*}
where $\beta\alpha = 0$ follows from one of the definining equation.
\begin{NB}
  $AB_2CD + (B_3 -Z')A + ab = A(Z' - B_2) + (B_3 - Z') A + ab = B_3 A
  - A B_2 + ab = 0$.
\end{NB}%
We claim that $\alpha$ is injective. We consider $S =
\Ker\alpha$. Then $A(S) = 0 = b(S)$ and $B_2(S) = (1 - C(1+DC)^{-1}D)(S)
= S$. Therefore the condition (S1) implies $S = 0$.
        
Let 
\begin{equation*}
  V_2^\tn := \Coker\alpha.
\end{equation*}
We define new bow data as
\begin{itemize}
\item $A^\tn$, $D^\tn$, $a^\tn$ are composition of inclusions of $V_1$,
$V_3$, $\CC$, to $V_1\oplus V_3\oplus\CC$ and the projection
$V_1\oplus V_3\oplus\CC\twoheadrightarrow V_2^\tn$ respectively.
\item $b^\tn = bCB_1$, and $C^\tn$ is a homomorphism
  induced from $-B_3^{-1}\beta$.
\item $B_2^\tn = tZ(1 + D^\tn C^\tn)^{-1} = tZ (1 - D^\tn(Z^{\prime-1} B_3)
C^\tn)$. 
\end{itemize}
The last definition requires checking of the invertibility of
$1 + D^\tn C^\tn$.
Let us consider $1 + C^\tn D^\tn$. From the definition it is
$1 - B_3^{-1}(B_3 - Z') = Z' B_3^{-1}$. This is the second of the
defining equation. Thus $1+C^\tn D^\tn$ is invertible, and hence
$1 + D^\tn C^\tn$ is also invertible. (In fact,
$(1 + D^\tn C^\tn)^{-1} = 1 - D^\tn(1 + C^\tn D^\tn)^{-1} C^\tn$.)
Hence the above $B_2^\tn$ is well defined.

In order to check the remaining defining equation, we consider
\begin{equation*}
  \begin{CD}
    V_1 @>\alpha^\tn =\left[
      \begin{smallmatrix}
        tZ-B_1 \\ -C^\tn B_2^\tn A^\tn \\ b^\tn
      \end{smallmatrix}\right]>>
    V_1 \oplus V_3\oplus \CC @>{\beta^\tn = \left[
        \begin{smallmatrix}
          A^\tn & D^\tn & a^\tn
        \end{smallmatrix}\right]}>> V_2^\tn.
  \end{CD}
\end{equation*}
This is a complex if and only if the last defining equation holds.
\begin{NB}
  $A^\tn(tZ - B_1) - D^\tn C^\tn B_2^\tn A^\tn + a^\tn b^\tn
  = B_2^\tn A^\tn - A^\tn B_1 + a^\tn b^\tn$, as
  $D^\tn C^\tn B_2^\tn = (1 + D^\tn C^\tn) B_2^\tn - B_2^\tn = tZ - B_2^\tn$.
\end{NB}%
Observe that $\beta^\tn$ is nothing but the natural projection. We
also have
\begin{equation}\label{eq:App1}
  \alpha^\tn =
  \begin{bmatrix}
    tZ-B_1 \\ -t Z Z^{\prime-1} B_3 C^\tn A^\tn \\ bC B_1
  \end{bmatrix}
  =
  \begin{bmatrix}
    DC B_1 \\ t Z Z^{\prime-1} AB_2 C \\ b C B_1
  \end{bmatrix}
  = \alpha CB_1,
\end{equation}
where we have used $C B_1 = tZ Z^{\prime-1} B_2 C$,
$C^\tn B_2^\tn = tZ Z^{\prime-1} B_3 C^\tn$.
\begin{NB}
  $CB_1^{-1} = t^{-1} Z^{-1} C(1+DC) = t^{-1} Z^{-1} (1 + CD) C
  = t^{-1} Z^{-1} Z' B_2^{-1} C$.
\end{NB}%
Therefore $\beta^\tn\alpha^\tn = 0$.

\begin{NB}
    In summary, we have
    \begin{equation*}
        \left[V_1\xrightarrow{\alpha^\tn} V_1\oplus V_3\oplus\CC
        \xrightarrow{\beta^\tn} V_2^\tn\right]
      = \left[V_1\xrightarrow{\alpha C B_1} V_1\oplus V_3\oplus\CC
        \xrightarrow{\text{projection}} \Coker\alpha\right].
    \end{equation*}
\end{NB}%

Let us check the condition (S1) for new data. Take a subspace
$S\subset V_1$ such that $B_1(S)\subset S$, $A^\tn(S) = 0 =
b^\tn(S)$. Observe that $A^\tn(S) = 0$ means
$S \oplus 0\oplus 0\subset \Ima\alpha$. Let us consider
$\tilde S = \alpha^{-1}(S \oplus 0\oplus 0)$. Then $D(\tilde S) = S$
and $A(\tilde S) = 0 = b(\tilde S)$. Therefore
\begin{equation*}
  \alpha (Z' - B_2)(\tilde S) = \alpha B_2 CD(\tilde S)
  = \alpha B_2 C(S) = \alpha CB_1(S) 
    = \alpha^\tn (S) =
    \begin{bmatrix}
        (Z-B_1)(S) \\ 0 \\ 0
    \end{bmatrix}.
\end{equation*}
The condition $B_1(S)\subset S$ implies
$B_2(\tilde S)\subset \tilde S$. Hence $\tilde S = 0$ thanks to (S1)
for the original data. We have $S=0$ as well.

Let us check the condition (S2). Suppose we have a subspace $T\subset
V_2^\tn$ such that $B_2^\tn(T)\subset T$, $\Ima A^\tn + \Ima a^\tn
\subset T$. We take its inverse image $\tilde T = (\beta^\tn)^{-1}(T)$
in $V_1\oplus V_3\oplus \CC$. By the second assumption, it contains
$V_1\oplus \{0\}\oplus \CC$. Hence $\tilde T$ is a form of $V_1 \oplus \bar
T\oplus\CC$ for $\bar T\subset V_3$.
We also have $\Ima\alpha\subset\tilde T$.
\begin{NB}
    as $\beta^\tn\alpha = 0$ by the definition of $\beta^\tn$.
\end{NB}%
Hence $A(V_2)\subset \bar T$.  As $B_2^\tn = tZ (1 - D^\tn(Z^{\prime-1} B_3)
C^\tn)$, 
\begin{NB}
    and $D^\tn B_3 C^\tn$ is the composite
    $\Coker\alpha\xrightarrow{-\beta} V_3 \xrightarrow{
      \begin{bmatrix}
          0\\ \id \\ 0
      \end{bmatrix}
      } V_1\oplus V_3\oplus \CC \xrightarrow{\beta^\tn} \Coker\alpha$.
\end{NB}%
the condition $B_2^\tn(T)\subset T$ implies
$0 \oplus \beta(\tilde T)\oplus 0\subset \tilde T$, i.e.,
$AB_2C(V_1)+(B_3-Z')(\bar T)+a(\CC)\subset \bar T$.
Hence $\bar T = V_3$ thanks to (S2) for the original data. We have
$T=V_2^\tn$ as well.

The inverse construction is clear. The original vector space $V_2$ is
recovered from the new data as $\Ker \beta^\tn$. Note also $\beta^\tn$
is surjective thanks to (S2).
\begin{NB}
  Consider $T = \Ima \beta^\tn$. We have
  $\Ima A^\tn + \Ima a^\tn \subset T$, and
  $B_2^\tn(T) = tZ(1-D^\tn(Z^{\prime-1} B_3)C^\tn)(T)\subset
  T$. Therefore $T = V_2$ by (S2).
\end{NB}%
Then $A$, $b$, $D$ are given as restrictions of projections
$V_1\oplus V_3\oplus\CC$ to $V_3$, $\CC$, $V_1$ to $\Ker\beta^\tn$
respectively, $a$ is $-B_3 C^\tn a^\tn$, and $C$ is
$\alpha^\tn B_1^{-1}$ by \eqref{eq:App1}.
\begin{NB}
  Note $\alpha$ now becomes the inclusion.
\end{NB}%
\begin{NB}
  Let us check that the first defining relation. We have
  $1 + D C = 1 + (t Z - B_1) B_1^{-1} = tZ B_1^{-1}$.
\end{NB}%
Finally we set $B_2 = Z'(1 + CD)^{-1}$.
\begin{NB}
  Then we have
  \begin{equation*}
    \left[V_2 \xrightarrow{\alpha} V_1\oplus V_3\oplus \CC
      \xrightarrow{\beta} V_3
    \right]
    = \left[\Ker \beta^\tn \xrightarrow{\text{inclusion}}
      V_1\oplus V_3\oplus \CC
      \xrightarrow{-B_3 C^\tn \beta^\tn} V_3
    \right].
  \end{equation*}
  In fact,
  \begin{equation*}
    \begin{split}
    \beta &= \left[
      \begin{smallmatrix}
        AB_2 C & (B_3- Z') & a
      \end{smallmatrix}
    \right]
    = \left[
      \begin{smallmatrix}
        (t Z)^{-1} Z' A C B_1 & -B_3 C^\tn D^\tn & -B_3 C^\tn a^\tn
      \end{smallmatrix}
    \right] \\
    &= \left[
      \begin{smallmatrix}
        - B_3 C^\tn A^\tn & -B_3 C^\tn D^\tn & -B_3 C^\tn a^\tn
      \end{smallmatrix}
    \right].
    \end{split}
  \end{equation*}
\end{NB}%

The conditions (S1),(S2) for $(A,B_2,B_3,a,b)$ follow from the
conditions (S1),(S2) for $(A^\tn,B_1,B_2^\tn,a^\tn,b^\tn)$. We leave
the details to the reader as an exercise.
\begin{NB}
  Suppose we have a subspace $S\subset V_2 = \Ker\beta^\tn$ such that
  $B_2(S)\subset S$ and $A(S) = 0 = b(S)$. Since $A$, $b$ are
  projections of $\Ker\beta^\tn\to V_3$, $\CC$ respectively, the
  second condition means $S\subset V_1$. Therefore $\beta^\tn(S) = 0$
  means $A^\tn(S) = 0$. Since
  $B_2 - Z' = - Z' C(1+DC)^{-1}D = - Z' (tZ)^{-1} C B_1 D =
  -Z'(tZ)^{-1} \alpha^\tn\circ (\Ker\beta^\tn\to V_1)$,
  $B_2(S)\subset S$ implies $B_1(S)\subset S$, $b^\tn(S) =
  0$. Therefore the condition (S1) for the `new' data implies $S=0$.

    Suppose we have a subspace $T\subset V_3$ such that $B_3(T)\subset
    T$, $\Ima A + \Ima a \subset T$. The second condition means $\Ima
    (\Ker\beta^\tn\to V_3) + \Ima B_3 C^\tn a^\tn \subset T$.
    We consider $\tilde T := \beta^\tn(V_1\oplus T\oplus \CC)$. It
    contains $\Ima A^\tn + \Ima a^\tn$. Let us check $B_2^\tn(\tilde
    T) \subset \tilde T$. Note $B_2^\tn\beta^\tn(V_1\oplus T\oplus\CC)
    = B_2^\tn A^\tn(V_1) + B_2^\tn D^\tn(T) + B_2^\tn a^\tn(\CC)$. We
    have $B_2^\tn A^\tn = A^\tn B_1 - a^\tn b^\tn$. Therefore $B_2^\tn
    A^\tn(V_1) \subset \Ima A^\tn + \Ima a^\tn \subset \widetilde
    T$. We also have
    \begin{equation*}
      B_2^\tn D^\tn
    = tZ (1 - D^\tn Z^{\prime-1} B_3 C^\tn) D^\tn
    = tZ D^\tn - tZ D^\tn Z^{\prime-1} B_3 (Z'B_3^{-1} - 1)
    = tZ Z^{\prime-1} D^\tn B_3.
    \end{equation*}
    Therefore $B_2^\tn D^\tn(T) \subset D^\tn T$. We have
    \begin{equation*}
      B_2^\tn a^\tn(\CC) = tZ (1 - D^\tn Z^{\prime-1} B_3 C^\tn) a^\tn(\CC)
      \subset \Ima a^\tn + D^\tn T.
    \end{equation*}
    By the definition of $\beta^\tn$, we have $D^\tn T\subset \tilde
    T$. Thus $\tilde T = V_2^\tn$ by (S2) for the `new' data.

    Take $v_3\in V_3$. Then there exists
    $v_1\oplus v_3'\oplus \lambda \in V_1\oplus T\oplus\CC$ such that
    $\beta^\tn(v_1\oplus v_3'\oplus\lambda) = \beta^\tn(0\oplus v_3
    \oplus 0)$ by the above discussion. Therefore
    $v_1\oplus (v_3'-v_3)\oplus \lambda\in\Ker\beta^\tn$. Since
    $T\supset \Ima (\Ker\beta^\tn\to V_3)$, we have $v_3' - v_3\in
    T$. Therefore $v_3\in T$ as well. It means $T = V_3$.
\end{NB}%
\begin{NB}
Let us check the stability condition. Note that the isomorphism
respects the group action by $\GL(V_1)\times \GL(V_3)$. As we can
choose any segment $\zeta$ from a wavy line $\xp$ in the numerical
criterion, we could also choose any $\zeta$ in the definition of
$\chi$ for the stability condition. In particular, if we choose
$\GL(V_1)$ and $\GL(V_3)$ for $\chi$, it is clear that the stability
condition is unchanged under the transition.
\end{NB}%
\end{proof}

Once this isomorphism is established, the remaining arguments of
\cite[\S7]{NT} only use dimension vectors, hence work for the
multiplicative case. To state the result, we recall invariants (see
\cite[\S 7.3]{NT}):
Let $N_{\xl_i}$ be the difference (left minus right) of entries of the
dimension vector at $\xl_i$. Then set $N(\xl_i,\xl_{i+1}) = N_{\xl_i}
- N_{\xl_{i+1}}$ plus the number of $\boldsymbol\medcirc$ between
$\xl_i\to \xl_{i+1}$.
It is invariant under Hanany-Witten transition.
Now by \cite[Prop.~7.20]{NT} we have

\begin{proposition}\label{prop:dominant}
  If $N(\xl_i,\xl_{i+1})\ge 0$ for any $i$, the multiplicative bow
  variety is isomorphic to another multiplicative bow variety with a
  cobalanced dimension vector by successive applications of
  Hanany-Witten transitions of Proposition~\ref{prop:HW}.
\end{proposition}

Recall the cobalanced condition is $N_{\xl_i} = 0$ for all $i$.

Recall (Section~\ref{Coul}) a $K$-theoretic Coulomb branch for a
framed quiver gauge theory of an affine type with dimension vectors
$\underline{\bv}$, $\underline{\bw}$ is isomorphic to a multiplicative
bow variety with the balanced condition. In this case, the condition
$N(\xl_i,\xl_{i+1})\ge 0$ for all $i$ is equivalent to the dominance
condition $\underline{\bw} - C\underline{\bv} \in \ZZ_{\ge 0}^n$, as
the number of $\boldsymbol\medcirc$ between $\xl_i\to \xl_{i+1}$ is
$\bw_i$.

\subsection{Multiplicative bow varieties with cobalanced dimension
  vector}\label{subsec:cobalanced}

Now we study a multiplicative bow variety with cobalanced
condition. Consider the following part of a bow diagram:
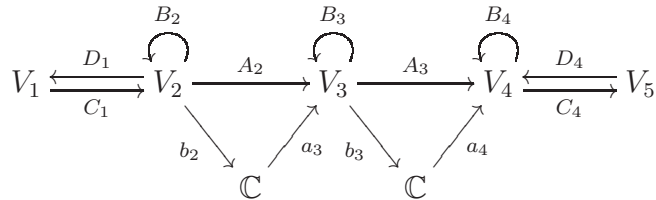
\begin{equation*}
  \xymatrix@C=1.2em{ V_1 
    \ar@<-.5ex>[rr]_{C_1} 
  && V_2 \ar@(ur,ul)_{B_2}
  \ar@<-.5ex>[ll]_{D_1} \ar[rr]^{A_2} \ar[dr]_{b_2}
  && V_3 \ar@(ur,ul)_{B_3} 
  \ar[rr]^{A_3} \ar[dr]_{b_3}
  &&
  V_4 \ar@(ur,ul)_{B_4} \ar@<-.5ex>[rr]_{C_4}
  &&
  V_5 \ar@<-.5ex>[ll]_{D_4}
  \\
  &&& \CC \ar[ur]_{a_3} && \CC \ar[ur]_{a_4} & }
\end{equation*}
We assume $\dim V_2 = \dim V_3 = \dim V_4$ by the cobalanced
condition.  By \cite[Lemma~2.18]{Ta}, $A_2$, $A_3$ are
isomorphisms. So we normalize $A_2 = A_3 = \id$ by the action of
$\GL(V_3)\times\GL(V_4)$. Then the defining equation becomes
\begin{equation*}
  \begin{split}
  & 1 + C_1 D_1 = Z_2 B_2^{-1} = Z_2 (B_3 + a_3 b_2)^{-1}
  = Z_2 (1 + B_3^{-1} a_3 b_2)^{-1} B_3^{-1}\\
  =\; & t^{-1} Z_2 (1 + B_3^{-1} a_3 b_2)^{-1} (1 + B_4^{-1} a_4 b_3)^{-1}
  B_4^{-1} \\
  =\; & t^{-2} Z_4^{-1} Z_2 
  (1 + B_3^{-1} a_3 b_2)^{-1} (1 + B_4^{-1} a_4 b_3)^{-1}
  (1 + D_4 C_4).
  \end{split}
\end{equation*}
This can be regarded as the defining equation of a multiplicative
quiver variety with
\begin{equation}\label{eq:App4}
  \xymatrix@C=1.2em{ V_1 
    \ar@<-.5ex>[rr]_{C_1} 
  && V_2 
  \ar@<-.5ex>[ll]_{D_1} \ar@<-.5ex>[rr]_{C_4} 
  \ar@<-.5ex>[dl]_(.7){b_2} \ar@<-.5ex>[dr]_(.6){b_3}
  &&
  V_5 \ar@<-.5ex>[ll]_{D_4}
  \\
  & \CC \ar@<-.5ex>[ur]|{\hole\hole B_3^{-1}a_3} &&
  \CC \ar@<-.5ex>[ul]|{\hole\hole B_4^{-1}a_4} &
  },
\end{equation}
where we do not take the quotient by $\CC^\times\times\CC^\times$ at
bottom vertices. We can further make
\begin{equation}\label{eq:App5}
  \xymatrix@C=1.2em{ V_1 
    \ar@<-.5ex>[rr]_{C_1} 
  && V_2 
  \ar@<-.5ex>[ll]_{D_1} \ar@<-.5ex>[rr]_{C_4} 
  \ar@<-.5ex>[d]_{\left[\begin{smallmatrix} b_2 \\ b_3\end{smallmatrix}\right]} 
  &&
  V_5 \ar@<-.5ex>[ll]_{D_4}
  \\
  && \CC^2 \ar@<-.5ex>[u]_{\left[
      \begin{smallmatrix}
        \tilde a_3 & \tilde a_4
      \end{smallmatrix}\right]}
  &&
},
\end{equation}
where $\tilde a_3 = (1 + B_4^{-1} a_4 b_3) B_3^{-1} a_3$,
$\tilde a_4 = B_4^{-1} a_4$. 
\begin{NB}
as $(1 + B_4^{-1} a_4 b_3)(1 + B_3^{-1} a_3 b_2) = 1 + \left[
  \begin{smallmatrix}
    (1 + B_4^{-1} a_4 b_3) B_3^{-1} a_3 & B_4^{-1} a_4
  \end{smallmatrix}
\right]\left[
\begin{smallmatrix}
  b_2 \\ b_3
\end{smallmatrix}\right]$.
\end{NB}%
For additive quiver varieties \eqref{eq:App4} and \eqref{eq:App5} give
isomorphic varieties. But in the multiplicative case, \eqref{eq:App4}
gives an open subvariety in \eqref{eq:App5} as the invertibility of
both $(1 + B_3^{-1} a_3 b_2)$, $(1 + B_4^{-1} a_4 b_3)$ is stronger than the
invertibility of $(1 + \tilde a_3 b_2 + \tilde a_4 b_3)$. In fact, we
have an additional requirement that $(1 + \tilde a_4 b_3)$ is invertible.
Note that the difference between \eqref{eq:App4} and \eqref{eq:App5}
disappears when we only have one $\CC$.

\subsection{Multiplicative quiver varieties}\label{subsec:mq}

Before generalizing the definition of multiplicative quiver varieties 
given in Section~\ref{quiver},
let us recall the notion of quasi-Hamiltonian spaces 
and fusion/reduction procedure introduced in \cite{AMM}.

Let $G$ be a complex reductive group with Lie algebra $\mathfrak{g}$ 
and fix a non-degenerate $\Ad_G$-invariant symmetric bilinear form 
$(\phantom{a},\phantom{a})$ on $\mathfrak{g}$.
Let $\theta$ (resp.\ $\ol{\theta}$) be the left (resp.\ right) invariant 
Maurer--Cartan form on $G$.

A \emph{quasi-Hamiltonian $G$-space} is 
a smooth $G$-variety $M$ equipped with 
a $G$-invariant two-form $\omega$ on $M$ and 
a $G$-equivariant morphism $\mu \colon M\to G$ (where $G$ acts on $G$ by conjugation) 
satisfying the following three axioms:
\begin{enumerate}
\item[(QH1)]
$12d\omega =  - \mu^*(\theta \wedge [\theta \wedge \theta])$.
\begin{NB}
    This is a mixture of the exterior product $\wedge$ and $(\ ,\
    )$.
\end{NB}%
\item[(QH2)]
$2 \iota(\xi^*) \omega = \mu^* (\theta+\ol{\theta} , \xi)$
for all $\xi \in \mathfrak{g}$. 
\begin{NB}
    The same applies.
\end{NB}%
Here $\xi^*$ is the fundamental vector field:
\[
\xi^*_x = \left. \frac{d}{dt} \exp (t\xi) \cdot x \,\right|_{t=0}
\quad (x \in M).
\]
\item[(QH3)]
$\Ker \omega_x = \{\,\xi^*_x \mid \xi \in \Ker (\Ad_{\mu(x)} + 1)\,\}$
for all $x \in M$.
\end{enumerate} 
The map $\mu$ is called the \emph{group-valued moment map}.

\begin{example}\label{ex:conjugacy}
Any conjugacy class $\mathcal{C}$ of $G$ has a structure of 
quasi-Hamiltonian $G$-space whose group-valued moment map is just    
the inclusion $\mathcal{C} \hookrightarrow G$.
This is a multiplicative analogue of a coadjoint orbit of $\mathfrak g^*$.

Also, the \emph{double} $G \times G$ is a quasi-Hamiltonian $G \times G$-space,
which is a multiplicative analogue of $T^*G$.
\end{example}

\begin{example}
In \cite{VdB} Van den Bergh introduced a multiplicative analogue of 
the Hamiltonian $\GL(V) \times \GL(W)$-space $T^*\Hom(V,W)$,   
where $V, W$ are finite-dimensional $\CC$-vector spaces.
It is defined to be the quasi-Hamiltonian $\GL(V) \times \GL(W)$-space  
\[
\mathcal{B}(V,W) = \{\,(X,Y) \in \Hom(V,W) \times \Hom(W,V) \mid \det(1+XY) \neq 0\,\},
\]
where the two-form and the group-valued moment map are given by
\begin{gather*}
\omega = \frac12 \Tr \left( (1+XY)^{-1}dX \wedge dY \right)
-\frac12 \Tr \left( (1+YX)^{-1}dY \wedge dX \right), \\
\mu(X,Y) =((1+YX)^{-1},1+XY).
\end{gather*}
This is called the \emph{Van den Bergh space}.
\end{example}

Next we introduce the internal fusion of quasi-Hamiltonian spaces.
Let $H$ be another complex reductive group equipped with 
a non-degenerate $\Ad_H$-invariant symmetric bilinear form on 
its Lie algebra $\mathfrak{h}$.

\begin{theorem}[{\cite{AMM}}]\label{thm:fusion}
Let $(M,\omega,(\mu_1,\mu_2,\nu))$ be a 
quasi-Hamiltonian $G\times G\times H$-space. 
Let $G\times H$ act on $M$ through the diagonal embedding $(g,h)\mapsto (g,g,h)$. 
Then $M$ equipped with the two-form 
\[ 
\omega_\fus :=\omega + \frac12 (\mu_1^* \theta \wedge \mu_2^* \ol{\theta})
\]
and the map
\[
(\mu_\fus,\nu) := (\mu_1\cdot \mu_2,\nu) \colon M \to G \times H
\]
is a quasi-Hamiltonian $G\times H$-space. 
\end{theorem}

The quasi-Hamiltonian $G \times H$-space 
$(M,\omega_\fus,(\mu_\fus,\nu))$ is called the \emph{internal fusion} 
and denoted by $M_\fus$.

The internal fusion procedure is associative in the following sense: 
If $M$ is a quasi-Hamiltonian 
$G\times G\times G\times H$-space,
then two quasi-Hamiltonian $G\times H$-spaces 
$M_{(12)3}$ obtained by first fusioning 
the first two $G$-factors and $M_{1(23)}$ 
obtained by first fusioning the last two $G$-factors are identical.
Therefore, if $I$ is a non-empty totally ordered finite set
and $M$ is a quasi-Hamiltonian $G^I \times H$-space  
with group-valued moment map $((\mu_i)_{i \in I}, \nu)$,
then we can define its internal fusion $M_\fus$   
as a quasi-Hamiltonian $G \times H$-space in a canonical way so that 
its group-valued moment map is $(\mu,\nu)$, where $\mu := \prod_{i \in I}^< \mu_i$.

A multiplicative analogue of the Marsden--Weinstein theorem is the following theorem:
\begin{theorem}[{\cite{AMM}}]\label{thm:qhamred}
Let $(M,\omega,(\mu,\nu))$ be 
a quasi-Hamiltonian $G \times H$-space
and $\mathcal{C}$ be a conjugacy class of $G$.
If the $G$-action on $\mu^{-1}(\mathcal{C})$ is free, 
then $\mu^{-1}(\mathcal{C})$ is a smooth subvariety of $M$,
and if furthermore the action has a geometric quotient 
$\mu^{-1}(\mathcal{C})/G$, then
$\omega$ and $\nu$ induce a quasi-Hamiltonian $H$-structure on 
$\mu^{-1}(\mathcal{C})/G$. 
\end{theorem}

The above quasi-Hamiltonian $H$-space $\mu^{-1}(\mathcal{C})/G$ is 
called the \emph{quasi-Hamiltonian reduction} of $M$ by $G$
along $\mathcal{C}$ and denoted by $M /\!\!/_{\!\mathcal{C}}\, G$.

Note that in the above situation if $H$ is abelian then $M /\!\!/_{\!\mathcal{C}}\, G$ 
is symplectic. The following fact is also useful.

\begin{theorem}\label{thm:qhampoi}
Let $(M,\omega,(\mu,\nu))$ be 
a quasi-Hamiltonian $G \times H$-space.
Assume that $H$ is abelian and $M$ has a good quotient $\pi \colon M \to M/G$.
Then for any open $U \subset M/G$ and $f \in \Gamma(U,\CO_{M/G})$, 
there exists a unique vector field $v_f \in \Gamma(\pi^{-1}(U),\Theta_M)$ such that
\[
\iota(v_f)\omega = d\pi^*f, \quad \iota(v_f)\mu^*\theta =0.
\]
Each $v_f$ is $G$-invariant and preserves $\omega$ and $\mu$.
Also, $\pi^*\{ f,g \}=v_f(\pi^*g)$ ($f,g \in \CO_{M/G}$) defines 
an $H$-invariant Poisson structure on $M/G$.  
\end{theorem}

See \cite[Proposition~4.6]{AMM}, where $H$ is assumed to be trivial 
but the arguments in the proof work in the general case.

Since each $v_f$ preserves $\mu$, 
for any conjugacy class $\mathcal{C} \subset G$
the closed subvariety 
$M /\!\!/_{\!\ol{\mathcal{C}}}\, G :=\mu^{-1}(\ol{\mathcal{C}})/G \subset M/G$ is Poisson. 
We call it the quasi-Hamiltonian reduction of $M$ by $G$ along $\ol{\mathcal{C}}$.
On the other hand, 
if there exists an open $U \subset \mu^{-1}(\ol{\mathcal{C}})/G$ such that 
$\mu(\pi^{-1}(U)) \subset \mathcal{C}$ and $G$ acts freely on $\pi^{-1}(U)$,
then Theorem~\ref{thm:qhamred} shows that $U$ is a quasi-Hamiltonian $H$-space,
in particular symplectic. 
By the definition, 
the Poisson structure on $U$ given by Theorem~\ref{thm:qhampoi} 
coincides with one induced from the symplectic structure.

Now we introduce the multiplicative quiver varieties.
Let $\qQ=(\qQ_0,\qQ_1,\qs,\qt)$ be a finite quiver. Let 
$V=\bigoplus_{\qv \in \qQ_0} V_\qv$ and $W=\bigoplus_{\qv \in \qQ_0} W_\qv$ 
be two finite-dimensional $\qQ_0$-graded $\CC$-vector spaces
with dimension vectors 
$\vbv=(\bv_\qv)_{\qv \in \qQ_0}$ and $\vbw=(\bw_\qv)_{\qv \in \qQ_0}$,
respectively.  
For each $\qv \in \qQ_0$ fix a decomposition of $W_\qv$ 
into one-dimensional pieces 
\[
W_\qv = \bigoplus_{j=1}^{\bw_\qv} W_{\qv,j}, \quad \dim W_{\qv,j} =1,
\] 
and let $T(W) \subset \GL(W):=\prod_{\qv \in \qQ_0} \GL(W_\qv)$
be the associated maximal torus.
Define 
\[
\widetilde{\mathcal{B}}_\qQ(V,W) = \prod_{\qa \in \qQ_1} \mathcal{B}(V_{\qs(\qa)},V_{\qt(\qa)}) 
\times \prod_{\qv \in \qQ_0} \prod_{j=1}^{\bw_\qv} \mathcal{B}(W_{\qv,j},V_\qv),
\]
which is a quasi-Hamiltonian $\mathbb{G} \times T(W)$-space, where
\[
\mathbb{G} := \prod_{\qa \in \qQ_1} 
\left( \GL(V_{\qs(\qa)}) \times \GL(V_{\qt(\qa)}) \right)
\times \prod_{\qv \in \qQ_0} \GL(V_\qv)^{\bw_\qv}.
\]
Using the double $\ol{\qQ}$ of $\qQ$ 
with involution $* \colon \ol{\qQ}_1 \to \ol{\qQ}_1$ 
(so $\ol{\qQ}_1 = \qQ_1 \sqcup \qQ^*_1$ 
and $\qs \circ * = \qt$, $\qt \circ * = \qs$), 
we denote an element of $\widetilde{\mathcal{B}}_\qQ(V,W)$ by 
$(C,a,b)$, where $C=(C_\qa)_{\qa \in \ol{\qQ}_1}$,
$a=(a_{\qv,j})$, $b=(b_{\qv,j})$ and
\begin{equation*}
\begin{split}
(C_\qa, C_{\qa^*}) &\in \mathcal{B}(V_{\qs(\qa)},V_{\qt(\qa)}) \quad (\qa \in \qQ_1), \\
(a_{\qv,j}, b_{\qv,j}) &\in \mathcal{B}(W_{\qv,j},V_\qv) \quad 
(\qv \in \qQ_0,\, j=1, \dots ,\bw_\qv).
\end{split}
\end{equation*}
(In other places, $C_\qa$ is denoted by $C_i$ or $D_i$ according to whether
$\qa$ is in $\qQ_1$ or not.)
Fix a total ordering $<$ on $\ol{\qQ}_1$ and define 
a quasi-Hamiltonian $\GL(V) \times T(W)$-space by
\[
\mathcal{B}_\qQ(V,W) = \widetilde{\mathcal{B}}_\qQ(V,W)_{\fus,<},
\]
where the internal fusion is taken 
with respect to the diagonal embedding $\GL(V) \hookrightarrow \mathbb{G}$ 
so that the resulting group-valued moment map 
$(\mu,\nu)=((\mu_\qv),(\nu_{\qv,j}))$ is
\begin{equation*}
\begin{split}
\mu_\qv(B,a,b) &= \prod^{<}_{\qa \in \ol{\qQ}_1; \qt(\qa)=\qv} 
(1 + C_\qa C_{\qa^*})^{\epsilon(\qa)} 
\prod_{j=1}^{\bw_\qv} (1 + a_{\qv,j} b_{\qv,j}), \\
\nu_{\qv,j}(C,a,b) &= (1 + b_{\qv,j} a_{\qv,j})^{-1},
\end{split}
\end{equation*}
where $\epsilon(\qa)=1$ if $\qa \in \qQ_1$ and 
$\epsilon(\qa)=-1$ if $\qa \in \qQ^*_1$.

\begin{definition}\label{dfn:mqv}
For $\gamma=(\gamma_\qv)_{\qv \in \qQ_0} \in (\CC^*)^{\qQ_0}$, 
put $\gamma_V = (\gamma_\qv\,1_{V_\qv})_{\qv \in \qQ_0} \in \GL(V)$ and define
\[
\MM_\gamma(\vbv,\vbw) = \mu^{-1}(\gamma_V)/\GL(V),
\] 
where the quotient is taken as the affine GIT quotient.
We call $\MM_\gamma(\vbv,\vbw)$ the \emph{multiplicative quiver variety}.
\end{definition}

By Theorem~\ref{thm:qhampoi}, 
$\MM_\gamma(\vbv,\vbw)= \mathcal{B}_\qQ(V,W) /\!\!/_{\!\gamma_V} \GL(V)$ 
is a Poisson variety.
Also, as in the case of additive quiver varieties, the $\GL(V)$-action on the open subset 
$\mu^{-1}(\gamma_V)^s \subset \mu^{-1}(\gamma_V)$  
consisting of stable points has stabilizer $\CC^* 1_V$ everywhere 
(where the definition of stability is exactly the same as in the additive case), 
and Theorem~\ref{thm:qhamred} implies that the open subset 
$\MMreg_\gamma(\vbv,\vbw) := \mu^{-1}(\gamma_V)^s/\GL(V)$ of $\MM_\gamma(\vbv,\vbw)$ is 
a quasi-Hamiltonian $T(W)$-space.
We also call $\MMreg_\gamma(\vbv,\vbw)$ the multiplicative quiver variety.

\begin{example}
The multiplicative quiver variety 
${\mathcal M}_N^l(Z,t)$ introduced in Section~\ref{quiver} coincides with 
$\MM_\gamma(\vbv,\vbw)$ with $\qQ = \hat{A}_{l-1}$ (the orientation is 
$1 \to l \to \cdots \to 2 \to 1$) and 
\[
\vbv =(N,N, \dots ,N), \quad \vbw = (1,0, \dots ,0), \quad 
\gamma = \left( \frac{Z_l}{Z_1}t,\frac{Z_1}{Z_2}, \dots ,\frac{Z_{l-1}}{Z_l} \right).
\]
\end{example}

\begin{remark}
Definition~\ref{dfn:mqv} is close to the definition of 
\emph{framed multiplicative quiver variety} introduced in \cite{Y}.
Let us identify all $W_{\qv,j}$ with $\CC$. 
Then the framed multiplicative quiver variety is obtained by  
replacing $\mathcal{B}_\qQ(V,W)$ with its internal fusion  
for the diagonal subgroup $\CC^* \subset T(W)$ 
in the definition of $\MM_\gamma(\vbv,\vbw)$.
It is the same as $\MM_\gamma(\vbv,\vbw)$ as a variety,
but the Poisson brackets are different in general.
\end{remark}

\begin{remark}\label{rem:another}
There is another multiplicative analogue of quiver variety. 
It is obtained by replacing $\widetilde{\mathcal{B}}_\qQ(V,W)$ with the quasi-Hamiltonian $\mathbb{G} \times \GL(W)$-space
\[
\widetilde{\mathcal{B}}'_\qQ(V,W) := \prod_{\qa \in \qQ_1} \mathcal{B}(V_{\qs(\qa)},V_{\qt(\qa)}) 
\times \prod_{\qv \in \qQ_0} \mathcal{B}(W_\qv,V_\qv) 
\]
in the definition of $\MM_\gamma(\vbv,\vbw)$.
But note that in general the two-form on $\widetilde{\mathcal{B}}'_\qQ(V,W)_{\fus,<}$ does not induce a Poisson structure 
on the resulting variety. 
It is why we decompose each $W_i$ into one-dimensional pieces.
Of course if $\dim W_\qv \leq 1$ for all $\qv \in \qQ_0$ then the two definitions coincide.
\end{remark}

\begin{example}\label{ex:cell}
  Consider $\qQ = A_{\ell-1}$ with $\vbw = (\ell,0,\dots,0)$. The
  usual quiver variety $\mathfrak M_\zeta(\vbv,\vbw)$, if it is not
  $\emptyset$, is a semisimple coadjoint orbit of
  $\mathfrak{gl}(\ell)$ for generic $\zeta$. The multiplicative quiver
  variety in the definition in Remark~\ref{rem:another} is a conjugacy
  class of $\GL(\ell)$. (See e.g., \cite[\S8]{CBS}.) On the other
  hand, Definition~\ref{dfn:mqv} gives its open subset, the
  intersection with a big Bruhat cell of $\GL(\ell)$. To see this,
  take a decomposition $W = \bigoplus_{j=1}^\ell W_j$ into 
  one-dimensional spaces, and write $a = (a_j)$, $b = (b_j)$. Then
  $\tilde a = (\tilde a_1,\dots,\tilde a_\ell)$,
  $\tilde b = {}^t(\tilde b_1,\dots,\tilde b_\ell)$ with
  $\tilde a_j = (1 + a_1 b_1)\dots (1 + a_{j-1} b_{j-1})a_j$,
  $\tilde b_j = b_j$ satisfy the defining equation in
  Remark~\ref{rem:another} as
  \begin{equation*}
    1 + \sum_{j=1}^\ell \tilde a_j \tilde b_j = (1 + a_1 b_1)\dots
    (1 + a_\ell b_\ell).
  \end{equation*}
  The isomorphism to a conjugacy class is given by the group-valued
  moment map $(1 + \tilde b\tilde a)^{-1}$. If an element
  $1 + \tilde b\tilde a$ is coming from $(a_j)$, $(b_j)$, we have
  additional constraint $\det(1 + a_j b_j)\neq 0$ for
  $j=1,\dots,\ell$, which is equivalent to require that every leading
  principal minor of $1 + \tilde b\tilde a$ is nonzero.
\end{example}

Now we understand the correct definition of multiplicative quiver
varieties, and the argument in Section~\ref{subsec:cobalanced} gives
the following.

\begin{theorem}\label{thm:cobalanced}
  Consider a multiplicative bow variety with a cobalanced dimension
  vector. It is isomorphic to a multiplicative quiver variety
  $\MM_{\gamma}(\underline{\bv}, \underline{\bw})$ such
  that $\mathbf v_j$ \textup(resp.\ $\mathbf w_j$\textup) is the
  dimension of vector spaces \textup(resp.\ the number of
  $\boldsymbol\times$\textup) between $h_j$ and $h_{j+1}$, and
  $\gamma_j = t^{\bw_j} Z_j/Z_{j+1}$.
\end{theorem}

Combining this with Proposition~\ref{prop:dominant}, we obtain an
isomorphism between a $K$-theoretic Coulomb branch and a
multiplicative quiver variety when the dominance condition is
satisfied. We conjecture that this is an isomorphism of Poisson
varieties. More generally we conjecture

\begin{conjecture}
  \textup{(1)} The space consisting of $A_{i-1}$, $B_{i-1}$, $B'_i$,
  $a_i$, $b_{i-1}$ with \textup{(i)}, \textup{(a)}, \textup{(S1)},
  \textup{(S2)} in Section~\ref{bow} is a quasi-Hamiltonian
  $\GL(V_{i-1})\times \GL(V_i)\times\CC^*$-space with the
  group-valued moment map $B_{i-1}^{-1}$, $B'_i$,
  $\det B_{i-1} \det B_i^{\prime-1}$.
    Therefore a multiplicative bow variety is a quasi-Hamiltonian
    reduction.

    \textup{(2)} The Hanany-Witten transition is an isomorphism of
    quasi-Hamiltonian $\GL(V_1)\times\GL(V_3)\times\CC^*$-spaces.

    \textup{(3)} The isomorphism of Theorem~\ref{thm:cobalanced}
    is an isomorphism of quasi-Hamiltonian $T(W)$-spaces.
\end{conjecture}

\begin{remark}
  Both homological and $K$-theoretic Coulomb branches have a torus
  action induced from $\Hom(\pi_1(G),\mathbb Z)$. See
  \cite[4(iii)(c)]{N} and the original physics literature therein. It
  is expected that the torus action extends to a nonabelian group
  action for homological Coulomb branches \cite[4(iii)(d)]{N}, and
  proved for $ADE$ quiver gauge theories in
  \cite[Remark~3.12]{BFN3}. On the other hand, it is not \emph{true}
  for $K$-theoretic Coulomb branches as Example~\ref{ex:cell} shows.
\end{remark}




\begin{thebibliography}{999999}
\bibitem[AMM]{AMM} A. Alekseev, A. Malkin, E. Meinrenken, 
Lie group valued moment maps, 
J.\ Differential Geom.\ 48 (1998), no.~3, 445--495.

\bibitem[AS]{AS}
  M.~F.~Atiyah, G.~B.~Segal,
  Equivariant $K$-theory and completion,
  J.\ Differential Geom.\ 3 (1969), 1--18.

\bibitem[BEG]{BEG} Y. Berest, P. Etingof, V. Ginzburg, Cherednik algebras and
differential operators on quasi-invariants, Duke Math. J. 118 (2003), no. 2,
279--337.

\bibitem[BC]{BC} Y. Berest, O. Chalykh, Quasi-invariants of complex reflection groups, Compos. Math. 147 (2011), no.3, pp. 965--1002.

\bibitem[BE]{BE} R. Bezrukavnikov, P. Etingof. Parabolic induction and restriction functors for rational Cherednik algebras, with an appendix by P. Etingof, Selecta Math. (N.S.), 14 (2009), pp. 397--425.

\bibitem[BF]{BF} T. Baker, P. Forrester, A q-analogue of the type A Dunkl operator and integral kernel, IMRN, issue 14, p.667--686, 1997.

\bibitem[BFNa]{BFN2} A.~Braverman, M.~Finkelberg, H.~Nakajima,
Towards a mathematical definition of Coulomb branches of 3-dimensional $\CN=4$
gauge theories, II, Adv.\ Theor.\ Math.\ Phys.\ 22 (2018), 1071--1147.

\bibitem[BFNb]{BFN3} A.~Braverman, M.~Finkelberg, H.~Nakajima,
Coulomb branches of $3d {\mathcal N}=4$ quiver gauge theories and slices in the
affine Grassmannian (with appendices by Alexander Braverman, Michael Finkelberg,
Joel Kamnitzer, Ryosuke Kodera, Hiraku Nakajima, Ben Webster, and Alex Weekes),
Adv.\ Theor.\ Math.\ Phys.\ 23 (2019), 75--166.

\bibitem[CBS]{CBS} W. Crawley-Boevey, P. Shaw,
Multiplicative preprojective algebras, middle convolution and the Deligne-Simpson problem, Adv. Math. 201 (2006), 180-208. (math.RA/0404186).

\bibitem[C]{C} O. Chalykh, Macdonald polynomials and algebraic integrability,
Adv. Math. 166 (2002), no.~2, 193--259.

\bibitem[CF]{CF} O. Chalykh, M. Fairon, 
  Multiplicative quiver varieties and generalized Ruijsenaars-Schneider models,
  J.\ Geom.\ Phys.\ 121 (2017), 413--437.

\bibitem[CV]{CV} O. Chalykh, A. Veselov, 
Commutative rings of partial differential operators
and Lie algebras, Commun. Math. Phys., 126 (1990), 597--611.

\bibitem[CE]{CE} 
O. Chalykh, P. Etingof, 
Orthogonality relations and Cherednik identities for multivariable Baker-Akhiezer functions, with an appendix by Oleg Chalykh, arXiv:1111.0515,
 Advances in Mathematics, 238 (2013), 246--289.

\bibitem[Ch1]{Ch1} I. Cherednik, Double affine Hecke algebras, London Mathematical Society Lecture Note Series 319, 2005.

\bibitem[Ch2]{Ch2}
I. Cherednik,
A unification of Knizhnik-Zamolodchikov and Dunkl operators via affine Hecke algebras.
Invent. Math.106 (1991), No.2, 411--431.



\bibitem[DE]{DE} J. F. van Diejen and E. Emsiz, 
Spectrum and Eigenfunctions of the Lattice Hyperbolic Ruijsenaars-Schneider System with Exponential Morse Term,
Ann. Henri Poincar\'e 17 (2016), 1615--1629. 

\bibitem[DO]{DO} C.F. Dunkl, and E.M. Opdam, Dunkl operators for complex reflection groups, Proc. London Math. Soc. (3), 86 (2003), 70--108.

\bibitem[E1]{E1} P. Etingof, Cherednik and Hecke algebras of varieties with a finite group action,
  Mosc.\ Math.\ J.\ 17 (2017), 635--666.

\bibitem[EGGO]{EGGO} P. Etingof, W. L. Gan, V. Ginzburg, A. Oblomkov,
Harish-Chandra homomorphisms and symplectic reflection algebras for wreath-products, Publ. Math. de l'IHES, 105 (2007), 91--155.

\bibitem[EG1]{EG1} P. Etingof, V. Ginzburg, Symplectic reflection algebras, Calogero-Moser space, and deformed Harish-Chandra homomorphism,
Inventiones mathematicae, Volume 147 (2002), no. 2, pp. 243--348.

\bibitem[EG2]{EG2}
P. Etingof and V. Ginzburg, On $m$-quasi-invariants of a Coxeter group, 
Mosc.\ Math.\ J.\ 2 (2002), no.~3, 555--566.

\bibitem[EM]{EM} P. Etingof, X. Ma, Lecture notes on Cherednik algebras,  arXiv:1001.0432.

\bibitem[ER]{ER} P. Etingof, E. Rains, with an appendix by Misha Feigin, 
On Cohen-Macaulayness of algebras generated by generalized power sums, arXiv:1507.07485, 
Communications in Mathematical Physics, 347 (2016), no.~1, pp. 163--182.

\bibitem[FV]{FV} M. Feigin and A. Veselov, Quasi-invariants of Coxeter groups and
$m$-harmonic polynomials, Int. Math. Res. Not. (2002), no. 10, 521--545.

\bibitem[FeV]{FeV} G. Felder, A. Veselov, Action of Coxeter groups on $m$-harmonic polynomials and Knizhnik-Zamolodchikov equations,
Moscow Mathematical Journal, 3 (2003), no.~4, 1269--1291.

\bibitem[GKV]{GKV} V. Ginzburg, M. Kapranov, E. Vasserot,   Residue Construction of Hecke Algebras, Advances in Mathematics,
128 (1997), no.~1, 1--19; arXiv:alg-geom/9512017. 

\bibitem[G]{G} S.~Griffeth, Unitary representations of cyclotomic rational Cherednik
algebras, J.\ Algebra, 512 (2018), 310--356.

\bibitem[J]{J} D. Jordan, Quantized multiplicative quiver varieties.
Advances in Mathematics, 250 (2014), 420--466.

\bibitem[KN]{KN} R.~Kodera and H.~Nakajima, Quantized Coulomb branches of
Jordan quiver gauge theories and cyclotomic rational Cherednik algebras,
Proc.\ Sympos.\ Pure Math., 98 (2018), 49--78.

\bibitem[L]{l91} G.~Lusztig, Quivers, perverse sheaves, and quantized
enveloping algebras, Journal of the AMS, 3 (1991), no.~2, 365--421.

\bibitem[Mo]{Mo} S. Montgomery, Fixed Rings of Finite Automorphism Groups of Associative Rings, LN in
Math., 818 Springer-Verlag, Berlin/New-York, 1980.

\bibitem[N]{N} H.~Nakajima, Towards a mathematical definition of Coulomb
branches of 3-dimensional ${\mathcal N}=4$ gauge theories, I.
Advances in Theoretical and Mathematical Physics, 20 (2016), no.~3, 595--669.

\bibitem[NT]{NT} H.~Nakajima, Y.~Takayama, Cherkis bow varieties and Coulomb
  branches of quiver gauge theories of affine type $A$, Selecta Math.\ (N.S.) 23 (2017), no.~4,
  2553--2633.

\bibitem[Ob1]{Ob1} A. Oblomkov,
Double affine Hecke algebras and Calogero-Moser spaces,
math.RT/0303190, Represent. Theory 8 (2004), 243--266.

\bibitem[Ob2]{Ob2} A. Oblomkov, Deformed Harish-Chandra homomorphism for the cyclic quiver, Math. Res. Lett, 14 (2007), no. 3, 359--372.

\bibitem[OY]{OY} A.~Oblomkov, Z.~Yun, Geometric representations of graded and
rational Cherednik algebras. Adv. Math. 292 (2016), 601--706.

\bibitem[R1]{R1} E. Rains, The noncommutative geometry of elliptic difference equations, 
arXiv:1607.08876. 

\bibitem[R2]{R2} E. Rains, Elliptic double affine Hecke algebras, arXiv:1709.02989.  

\bibitem[Su]{Su} T. Suzuki, Rational and trigonometric degeneration of the
double affine Hecke algebra of type $A$,
Int. Mathematics Research Notices,
Volume 2005, Issue 37, 2249--2262.

\bibitem[Ta]{Ta}
  Y.~Takayama,
  Nahm's equations, quiver varieties and parabolic sheaves,
  Publ. Res. Inst. Math. Sci. 52 (2016), no.~1, 1--41.

\bibitem[Th]{To}
  R.~W.~Thomason,
  Equivariant algebraic vs.\ topological $K$-homology Atiyah-Segal-style,
  Duke Math.\ J.\ 56 (1988), no.~3, 589--636.

\bibitem[VdB]{VdB} M. van den Bergh, Double Poisson algebras, 
Trans. Amer. Math. Soc. 360 (2008), 5711--5769. 

\bibitem[VV]{VV} M.~Varagnolo, E.~Vasserot,
Double affine Hecke algebras and affine flag manifolds, I. Affine flag
manifolds and principal bundles, Trends Math., Birkh\"auser/Springer
Basel AG, Basel (2010), 233--289.

\bibitem[W]{W} B.~Webster, A KLR-type presentation of the cyclotomic
Cherednik algebra, arXiv:1609.05494.

\bibitem[Y]{Y} D.~Yamakawa, Geometry of multiplicative preprojective algebra,
Int. Math. Res. Pap. IMRP 2008, Art. ID rpn008, 77pp.


\end{thebibliography}
\end{document}